\documentclass[11pt,english]{smfart}
\usepackage[french]{babel}

\usepackage[utf8]{inputenc}
\usepackage[T1]{fontenc}

\usepackage[]{hyperref}
\hypersetup{
    colorlinks=true,       % false: boxed links; true: colored links
    linkcolor=red,          % color of internal links
    citecolor=blue,        % color of links to bibliography
    filecolor=magenta,      % color of file links
    urlcolor=cyan           % color of external links
}

\usepackage[leqno]{amsmath}
\usepackage{amssymb}
\usepackage{mathrsfs}
\usepackage{amsthm}
\usepackage{amsxtra}
\usepackage{bm}
\usepackage[titletoc]{appendix}

\usepackage{graphicx}

\theoremstyle{plain}
\newtheorem{theo}{Theorem}[section]
\newtheorem{prop}[theo]{Proposition}
\newtheorem{lemm}[theo]{Lemma}
\newtheorem{coro}[theo]{Corollary}

\newtheorem{defi}[theo]{Definition}
\theoremstyle{definition}
\newtheorem{rema}[theo]{Remark}
\newtheorem{nota}[theo]{Notation}

\numberwithin{equation}{section}

\title[A paradifferential reduction for the gravity-capillary waves]{A paradifferential reduction for the gravity-capillary waves system at low regularity and applications}

\author{Thibault de Poyferr\'{e}}
\address{UMR 8553 du CNRS, Laboratoire de Math\'ematiques et Applications de l'Ecole Normale Sup\'erieure, 75005 Paris, France}
\email{tdepoyfe@dma.ens.fr}
\author{Quang-Huy Nguyen}
\address{UMR 8628 du CNRS, Laboratoire de Math\'ematiques d'Orsay, Univ. Paris-Sud, CNRS, Universit\'e Paris-Saclay, 91405 Orsay, France.}
\email{quang-huy.nguyen@math.u-psud.fr }
\date{}

\DeclareMathOperator{\RE}{Re}

\DeclareMathOperator{\dive}{div}

\def \div {\dive}
\def\A{\mathcal{A}}
\def\B{\mathcal{B}}
\def\C{\mathcal{C}}

\def\wL{\widetilde L}
\def\l{\ell}
\def\Cs{C_*}
\def\d{\,\mathrm{d}}
\def\eps{\varepsilon}

\def\la{\left\vert}
\def\lA{\left\Vert}
\def\lb{\left[}
\def\lB{\left\{}
\def\lp{\left(}

\def\les{\lesssim}
\def\mez{\frac{1}{2}}
\def\tmez{\frac{3}{2}}

\def\ph{\varphi}
\def\R{\bm{\mathrm{R}}}
\def\ra{\right\vert}
\def\rA{\right\Vert}
\def\rb{\right]}
\def\rB{\right\}}
\def\rp{\right)}

\DeclareMathOperator{\cnx}{div}

\DeclareMathOperator{\dist}{dist}

\DeclareSymbolFont{pletters}{OT1}{cmr}{m}{sl}
\DeclareMathSymbol{s}{\mathalpha}{pletters}{`s}

\def\defn{\mathrel{:=}}

\def\Bl{B_{\infty, 1}}

\def\eps{\varepsilon}

\def\cM{\mathcal{M}}

\def\cH{\mathcal{H}}

\def\la{\left\vert}
\def\lA{\left\Vert}
\def\le{\leq}
\def\les{\lesssim}
\def\leo{}
\def\L#1{\langle #1 \rangle}

\def\eps{\varepsilon}
\def\mez{\frac{1}{2}}
\def\partialx{\nabla}
\def\partialyx{\nabla_{x,y}}
\def\ra{\right\vert}
\def\rA{\right\Vert}

\def\tdm{\frac{3}{2}}

\def\xC{\mathbf{C}}
\def\xN{\mathbf{N}}
\def\xR{\mathbf{R}}

\def\xZ{\mathbf{Z}}
\def\cF{ \mathcal{F}}

\def\Re{\text{Re}}

\newcommand{\bq}{\begin{equation}}
\newcommand{\eq}{\end{equation}}
\newcommand{\bqa}{\begin{eqnarray*}}
\newcommand{\eqa}{\end{eqnarray*}}

\newcommand{\hk}{\hspace*{.15in}}

\begin{document}
\def\smfbyname{}

\begin{abstract}
We consider in this article the system of gravity-capillary waves in all dimensions and under the Zakharov/Craig-Sulem formulation. Using a paradifferential approach introduced by Alazard-Burq-Zuily, we symmetrize this system into a quasilinear dispersive equation whose principal part is of order $3/2$. The main novelty, compared to earlier studies, is that this reduction is performed at the  Sobolev regularity of quasilinear pdes: $H^s(\xR^d)$ with $s>3/2+d/2$, $d$ being the dimension of the free surface. \\
\hk From this reduction, we deduce a blow-up criterion involving solely the Lipschitz norm of the velocity trace and the $C^{\frac 52+}$-norm of the free surface. Moreover, we obtain an a priori estimate in the $H^s$-norm and the contraction of the solution map in  the $H^{s-\tdm}$-norm using the control of a Strichartz norm. These results have been applied in establishing a local well-posedness theory for non-Lipschitz initial velocity in our companion paper \cite{NgPo}.
\end{abstract}
\maketitle
\tableofcontents
\section{Introduction}
\hk We consider the system of gravity-capillary waves describing the motion of a fluid interface under the effect of both gravity and surface tension. From the well-posedness result in Sobolev spaces  of Yosihara \cite{Yosihara} (see also Wu \cite{WuInvent, WuJAMS} for pure gravity waves) it is known that the system is quasilinear in nature. In the more recent work \cite{ABZ1}, Alazard-Burq-Zuily showed explicitly this quasilinearity by using a paradifferential approach (see Appendix \ref{Appendix}) to symmetrize the system into the following paradifferential equation
\bq\label{intro:reduce}
\big(\partial_t +T_{V(t, x)}\cdot\nabla +iT_{\gamma(t, x, \xi)} \big) u(t,x)=f(t,x)
\eq
where $V$ is the horizontal component of the trace of the velocity field on the free surface, $\gamma$ is an elliptic symbol of order $3/2$, depending only on the free surface. In other words, the transport part comes from the fluid and the dispersive part comes from the free boundary. The reduction \eqref{intro:reduce} was implemented for 
\bq\label{intro:reg:ABZ}
u\in L^\infty_tH^s_x\quad s>2+\frac{d}{2},
\eq
$d$ being the dimension of the free surface. It has many consequences, among them are the local well-posedness and smoothing effect in \cite{ABZ1}, Strichartz estimates in \cite{ABZ2}. As remarked in \cite{ABZ1}, $s>2+d/2$ is the minimal Sobolev index (in term of Sobolev's embedding) to ensure that the velocity filed is Lipschitz up to the boundary, without taking into account the dispersive property.  From the works of Alazard-Burq-Zuily \cite{ABZ3, ABZ4},  Hunter-Ifrim-Tataru \cite{HuIfTa} for pure gravity waves, it seems natural to require that {\it the velocity is Lipschitz} so that the particles flow is well-defined, in view of the Cauchy-Lipschitz theorem. On the other hand, from {\it the standard theory of quasilinear pdes}, it is natural to ask if the reduction \eqref{intro:reduce} holds at the Sobolev threshold $s>3/2+d/2$ and then, if a local-wellposedness theory holds at the same level of regularity?  The two observations above motivate us to study the gravity-capillary system at the following regularity level: 
\bq\label{intro:reg}
u\in \mathcal{X}:=L^\infty_tH^s_x\cap L^p_tW^{2, \infty}_x \quad\text{with}~
s>\frac{3}{2}+\frac{d}{2},
\eq
which exhibits a gap of $1/2$ derivative that may be filled up by Strichartz estimates. \eqref{intro:reg} means that on the one hand, the Sobolev regularity is that of quasilinear equations of order $3/2$; on the other hand, the  $L^p_tW^{2, \infty}_x$-norm ensures that the velocity is still Lipschitz for {\it a.e.} $t\in [0, T]$ (which is the threshold \eqref{intro:reg:ABZ} after applying Sobolev's embedding).\\
\hk By sharpening the analysis in \cite{ABZ1}, we shall perform the reduction \eqref{intro:reduce} assuming merely the regularity $\mathcal{X}$ of the solution. In order to do so, the main difficulty, compared to \cite{ABZ1}, is that further studies of the Dirichlet-Neumann operator in Besov spaces are demanded. Moreover, we have to keep all the estimates in the analysis to be {\it tame}, {\it i.e.}, linear with respect to the highest norm which is the H\"older norm in this case. \\
\hk  From this reduction, we deduce several consequences. The first one will be an {\it a priori estimate for the Sobolev norm}  $L^\infty_tH^s_x$ using in addition the Strichartz norm $L^p_tW^{2,\infty}_x$ (see Theorem \ref{intro:theo:apriori} below for an exact statement). This is an expected result, which follows the pattern established for other quasilinear equations. However, for water waves, it requires much more care due to the fact that the system is nonlocal and highly nonlinear. This problem has been addressed by Alazard-Burq-Zuily \cite{ABZ4} for pure gravity water waves. In the case with surface tension, though the regularity level is higher, it requires a more precise analysis of the Dirichlet-Neumann operator in that lower order terms in the expansion of this operator need to be taken into consideration (se Proposition \ref{prop:paralin:DN} below).\\
Another consequence will be a {\it blow-up criterion} (see Theorem \ref{intro:theo:blowup}), which implies that the solution can be continued as long as the $\mathcal{X}$-norm of $u$  remained bounded (at least in the infinite depth case) with $p=1$, {\it i.e.}, merely integrable in time. It also implies that, starting from a smooth datum, the solution remains smooth provided its $C^{2+}$-norm  is bounded in time. \\
\hk For more precise discussions, let us recall the Zakharov/Craig-Sulem formulation of water waves.
\subsection{The Zakharov/Craig-Sulem formulation}
 We consider an incompressible, irrotational, inviscid fluid with unit density moving in a time-dependent domain  
$$
\Omega = \{(t,x,y) \in[0,T] \times \xR^d \times \xR:(x, y)\in \Omega_t\} 
$$
where each $\Omega_t$ is a domain located underneath a free surface 
$$
\Sigma_t = \{(x,y)  \in \xR^d \times \xR: y=\eta(t, x)\} 
$$
and above a fixed bottom $\Gamma=\partial\Omega_t\setminus \Sigma_t$. We make the following separation assumption $(H_t)$ on the  domain at time~$t$:\\
{\it
$\Omega_t$ is the intersection of the half space 
\[
\Omega_{1,t}= \{(x,y)  \in \xR^d \times \xR: y<\eta(t, x)\} 
\]
and an open connected set $\mathcal{O}$ containing a fixed strip around $\Sigma_t$,  {\it i.e.}, there exists $h>0$ such that 
\bq\label{sepbot}
 \{(x,y)  \in \xR^d \times \xR: \eta(x)-h\le y\le\eta(t, x)\} \subset \mathcal{O}.
\eq
}
\hk The velocity field $v$ admits a harmonic potential $\phi:\Omega \to \xR$, {\it i.e.}, $v=\nabla \phi$ and $\Delta \phi=0$. Using the  idea of Zakharov, we introduce the trace of $\phi$ on the free surface
$$\psi(t,x)= \phi(t,x,\eta(t,x)).$$ 
 Then $\phi(t, x, y)$ is the unique variational solution to the problem
\bq
\Delta\phi =0\text{~in}~\Omega_t,\quad \phi(t, x, \eta(t, x))=\psi(t, x),\quad \partial_n\phi(t)\rvert_\Gamma=0.
\eq
The Dirichlet-Neumann operator is then defined by
\bq\label{defi:DN:intro}
\begin{aligned}
G(\eta) \psi &= \sqrt{1 + \vert \nabla_x \eta \vert ^2}
\Big( \frac{\partial \phi}{\partial n} \Big \arrowvert_{\Sigma}\Big)\\
&= (\partial_y \phi)(t,x,\eta(t,x)) - \nabla_x \eta(t,x) \cdot(\nabla_x \phi)(t,x,\eta(t,x)).
\end{aligned}
\eq
The gravity-capillary water waves problem with surface tension consists in solving the following so-called Zakharov-Craig-Sulem system of $(\eta,\psi)$
\begin{equation}\label{ww}
\left\{
\begin{aligned}
&\partial_t \eta = G(\eta) \psi,\\
&\partial_t \psi =- g\eta-H(\eta)-\mez \vert \nabla_x \psi \vert^2 + \mez \frac{(\nabla_x \eta \cdot \nabla_x \psi + G(\eta)\psi)^2}{1+ \vert \nabla_x \eta \vert^2}.
\end{aligned}
\right.
\end{equation}
Here, $H(\eta)$ denotes the mean curvature of the free surface:
\[
H(\eta)=-\cnx\Big( \frac{\nabla\eta}{\sqrt{1+|\nabla\eta|^2}}\Big).
\]
The vertical and horizontal components of the velocity on $\Sigma$  can be expressed in terms of $\eta$ and $\psi$ as
\begin{equation}\label{BV}
B = (v_y)\arrowvert_\Sigma = \frac{ \nabla_x \eta \cdot \nabla_x \psi + G(\eta)\psi} {1+ \vert \nabla_x \eta \vert^2},\quad V= (v_x)\arrowvert_\Sigma  =\nabla_x \psi - B \nabla_x \eta.
 \end{equation}
As observed by Zakharov (see \cite{Zakharov1968} and the references therein), \eqref{ww} has a Hamiltonian canonical Hamiltonian structure
\[
\frac{\partial\eta}{\partial t}=\frac{\delta \mathcal{H}}{\delta\psi},\quad \frac{\partial\psi}{\partial t}=-\frac{\delta\mathcal{H} }{\delta\eta},
\]
where the Hamiltonian $\mathcal{H}$ is the total energy given by 
\bq\label{Hamiltonian}
\mathcal{H}=\frac{1}{2}\int_{\xR^d} \psi G(\eta)\psi\ dx+\frac{g}{2}\int_{\xR^d}\eta^2 dx+\int_{\xR^d}\big(\sqrt{1+|\nabla\eta|^2}-1\big)dx.
\eq
\subsection{Main results}
The Cauchy problem has been extensively studied, for example in Nalimov \cite{Nalimov}, Yosihara \cite{Yosihara},  Coutand- Shkoller \cite{CS}, Craig \cite{Craig1985}, Shatah-Zeng \cite{SZ1, SZ2, SZ3},  Ming-Zhang \cite{MiZh},  Lannes \cite{LannesLivre}: for sufficiently smooth solutions and  Alazard-Burq-Zuily \cite{ABZ1} for solutions at the energy threshold. See also Craig \cite{Craig1985}, Wu \cite{WuInvent, WuJAMS}, Lannes \cite{LannesJAMS} for the studies on gravity waves. Observe that the linearized system of \eqref{ww} about the rest state $(\eta=0, \psi=0)$ (modulo a lower order term, taking $g=0$) reads
\[
\begin{cases}
\partial_t\eta-|D_x|\psi=0,\\
\partial_t\psi-\Delta \eta=0.
\end{cases}
\]
Put $\Phi=|D_x|^\mez\eta+i\psi$, this becomes 
\bq\label{eq:lin}
\partial_t\Phi +i|D_x|^\tdm \Phi=0.
\eq
Therefore, it is natural to study \eqref{ww} at the following algebraic scaling
\[
(\eta, \psi)\in H^{s+\mez}(\xR^d)\times H^s(\xR^d).
\]
From the formula \eqref{BV} for the velocity trace, we have that the Lipschitz threshold in \cite{ABZ1} corresponds to $s> 2+d/2.$  On the other hand, the threshold $s>3/2+d/2$ suggested by the quasilinear nature \eqref{intro:reduce} is also {\it the minimal Sobolev index to ensure that the mean curvature  $H(\eta)$ is bounded}. The question we are concerned with is the following:\\
\hk ({\bf Q}) Is the Cauchy problem for \eqref{ww} solvable for initial data 
\bq\label{intro:data}
(\eta_0, \psi_0)\in H^{s+\mez}\times H^s,\quad s>\tmez+\frac d2?
\eq
Assume now that 
\bq\label{intro:reg0}
(\eta,\psi)\in L^{\infty}\lp[0, T]; H^{s+\mez}\times H^s\rp\cap L^p\lp[0, T]; W^{r+\mez, \infty}\times W^{r, \infty}\rp
\eq
with 
\bq\label{intro:reg}
 s>\tdm+\frac{d}{2},\quad r >2
\eq
 is a solution with prescribed data as in \eqref{intro:data}. We shall prove in Proposition \ref{singleeq:Phi} that the quasilinear reduction \eqref{intro:reduce} of system \eqref{ww} still holds with the right-hand-side term $f(t, x)$ satisfying a tame estimate, meaning that it is linear with respect to the H\"older norm. To be concise in the following statements, let us define the quantities that control the system (see Definition \ref{defi:PL} for the definitions of functional spaces):
\begin{align*}
&\text{Sobolev norms}:~M_{\sigma,T}=\Vert (\eta, \psi)\Vert_{L^{\infty}([0, T]; H^{\sigma+\mez}\times H^\sigma)},~ M_{\sigma,0}=\Vert (\eta_0, \psi_0)\Vert_{H^{\sigma+\mez}\times H^\sigma},\\
%&\text{Blow-up norm}: N_{\sigma,T}=\Vert (\eta, \psi)\Vert_{L^1([0, T]; W^{\sigma+\mez, \infty}\times W^{\sigma, \infty})},\\
& \text{``Strichartz norm''}:N_{\sigma,T}=\Vert (\eta, \nabla\psi)\Vert_{L^1([0, T]; W^{\sigma+\mez, \infty}\times \Bl^1)}.
\end{align*}
\hk Our first result concerns an a priori estimate for the Sobolev norm $M_{s, T}$ in terms of itself and the Strichartz norm $N_{r,T}$.
\begin{theo}\label{intro:theo:apriori}
Let~$d\geq 1$, $h>0$, $r>2$ and $s>\tdm+\frac{d}{2}$. Then there exists a  nondecreasing function~$\cF:\xR^+\to \xR^+$, depending only on $(d, s, r, h)$, such that: for all $T\in (0, 1]$ and all $(\eta, \psi)$ solution to \eqref{ww} on $[0, T]$ with
\begin{align*}
&(\eta,\psi)\in L^{\infty}\lp[0, T]; H^{s+\mez}\times H^s\rp,\\
&(\eta, \nabla\psi)\in L^1\lp[0, T]; W^{r+\mez, \infty}\times \Bl^1\rp,\\
&\inf_{t\in [0,T]}\dist(\eta(t), \Gamma)>h,
\end{align*}
 there holds
\[
 M_{s, T}\leq\cF\big( M_{s,0}+T\cF(M_{s, T})+N_{r,T}\big).
\]
\end{theo}
\begin{rema} Some comments are in order with respect to the preceding a priori estimate.
\begin{enumerate}
\item[1.]  We require only $\nabla \psi\in \Bl^1$ instead of $\psi\in W^{r,\infty}$. 
\item[2.] The function $\cF$ above can be highly nonlinear. It is not simply a straightforward outcome of a Gr\"onwall inequality but also comes from  estimates of the Dirichlet-Neumann operator in Sobolev spaces and Besov spaces (see the proof of Theorem \ref{theo:aprioriestimate}).
\item[3.] When $s>2+d/2$ one can take $r=s-\frac d2$ and retrieves by Sobolev embeddings the a priori estimate of \cite{ABZ1} (see Proposition $5.2$ there).
\end{enumerate}
\end{rema}
\hk  Our second result provides a blow-up criterion for solutions at the  energy threshold constructed in \cite{ABZ1}. Let $\Cs^r$ denote the Zymund space of order $r$ (see Definition \ref{defi:PL}). Note that $\Cs^r=W^{r, \infty}$ if $r\in (0, \infty)\setminus \{1, 2, 3, ...\}$ while $W^{r, \infty} \subsetneq \Cs^r$ if  $r\in\{0, 1, 2, ...\}$.
\begin{theo}\label{intro:theo:blowup}
Let~$d\geq 1$,  $h>0$ and $\sigma>2+\frac{d}{2}$. Let
\[ 
(\eta_0, \psi_0)\in  H^{\sigma+\mez}\times H^\sigma,\quad \dist(\eta_0, \Gamma)>h>0.
\]
Let $T^*=T^*(\eta_0, \psi_0, \sigma, h)$ be the maximal time of existence defined by \eqref{def:T*} and
 \[
(\eta,\psi)\in L^{\infty}\lp[0, T^*); H^{\sigma+\mez}\times H^\sigma\rp
\]
be the maximal solution of ~\eqref{ww} with prescribed data $(\eta_0, \psi_0)$. If ~$T^*$ is finite, then for all $\eps>0$,
\bq\label{intro:criter1}
P_\eps(T^*)+\int_0^{T^*}Q_\eps(t)dt +\frac{1}{h(T^*)}=+\infty,
\eq
where 
\[
\begin{aligned}
&P_\eps(T^*)=\sup_{t\in [0, T^*)}\Vert \eta(t)\Vert_{\Cs^{2+\eps}}+\Vert \nabla \psi(t)\Vert_{\Bl^0},\\
&Q_\eps(t)=\Vert \eta(t)\Vert_{\Cs^{\frac{5}{2}+\eps}}+\Vert \nabla\psi(t)\Vert_{\Cs^1},\\
&h(T^*)=\inf_{t\in [0, T^*)}\dist(\eta(t),\Gamma).
\end{aligned}
\]
Consequently, if $T^*$ is finite then for all $\eps>0$,
\bq\label{intro:criter2}
 P^0_\eps(T^*)+\int_0^{T^*}Q^0_\eps(t)dt+\frac{1}{h(T^*)}=+\infty,
\eq
where
\[
\begin{aligned}
&P^0_\eps(T^*)=\sup_{t\in [0, T^*)}\Vert \eta(t)\Vert_{\Cs^{2+\eps}}+\Vert (V, B)(t)\Vert_{\Bl^0},\\
&Q^0_\eps(t)=\Vert \eta(t)\Vert_{\Cs^{\frac{5}{2}+\eps}}+\Vert (V, B)(t)\Vert_{\Cs^1}.
\end{aligned}
\]
\end{theo}
\begin{rema}
\begin{enumerate}
\item[1.] We shall prove in Proposition \ref{prop:grow} below that the Sobolev norm $\Vert (\eta, \psi)\Vert_{L^\infty([0, T]; H^{\sigma+\mez}\times H^\sigma)}$, $\sigma>2+\frac{d}{2}$, is bounded by a double exponential 
\[
\exp\big(e^{C(T)\int_0^TQ_\eps(t)dt}\big)
\]
where $C(T)$ depends only on the lower norm $P_\eps(T)$. In the preceding estimate, $Q_\eps$ can be replaced by $Q_\eps^0$ by virtue of \eqref{psi-VB}. These bounds are reminiscent of the well-known result due to Beale-Kato-Majda \cite{BKM} for the incompressible Euler equations in the whole space, where the $\Cs^1$-norm of the velocity was sharpened to the $L^\infty$-norm of the vorticity. An analogous result in bounded, simply connected domains was obtained by Ferrari \cite{Ferrari}.
\item[2.] If in  $Q_\eps$ the Zygmund norm $\Vert \nabla \psi\Vert_{\Cs^1}$ is replaced by the stronger norm $\Vert \nabla \psi\Vert_{\Bl^1}$, then one obtains the following exponential bound (see Remark \ref{rema:expo})
\[
\Vert (\eta, \psi)\Vert_{L^\infty([0, T]; H^{\sigma+\mez}\times H^\sigma)}\le C(T)\Vert (\eta(0), \psi(0))\Vert_{H^{\sigma+\mez}\times H^\sigma}\exp\big(C(T)\int_0^TQ_\eps(t)dt\big),
\]
where $C(T)$ depends only on the lower norm $P_\eps(T)$ and $\sigma>\tdm+\frac{d}{2}$. The same remarks applies to $Q^0_\eps$ and $(V, B)$.
\end{enumerate}
\end{rema}
\hk In the survey paper \cite{CrWa} Craig-Wayne posed (see Problem $3$ there) the following questions on {\it How do solutions break down?}: \\
{\bf (Q1)}~{\it For which $\alpha$ is it true that, if one knows a priori that $\sup_{[-T, T]}\| (\eta, \psi)\|_{C^\alpha}<+\infty$  then $C^\infty$ data $(\eta_0, \psi_0)$ implies that the solution is $C^\infty$ over the time interval $[-T, T]$?}\\
{\bf (Q2)}~{\it 
It would be more satisfying to say that the solution fails to exist because the "curvature of the surface has diverged at some point", or a related geometrical and/or physical statement}.\\
\hk With regard to question ({\bf Q1}), we deduce from Theorem \ref{intro:theo:blowup} (more precisely, from \eqref{intro:criter1})  the following persistence of Sobolev regularity.
\begin{coro}\label{intro:prop:regularity}
Let $T\in (0, +\infty)$ and $(\eta, \psi)$ be a distributional solution to  \eqref{ww} on the time interval $[0, T]$  such that  $\inf_{[0, T]}\dist(\eta(t), \Gamma)>0.$ Then the following property holds: if one knows a priori that for some $\eps_0>0$
\bq\label{condition:reg}
\sup_{[0, T]}\| (\eta(t), \nabla\psi(t))\|_{\Cs^{\frac{5}{2}+\eps_0}\times \Cs^{1}}<+\infty,
\eq
 then $(\eta(0), \psi(0))\in  H^\infty(\xR^d)^2$ implies that $(\eta, \psi)\in  L^\infty([0, T]; H^\infty(\xR^d))^2$.
\end{coro}
 Theorem \ref{intro:theo:blowup}  gives a partial answer to {\bf (Q2)}. Indeed, the criterion \eqref{intro:criter2} implies that the solution fails to exist if
\begin{itemize}
\item the Lipschitz norm of the velocity trace explodes, {\it i.e.}, $\sup_{[0, T^*)}\Vert (V, B)\Vert_{W^{1,\infty}}=+\infty$, or
\item  the bottom rises to the surface, {\it i.e.}, $h(T^*)=0$.
\end{itemize}
Some results are known about blow-up criteria for pure gravity water waves (without surface tension). Wang-Zhang \cite{WaZh} obtained a result stated in terms of the curvature $H(\eta)$ and the gradient of the velocity trace
\bq\label{criter:Wa}
\int_0^{T^*}\Vert (\nabla V, \nabla B)(t)\Vert^6_{L^\infty}dt+\sup_{t\in [0, T^*)}\Vert H(\eta(t))\Vert_{L^2\cap L^p}=+\infty,\quad p>2d.
\eq
 Thibault \cite{Thibault} showed, for highest regularities, 
\[
\int_0^{T^*}\big(\Vert \eta\Vert_{C^{\tdm+}}+\Vert (V, B)\Vert_{C^{1+}} \big)dt=+\infty;
\]
 the temporal integrability was thus improved.  In two space dimensions, using holomorphic coordinates, Hunter-Ifrim-Tataru \cite{HuIfTa} obtained a sharpened criterion with $\Vert (V, B)\Vert_{C^{1+}}$ replaced by $\Vert (\nabla V, \nabla B)\Vert_{BMO}$. Also in two space dimensions, Wu \cite{Wu2015} proved a blow-up criterion using the energy constructed by Kinsey-Wu \cite{KiWu}, which concerns water waves with angled crests, hence the surface is even not Lipschitz. Remark that all the above results but \cite{Thibault} consider the bottomless case. In a more recent paper, \cite{Wa2} considered rotational fluids and obtained
\[
\sup_{t\in [0, T^*)}\big(\Vert v(t)\Vert_{W^{1,\infty}}+\Vert H(\eta(t))\Vert_{L^2\cap L^p}\big)=+\infty\quad p>2d,
\]
$v$ being the Eulerian velocity. In order to obtain the sharp regularity for $\nabla \psi$ and $(V, B)$ in Theorem \ref{intro:theo:blowup}, we shall use a technical idea from \cite{Wa2}: deriving elliptic estimates in Chemin-Lerner type spaces.\\
\hk  Finally, we observe that the relation \eqref{intro:reg} exhibits a gap of $1/2$ derivative from $H^s$ to $W^{2,\infty}$ in terms of Sobolev's embedding. To fill up this gap we need to take into account the dispersive property of water waves to prove a Strichartz estimate with a gain of $1/2$ derivative. As remarked in \cite{NgPo} this gain can be achieved for the 3D linearized system ({\it i.e.} $d=2$) and corresponds to the so called {\it semiclassical Strichartz estimate}. The proof of Theorem \ref{theo:contraction} on the Lipschitz continuity of the solution map shows that if the semiclassical Strichartz estimate were proved, this theorem would hold with the gain $\mu=\mez$ in \eqref{gain:flowmap} (see Remark \ref{rema:contraction}). Then, applying Theorems \ref{intro:theo:apriori},  \ref{intro:theo:blowup} one would end up with an affirmative answer for ({\bf Q}) by implementing the standard method of regularizing initial data. Therefore, the problem boils down to studying Strichartz estimates for \eqref{ww}. As a first effort in this direction, we prove in the companion paper \cite{NgPo} Strichartz estimates with an intermediate gain $0<\mu<1/2$ which yields a Cauchy theory (see Theorem $1.6$, \cite{NgPo}) in which the initial velocity may fail to be Lipschitz (up to the boundary) but becomes Lipschitz at almost all later time; this is an analogue of the result in \cite{ABZ4} for pure gravity waves. \\
\hk The article is organized as follows. Section \ref{section:DN} is devoted to the  study of the Dirichlet-Neumann operator in Sobolev spaces, Besov spaces and Zygmund spaces. Next, in Section \ref{section:para} we adapt the method in \cite{ABZ1} to paralinearize and then symmetrize system \eqref{ww} at our level of regularity \eqref{intro:reg}. With this reduction, we use the standard energy method to derive an a priori estimate and a blow-up criterion in Section \ref{section:apriori}. Section \ref{section:contraction} is devoted to contraction estimates; more precisely, we establish the Lipschitz continuity of the solution map in weaker norms. Finally, we gather some basic features of the paradifferential calculus  and some technical results in Appendix \ref{Appendix}.
\section{Elliptic estimates and the Dirichlet-Neumann operator}\label{section:DN}
% We devote this section to study the Dirichlet-Neumann operator $G(\eta)f$. We shall establish 
\subsection{Construction of the Dirichlet-Neumann operator}\label{section:constructDN}
Let ~$\eta \in W^{1, \infty}({\mathbf{R}}^d)$  and~$ f\in H^{\mez}({\mathbf{R}}^d)$.  In order to define the Dirichlet-Neumann operator $G(\eta)f$, we consider the boundary value problem
\begin{equation}\label{elliptic:prob}
\Delta_{x,y}\phi=0\text{ in }\Omega,\quad
\phi\arrowvert_{\Sigma}=f,\quad \partial_n \phi\arrowvert_{\Gamma}=0.
\end{equation}
For any $h'\in (0, h]$, define  the curved strip of width $h'$ below the free surface
\begin{equation}
   	\Omega_{h'}:=\lB(x,y):x\in\xR^d,\eta(x)-h'<y<\eta(x)\rB.
\end{equation}
We recall here the construction of the variational solution to \eqref{elliptic:prob} in \cite{ABZ3}. 
\begin{nota}\label{notaD}
Denote by~$\mathscr{D}$ the space of functions~$u\in C^{\infty}(\Omega)$ such that
$\nabla_{x,y}u\in L^2(\Omega)$. We then define~$\mathscr{D}_0$ 
as the subspace of functions~$u\in \mathscr{D}$ such that~$u$ is
equal to~$0$ in a neighborhood of  the top boundary~$\Sigma$.
\end{nota}
\begin{prop}[\protect{see \cite[Proposition 2.2]{ABZ1}}]\label{corog}
There exists a positive weight~$g\in L^\infty_{loc} ( \Omega)$, locally bounded from below, equal to~$1$
near the top boundary of~$\Omega$, say in $\Omega_h$, and a constant $C>0$ 
such that for all~$u\in \mathscr{D}_0$,
\begin{equation}\label{eq.Poinc}
\iint_\Omega g(x,y) |u(x,y)|^2 \,dx dy \leq  C \iint_{\Omega} |\partialyx u (x,y)|^2 \,dx dy.
\end{equation}
\end{prop}
\begin{defi}\label{defi:H10}
Denote by~$H^{1,0}( \Omega)$ the completion of $\mathscr{D}_0$ under the norm
\[
\|u\|_*:= \|u\|_{L^2( \Omega, g(x,y)dxdy)}+\|\nabla_{x,y}u\|_{L^2(\Omega, dxdy)}.
\]
Owing to the Poincar\'e inequality \eqref{eq.Poinc}, ~$H^{1,0}(\Omega)$ endowed with the norm  $\Vert u\Vert = \Vert \partialyx u \Vert_{L^2( \Omega)}$ is a Hilbert space, see Definition $2.6$~\cite{ABZ1}.
\end{defi}
Now, let $\chi_0\in C^\infty(\xR)$ be such that $\chi_0(z)=1$ if $z\ge -\frac{1}{4}$, $\chi_0(z)=0$ if $z\le -\mez$.  Then with $f\in H^\mez$, define
\[
f_1(x,z)=\chi_0(z)e^{z\langle D_x\rangle}f(x),\quad x\in \xR^d,~z\le 0.
\]
Next, define 
\bq\label{defi:underf}
\underline f(x,y)=f_1(x, \frac{y-\eta(x)}{h}),\quad(x,y)\in \Omega.
\eq
This "lifting" function satisfies $\underline f\rvert_{y=\eta(x)}=f(x)$, $\underline f\equiv 0$ in $\Omega\setminus \Omega_{h/2}$ and 
\bq\label{est:underf}
\Vert \underline f\Vert_{H^1(\Omega)}\le K(1+\Vert \eta\Vert_{W^{1,\infty}})\Vert f\Vert_{H^\mez(\xR^d)}.
\eq
The map 
$$
\varphi\mapsto - \int_{\Omega}  \partialyx  \underline{f} \cdot \partialyx \varphi \, dx dy
$$
is thus a bounded linear form on~$H^{1,0}(\Omega)$. The Riesz theorem then provides  a unique ${u}\in H^{1,0}(\Omega)$ such that
\begin{equation}\label{eq:variational}
\forall \varphi \in H^{1,0} ( \Omega), \qquad \iint_\Omega \partialyx  u \cdot \partialyx \varphi \, dxdy
= - \iint_{\Omega}  \partialyx  \underline{f} \cdot \partialyx \varphi \, dx dy.
\end{equation}
\begin{defi}
With $\underline f$ and $u$ constructed as above, the function $\phi:=u+\underline f$ is defined to be the variational solution of the problem \eqref{elliptic:prob}.  The Dirichlet-Neumann operator is defined formally by
\begin{equation}\label{def:DN}
  G(\eta)\psi  
 = \sqrt{1+|\nabla \eta|^2}\, \partial_n\phi \big\arrowvert_{y=\eta(x)}
=\big[ \partial_y \phi-\partialx \eta\cdot \partialx \phi \big] \, \big\arrowvert _{y=\eta(x)} .
\end{equation}
\end{defi}
As a consequence of \eqref{est:underf} and \eqref{eq:variational}, the variational solution $\phi$ satisfies
\bq\label{est:vari}
\Vert \nabla_{x,y}\phi\Vert_{L^2(\Omega)} \le K(1+\Vert \eta\Vert_{W^{1,\infty}})\Vert f\Vert_{H^\mez(\xR^d)}.
\eq
Moreover, it was proved in \cite{Thibault} the following maximum principle.
\begin{prop}[see \protect{\cite[Proposition~$2.7$]{Thibault}}]\label{max:prin}
Let $\eta\in W^{1,\infty}(\xR^d)$ and $f\in H^\mez(\xR^d)$. There exists a constant $C>0$ independent of $\eta$, $\psi$ such that 
\[
\Vert \phi\Vert_{L^\infty(\Omega)}\le  C\Vert f\Vert_{L^\infty(\xR^d)}.
\]
\end{prop}
The continuity of $G(\eta)$ in Sobolev spaces is given in the next theorem.
\begin{theo}[see \protect{\cite[Theorem~$3.12$]{ABZ3}}]\label{DN:ABZ}
Let~$d\ge 1$,~$s>\mez+\frac{d}{2}$ and~$\frac 1 2 \leq \sigma \leq s+ \frac 1 2$. For all $\eta\in H^{s+\mez}({\mathbf{R}}^d)$, the operator 
\[
G(\eta): \quad H^\sigma \to H^{\sigma-1}
\]
is continuous.  Moreover, there exists a nondecreasing function~$\mathcal{F}\colon\xR_+\rightarrow\xR_+$ 
such that, for all 
$\eta\in H^{s+\mez}({\mathbf{R}}^d)$ and all~$f\in H^{\sigma}(\xR^d)$, there holds
\begin{equation}\label{ests+2}
\Vert G(\eta)f \Vert_{H^{\sigma-1}}\le 
\mathcal{F}(\| \eta \|_{H^{s+\mez}})\Vert f\Vert_{H^\sigma}.
\end{equation}
\end{theo}
\subsection{Elliptic estimates}\label{section:elliptic}
The Dirichlet-Neumann requires the regularity of $\nabla_{x,y}\phi$ at the free surface. We follow \cite{LannesJAMS} and \cite{ABZ3} straightening out $\Omega_h$ using the map
\bq\label{defi:rho}
\begin{aligned}
&\rho(x,z)= (1+z)e^{\delta z\L{D_x}}\eta(x)-z\lB e^{-(1+z)\delta\L{D_x}}\eta(x)-h\rB\\
& (x,z)\in S:=\xR^d\times (-1, 0).
\end{aligned}
\eq
According to Lemma $3.6$, \cite{ABZ3}, there exists an absolute constant $K>0$ such that if $\delta \Vert \eta\Vert_{W^{1,\infty}}\le K$ then 
\bq\label{drho:lowerbound}
\partial_z\rho\ge \frac{h}{2}
\eq
and the map $(x,z)\mapsto(x,\rho(x,z))$ is thus a Lipschitz diffeomorphism from~$S$ to~$\Omega_h$. Then if we call
\begin{equation}\label{defi:v}
v(x,z)=\phi(x,\rho(x,z))\quad \forall(x,z)\in S
\end{equation}
the image of~$\phi$ via this diffeomorphism, it solves 
\begin{equation}\label{eq:v}	\mathcal{L}v:=(\partial_z^2+\alpha\Delta_x+\beta\cdot\nabla_x\partial_z-\gamma\partial_z)v=0\quad\text{in}~S
\end{equation}
where
\begin{equation*}
	\begin{gathered}
		\alpha:=\frac{(\partial_z\rho)^2}{1+\la\nabla_x\rho\ra^2},
		  \quad\beta:=-2\frac{\partial_z\rho\nabla_x\rho}{1+\la\nabla_x\rho\ra^2},\quad
		\gamma:=\frac{1}{\partial_z\rho}(\partial_z^2\rho+\alpha\Delta_x\rho+\beta\cdot\nabla_x\partial_z\rho).
	\end{gathered}
\end{equation*}
\subsubsection{Sobolev estimates}
Define the following interpolation spaces
\bq
\begin{aligned}
	X^\mu(J)&=C_z(I;H^\mu(\R^d))\cap L^2_z(J;H^{\mu+\mez}(\R^d)),\\
	Y^\mu(J)&=L^1_z(I;H^\mu(\R^d))+L^2_z(J;H^{\mu-\mez}(\R^d)).
\end{aligned}
\eq
Remark that $\|\cdot\|_{Y^\mu(J)}\le \|\cdot\|_{X^{\mu-1}(J)}$ for any $\mu\in \xR$. We get started by providing estimates for the coefficients $\alpha, \beta,\gamma$. We refer the reader to Appendix \ref{Appendix} for a review of the paradifferential calculus and notations of functional spaces. 
\begin{nota}
We will denote $\cF$ any nondecreasing function from $\xR^+$ to $\xR^+$. $\cF$ may change from line to line but is independent of relevant parameters.
\end{nota}
\begin{lemm}
Denote  $I=[-1, 0]$.
\begin{enumerate}
\item[1.] For any $\sigma>\mez+\frac d2$ and $\eps>0$, there holds
\bq\label{alpha:S}
\Vert \alpha-h^2\Vert_{X^{\sigma-\mez}(I)}+\Vert \beta\Vert_{X^{\sigma-\mez}(I)}+\Vert\gamma\Vert_{X^{\sigma-\tdm}(I)}\le \cF(\Vert \eta\Vert_{\Cs^{1+\eps}})\Vert \eta\Vert_{H^{\sigma+\mez}}.
\eq
\item[2.] If $\mu >\tdm$ then
\begin{align}
&\Vert \alpha\Vert_{C(I; \Cs^{\mu-\mez})}+\Vert \beta\Vert_{C(I; \Cs^{\mu-\mez})}+\Vert\gamma\Vert_{C(I; \Cs^{\mu-\tmez})}\le \cF(\Vert \eta\Vert_{\Cs^{\mu+\mez}}),\label{alpha:C:0}\\
&\Vert \alpha\Vert_{\wL^2(I; \Cs^{\mu})}+\Vert \beta\Vert_{\wL^2(I; \Cs^{\mu})}+\Vert\gamma\Vert_{\wL^2(I; \Cs^{\mu-1})}\le \cF\big(\Vert \eta\Vert_{\Cs^{\mu+\mez}}+\Vert \eta\Vert_{L^2}\big).\label{alpha:C:2}
\end{align}
\end{enumerate}
\begin{proof}
These estimates stem from estimates for derivatives of $\rho$. For the proof of \eqref{alpha:S} we refer the reader to Lemmas $3.7$ and $3.19$ in \cite{ABZ3}. Concerning \eqref{alpha:C:0} we remark that $\alpha$ and $\beta$ involve merely derivatives up to order $1$ of $\eta$ while $\gamma$ involves second order derivatives of $\eta$. Finally, for \eqref{alpha:C:2} we use the following smoothing property of the Poison kernel in the high frequency regime (see Lemma $2.4$, \cite{BCD} and Lemma $3.2$, \cite{WaZh}): for all $\kappa>0$ and $p\in [1, \infty]$, there exists $C>0$ such that for all $j\ge 1$,
\[
\Vert e^{-\kappa \langle D_x\rangle}\Delta_ju\Vert_{L^p(\xR^d)}\le Ce^{-C2^j}\Vert \Delta_j u\Vert_{L^p(\xR^d)},
\] 
where, we recall the dyadic partition of unity in Definition \ref{defi:PL}: $
\text{Id}=\sum_{j=0}^\infty \Delta_j$. The low frequency part $\Delta_0$ can be trivially bounded by the $L^2$-norm using Bernstein's inequalities.
\end{proof}
\end{lemm}
We first use  the variational estimate \eqref{est:vari} to derive a regularity for $\nabla_{x,z}v$.
\begin{lemm}\label{est:base}
Let $f\in H^\mez$. Set  
\bq\label{defi:E}
E(\eta, f)=\Vert\nabla_{x,y}\phi\Vert_{L^2(\Omega_h)}.
\eq
\begin{enumerate}
\item[1.] If $\eta\in \Cs^{\tdm+\eps}$ with $\eps>0$ then $\nabla_{x,z}v\in C([-1,0 ]; H^{-\mez})$ and
\begin{align}
\label{est:base:Cx}
&\Vert \nabla_xv\Vert_{X^{-\mez}([-1, 0])}\le \cF(\Vert \eta\Vert_{\Cs^{1+\eps}})E(\eta, f)\\
\label{est:base:Cz}
&\Vert \nabla_zv\Vert_{X^{-\mez}([-1, 0])}\le \cF(\Vert \eta\Vert_{\Cs^{1+\eps}})\big(1+\Vert \eta\Vert_{\Cs^{\tdm+\eps}}\big)E(\eta, f).
\end{align}
\item[2.] If $\eta\in H^{s+\mez}$ with $s>\mez+\frac{d}{2}$ then  $\nabla_{x,z}v\in C([-1, 0 ]; H^{-\mez})$ and 
\bq\label{est:base:H}
\Vert \nabla_{x,z}v\Vert_{X^{-\mez}([-1, 0])}\le \cF(\Vert \eta\Vert_{H^{s+\mez}})E(\eta, f).
\eq
\end{enumerate}
\end{lemm}
\begin{rema}\label{rema:E}
1. By \eqref{est:vari}, we have 
\[
E(\eta,f)\le K\big(1+\Vert \eta\Vert_{W^{1,\infty}})\Vert f\Vert_{H^\mez}.
\]
However, we keep in the estimates \eqref{est:base:Cx}-\eqref{est:base:Cz} the quantity $E(\eta, f)$ instead of $\Vert f\Vert_{H^\mez}$ because $E(\eta, f)$ is controlled by the Hamiltonian, which is conserved under the flow. Moreover, as we shall derive blow-up criteria involving only Holder norms of the solution, we avoid using $\Vert f\Vert_{H^\mez}$.\\
2. The estimates \eqref{est:base:Cx}, \eqref{est:base:Cz}, \eqref{est:base:H} were proved in Proposition $4.3$, \cite{WaZh} as {\it a priori estimates} (see the proof there). It is worth noting that we establish here a real regularity result.
\end{rema}
\begin{proof}
Denote $I=[-1, 0]$. \\
1. Observe first that by changing variables, 
\bq\label{est:base:1}
\Vert \nabla_{x,z}v\Vert_{L^2(I; L^2)}\le  \cF(\Vert \eta\Vert_{W^{1,\infty}})\Vert \nabla_{x,y}\phi\Vert_{L^2(\Omega_h)}= \cF(\Vert \eta\Vert_{W^{1,\infty}})E(\eta, f).
\eq
Applying the interpolation Lemma \ref{lemm:inter}, we obtain $\nabla_xv\in X^{-\mez}(I)$ and 
\bq\label{dxu:base}\begin{aligned}
\Vert\nabla_{x}v\Vert_{X^{-\mez}(I)}&\les \Vert\nabla_{x}v\Vert_{L^2(I; L^2)}+\Vert \partial_z\nabla_{x}v\Vert_{L^2(I; H^{-1})}\\
&\les \Vert \nabla_{x,z}v\Vert_{L^2(I; L^2)}\le \cF(\Vert \eta\Vert_{W^{1,\infty}})E(\eta, f).
\end{aligned}
\eq
We are left with \eqref{est:base:Cz}. Again, by virtue of Lemma \ref{lemm:inter} and \eqref{est:base:1}, it suffices to prove 
\[
\Vert \partial^2_zv\Vert_{L^2(I; H^{-1})}\le \cF(\Vert \eta\Vert_{\Cs^{1+\eps}})\big(1+\Vert \eta\Vert_{\Cs^{\tdm+\eps}}\big)E(\eta, f)
\]
A natural way is to compute $\partial_z^2v$ using \eqref{eq:v} 
\begin{align*}
\partial_z^2v&=-\alpha\Delta_xv-\beta\cdot\nabla_x\partial_zv+\gamma\partial_zv
\end{align*}
and then estimate the right-hand side. However, this will lead to a loss of $\mez$ derivative of $\eta$. To remedy this,  further cancellations coming from the structure of the equation need to be invoked. We have 
\begin{align*}
&(\partial_y\phi)(x,\rho(x,z))=\frac{1}{\partial_z\rho}\partial_zv(x,z)=:(\Lambda_1v)(x,z),\\
&(\nabla_x\phi)(x,\rho(x,z))=\big(\nabla_x-\frac{\nabla_x\rho}{\partial_z\rho}\partial_z\big)v(x,z)=:(\Lambda_2v)(x,z).
\end{align*}
Set $U:=\Lambda_1v-\nabla_x\rho\Lambda_2v$, whose trace at $z=0$ is actually equal to $G(\eta)f$. Then, using the equation $\Delta_{x,y}\phi=0$, it was proved in \cite{ABZ3} (see the formula (3.19)  there)  that $\partial_zU$ has the divergence form
\[
\partial_zU=\nabla_x\cdot\big(\partial_z\rho\Lambda_2v\big).
\]
Then, by the interpolation Lemma \ref{lemm:inter}, it is readily seen that $U\in C(I; H^{-\mez})$ and
\[
\Vert U\Vert_{C(I; H^{-\mez})}\les \cF(\Vert \eta\Vert_{\Cs^{1+\eps}})E(\eta, f).
\]
Now, from the definition of $\Lambda_{1,2}$ one can compute
\[
\partial_zv=\frac{(U+\nabla_x\rho\cdot\nabla_xv)\partial_z\rho}{1+|\nabla_x\rho|^2}=:Ua+\nabla_x v\cdot b
\]
with 
\[
a:=\frac{\partial_z\rho}{1+|\nabla_x\rho|^2},\quad b:=\frac{\partial_z\rho\nabla_x\rho}{1+|\nabla_x\rho|^2}.
\]
We write $Ua=T_aU+T_Ua+R(a, U)$. By Theorem \ref{theo:sc} (i),
\[
\Vert T_a U\Vert_{C(I; H^{-\mez})}\les \Vert a\Vert_{C(I; L^\infty)}\Vert U\Vert_{C(I; H^{-\mez})}\les \cF(\Vert \eta\Vert_{\Cs^{1+\eps}})E(\eta, f).
\]
The term $T_Ua$ can be estimated by means of Lemma \ref{pest:n} as
\[
\Vert T_Ua\Vert_{C(I; H^{-\mez})}\les \Vert U\Vert_{C(I; H^{-\mez})}\Vert a\Vert_{C(I; \Cs^\eps)}\les \cF(\Vert \eta\Vert_{\Cs^{1+\eps}})E(\eta, f).
\]
Finally, for the remainder $R(a, U)$ we use \eqref{Bony2}, which leads to a loss of $\mez$ derivative for $\eta$, to get
\[
\Vert R(U, a)\Vert_{C(I; H^{-\mez})}\les \Vert a\Vert_{C(I; \Cs^{\mez+\eps})}\Vert U\Vert_{C(I; H^{-\mez})}\les \cF(\Vert \eta\Vert_{\Cs^{1+\eps}})\big(1+\Vert \eta\Vert_{\Cs^{\tdm+\eps}}\big)E(\eta, f)
\]
where, we have used
\[
\Vert a\Vert_{C(I; \Cs^{\mez+\eps})}\les \cF(\Vert \eta\Vert_{\Cs^{1+\eps}})\big(1+\Vert \eta\Vert_{\Cs^{\tdm+\eps}}\big).
\]
Finally, the term $b\nabla_x v$ can be treated using the same argument as we have shown that $\nabla_xv\in C(I; H^{-\mez})$. The proof of \eqref{est:base:Cz} is complete.\\
2. We turn to prove \eqref{est:base:H}. Observe that by the embedding
\bq\label{embedding:HC}
\Vert \eta\Vert_{C^{1+\eps}}\le C\Vert \eta\Vert_{H^{s+\mez}}
\eq
with $0<\eps<s-\mez-\frac d2$, \eqref{est:base:Cx} implies the estimate of $\nabla_xv$ in \eqref{est:base:H}. For $\partial_zv$, we follow the above proof of \eqref{est:base:Cz}. It suffices to prove $aU\in C(I; H^{-\mez})$ with norm bounded by the right-hand side of \eqref{est:base:H}. To this end, we write $aU=hU+T_{(a-h)}U+T_U(a-h)+R(a-h, U)$. The proof of \eqref{est:base:Cz}, combined with \eqref{embedding:HC}, shows that 
\[
\Vert T_U(a-h)\Vert_{C(I; H^{-\mez})}+\Vert T_{(a-h)} U\Vert_{C(I; H^{-\mez})}\les \cF(\Vert \eta\Vert_{H^{s+\mez}})E(\eta, f).
\]
Finally, by applying \eqref{Bony1}  (notice that $\frac d2\ge \mez$) and using the estimate 
\[
\Vert a\Vert_{C(I; H^{\frac{d}{2}+\eps})}\les \Vert a\Vert_{C(I; H^{s-\mez})}\les \cF(\Vert \eta\Vert_{H^{s+\mez}}),
\]
we conclude that
\[
\Vert R(U, a)\Vert_{C(I; H^{-\mez})}\les\Vert U\Vert_{C(I; H^{-\mez})}\Vert a\Vert_{C(I; H^{\frac{d}{2}+\eps})}\cF(\Vert \eta\Vert_{H^{s+\mez}})E(\eta, f).
\]
\end{proof} 
According to the preceding lemma, the trace $\nabla_{x,z}v\rvert_{z=0}$ is well-defined and belongs to $H^{-\mez}$. Estimates in higher order Sobolev spaces are given in the next proposition.
\begin{prop}[see \protect{\cite[Proposition~$3.16$]{ABZ3}}]\label{elliptic:ABZ}
Let $s> \mez + \frac d 2,~- \frac 1 2 \leq \sigma \leq s - \frac 1 2$.  Assume that $\eta\in H^{s+\mez}$ and $f\in H^{\sigma+1}$ and for some  $z_0\in (-1, 0)$
\[
\Vert\nabla_{x,z}v\Vert_{X^{-\mez}([z_0, 0])}<+\infty.
\]
Then for any~$z_1\in (-1,0)$, $z_1>z_0$, we have $\nabla_{x,z}v \in X^{\sigma}([z_1,0])$ and
$$
\lA \nabla_{x,z} v\rA_{X^{\sigma}([z_1,0])}
\le \mathcal{F}(\| \eta \|_{H^{s+\mez}})\lB \lA f\rA_{H^{\sigma+1}}+\Vert\nabla_{x,z}v\Vert_{X^{-\mez}([z_0, 0])}\rB,
$$
where $\cF$ depends only on~$\sigma$ and $z_0, z_1$. 
\end{prop}
A combination of \eqref{est:base:H} and Remark \ref{rema:E} implies 
\[
\Vert \nabla_{x,z}v\Vert_{X^{-\mez}([-1, 0])}\le \cF(\Vert \eta\Vert_{H^{s+\mez}})\Vert f\Vert_{H^\mez}
\]
provided $s>\mez+\frac{d}{2}$. With the aid of Proposition \ref{elliptic:ABZ}, we prove the following identity, which will be used later in the proof of blow-up criteria.
\begin{prop}\label{identity}
Let $s>\mez+\frac{d}{2}$. Assume that $\eta \in H^{s+\mez}$ and $f\in H^\tdm$. Then $\phi\in H^2(\Omega_{3h/4})$ and the following identity holds
\[
\int_{\xR^d}fG(\eta)f=\Vert \nabla_{x,y}\phi\Vert^2_{L^2(\Omega)}.
\]
\end{prop}
\begin{proof}
We first recall from the construction in subsection \ref{section:constructDN} that $\phi=u+\underline f$, where $\underline f$ is defined by \eqref{defi:underf} and $u\in H^{1,0}(\Omega)$ is the unique solution of \eqref{eq:variational}. By the Poincar\'e inequality of Lemma \ref{corog} and \eqref{eq:variational}, \eqref{est:underf}
\[
\Vert u\Vert_{L^2(\Omega_h)}\le C\Vert \nabla_{x,y}u\Vert_{L^2(\Omega)}\le K(1+\Vert \eta\Vert_{W^{1,\infty}})\Vert f\Vert_{H^\mez}.
\]
Therefore, $\phi\in L^2(\Omega_h)$ and thus, by \eqref{est:vari}, $\phi\in H^1(\Omega_h)$. Now, applying Proposition \ref{elliptic:ABZ} we have that $v=\phi(x, \rho(x,z))$ satisfies for any $z_1\in (-1, 0)$
\[
\Vert \nabla_{x,z}v\Vert_{L^2([z_1, 0]; H^1)}\le \cF(\Vert \eta\Vert_{H^{s+\mez}})\Vert f\Vert_{H^\tdm}.
\]
Then using equation \eqref{eq:v} together with the product rules one can prove that
\[
\Vert \partial_z^2v\Vert_{L^2([z_1, 0]; L^2)}\le \cF(\Vert \eta\Vert_{H^{s+\mez}})\Vert f\Vert_{H^\tdm}.
\]
By a change of variables we obtain $\nabla_{x,y}\phi\in H^1(\Omega_{3h/4})$ and thus $\phi\in H^2(\Omega_{3h/4})$. \\
Now, taking $\varphi=u\in H^{1,0}(\Omega)$ in the variational equation \eqref{eq:variational} gives
\[
\int_{\Omega}\nabla_{x,z}\phi\nabla_{x,z}u=0.
\]
Consequently
\[
\int_{\Omega}|\nabla_{x,y}\phi|^2=\int_{\Omega}|\nabla_{x,y}\phi|^2-\int_{\Omega}\nabla_{x,z}\phi\nabla_{x,z}u=\int_{\Omega}\nabla_{x,y}\phi\nabla_{x,y}\underline f.
\]
Since $\underline f\equiv 0$ in $\Omega\setminus\Omega_{h/2}$, this implies
\[
\int_{\Omega}|\nabla_{x,y}\phi|^2=\int_{\Omega_{3h/4}}\nabla_{x,y}\phi\nabla_{x,y}\underline f.
\]
We have proved that in $\Omega_{3h/4}$, the harmonic function $\phi$ is $H^2$. Notice in addition that $\phi\equiv 0$ near $\{y=\eta-3h/4\}$. As $\partial \Omega_{3h/4}$ is Lipschitz ($\eta \in H^{1+\frac{d}{2}+}\subset W^{1, \infty}$), an integration by parts then yields
\[
\int_{\Omega}|\nabla_{x,y}\phi|^2=\int_\Sigma \underline f\partial_n\phi=\int_{\xR^d} fG(\eta)f,
\]
which is the desired identity.
\end{proof}
The next proposition is an impovement of Proposition \ref{elliptic:ABZ} in the sense that it gives tame estimates  with respect to the highest derivatives of $\eta$ and $f$, provided $\nabla_{x,z}v\in L^\infty_zL^\infty_x$.
\begin{prop}[see \protect{\cite[Proposition~$2.12$]{Thibault}}] \label{elliptic:Th}
	Let  $s>\mez+\frac d2$, $-\mez\le \sigma\le s-\mez$. Assume that $\eta\in H^{s+\mez}$, $f\in H^{\sigma+1}$ and 
\[
\nabla_{x,z}v\in L^\infty([z_0, 0]; L^\infty)
\]
 for some $z_0\in (-1, 0)$. Then for any~$z_1\in (z_0, 0)$ and $\eps\in (0, s-\mez-\frac{d}{2})$, there exists an increasing function $\cF$ depending only on $s, \sigma, z_0, \eps$ such that
\begin{multline}\label{est:elliptic}
\lA\nabla_{x,z}v\rA_{X^\sigma([z_1 ,0])}\leq\cF(\lA\eta\rA_{\Cs^{1+\eps}})\lB \Vert f\Vert_{H^{\sigma+1}}+\Vert \eta\Vert_{H^{s+\mez}}\Vert\nabla_{x,z}v\Vert_{L^\infty([z_0, 0];L^\infty)}\right.\\
\left.+\Vert\nabla_{x,z}v\Vert_{X^{-\mez}([z_0, 0])}\rB.
\end{multline}
\end{prop}
\subsubsection{Besov estimates}
Our goal is to establish regularity results for $\nabla_{x,z}v$ in Besov spaces. In particular, we shall need such results in the Zygmund space with negative index $\Cs^{-\mez}$, which is one of the new technical issues compared to \cite{AM, ABZ1, ABZ3, ABZ4, WaZh}. To this end, we follow the general strategy in \cite{AM} by first paralinearizing equation \eqref{eq:v} and  then factorizing this second order elliptic operator into the product of a forward and a backward parabolic operator.  The study of $\nabla_{x,z}v$ in $\Cs^{-\mez}$ will make use of the maximum principle in Proposition \ref{max:prin}. The proof of the next lemma is straightforward.
\begin{lemm}\label{lemm:factor}
Set 
\bq\label{defi:R1}
R_1v=(\alpha-T_\alpha)\Delta_xv+(\beta-T_\beta)\cdot\nabla\partial_zv-(\gamma-T_\gamma)\partial_z v,\quad R_2v=T_\gamma\partial_z v.
\eq
Consider two symbols 
\begin{equation}\label{aA:1}
	\begin{aligned}
		a^{(1)}&=\mez\big(-i\beta\cdot\xi-\sqrt{4\alpha\la\xi\ra^2-(\beta\cdot\xi)^2}\big),\\
		A^{(1)}&=\mez\big(-i\beta\cdot\xi+\sqrt{4\alpha\la\xi\ra^2-(\beta\cdot\xi)^2}\big),
	\end{aligned}
\end{equation}
which satisfy $a+A=-i\beta\cdot\xi$, $aA=-\alpha |\xi|^2$. Next, set 
\bq\label{defi:R2}
R_3=-\big(T_{a^{(1)}}T_{A^{(1)}}-T_{\alpha}\Delta\big)+T_{\partial_zA^{(1)}}.
\eq
Then we have
\[
\mathcal{L}v=(\partial_z-T_{a^{(1)}})(\partial_z-T_{A^{(1)}})v+R_1v+R_2v+R_3v.
\]
\end{lemm}
The next proposition provides a regularity bootstrap for $\nabla_{x,z}v$ in $\Bl^r$ with $r\ge 0$. Its proof is inspired by that of Proposition~$4.9$ in \cite{WaZh}.
\begin{prop}%[see \protect{\cite[Proposition~$4.9$ ]{WaZh}}]
\label{prop:est:v:C1:apriori}
Let $\eps_0>0$ and $r\in [0, 1+\eps_0)$. Assume that $\eta\in \Cs^{2+\eps_0}\cap L^2$, $f\in H^\mez$,  $\nabla f\in \Bl^r$ and for some $z_0\in (-1, 0)$
\bq\label{assume:C1}
\nabla_{x,z}v\in \wL^\infty([z_0, 0]; \Bl^{r-\mez})\cap \wL^\infty([z_0, 0]; \Bl^0) 
\eq
Then, for any $z_1\in (z_0, 0)$,  we have $\nabla_{x,z}v\in C([z_1, 0]; \Bl^r)$ and
\bq\label{est:v:C1:apriori}
\Vert \nabla_{x,z}v\Vert_{C([z_1, 0]; \Bl^r)}\le_{K_{\eta, \eps_0}}\Vert \nabla f\Vert_{\Bl^r}+E(\eta, f),
\eq
 where,  $K_{\eta, \eps_0}$ is a constant of the form 
\bq\label{defi:Keps}
\cF\big(\Vert \eta\Vert_{\Cs^{2+\eps_0}}+\Vert \eta\Vert_{L^2}\big)
\eq
with $\cF:\xR^+\to \xR^+$ nondecreasing.
\end{prop}
\begin{rema}
It is important for later applications that our estimate involves only the Besov norm of $\nabla f$ and not $f$ itself.\\
Proposition \ref{prop:est:v:C1:apriori} is a conditional regularity result. It assumes  weaker regularities of $\nabla_{x,z}v$ to derive the regularity in $C([z_1, 0]; \Bl^r)$. The later will allow us to estimate the trace $\nabla_{x,z}v\rvert_{z=0}$ in the same space.
\end{rema}
\begin{proof}
Recall the definitions of $R_j$ $j=1, 2, 3$ in Lemma \ref{lemm:factor}. Pick $\eps>0$ such that $2\eps<\min\big\{\mez,  1+\eps_0-r\}$. We shall frequently use the following fact: for all $s\in \xR$ and for all $\delta>0$, there exists $C>0$ such that
\bq\label{CBC}
\frac{1}{C}\Vert u\Vert_{\Cs^s}\le\Vert u\Vert_{\Bl^s}\le C\Vert u\Vert_{\Cs^{s+\delta}}.
\eq
{\bf Step 1.} In this step, we estimate $R_jv$ in $L^2(J;\Bl^{r-\mez})$ for any $J\subset [-1, 0]$. For $R_1$ we write using the Bony decomposition 
\[
(\alpha-T_\alpha)\Delta_x v=T_{\Delta_xv}\alpha+R(\Delta_xv, \alpha).
\]
Applying \eqref{pest:CL} and the assumption \eqref{assume:C1} (i) gives
\[
\Vert T_{\Delta_xv}\alpha \Vert_{\wL^2\Bl^{r-\mez}}\les \Vert\alpha\Vert_{\wL^2 \Bl^{r+\mez+\eps}}\Vert \Delta_xv\Vert_{\wL^\infty \Cs^{-1-\eps}}\les K_{\eta, \eps_0}\Vert \nabla_xv\Vert_{\wL^\infty\Bl^{-\eps}},
\]
where we have used the facts that $r+\mez+2\eps\le \tdm+\eps_0$ and (by \eqref{alpha:C:2})
\[
 \Vert\alpha\Vert_{\wL^2 \Bl^{r+\mez+\eps}}\les \Vert\alpha\Vert_{\wL^2 \Cs^{r+\mez+2\eps}} \les K_{\eta, \eps_0}.
\]
Next, noticing that $(\tdm+\eps_0)+(-1-\eps)>0$ and $\tdm+\eps_0-1-\eps\ge r-\mez$, we obtain by using  \eqref{Bony:CL}
\[
\Vert R(\Delta_xv,\alpha) \Vert_{\wL^2 \Bl^{r-\mez}}\les \Vert\alpha\Vert_{\wL^2\Cs^{\tdm+\eps_0}}\Vert \Delta_xv\Vert_{\wL^\infty \Bl^{-1-\eps}}\les K_{\eta,\eps_0}\Vert \nabla_xv\Vert_{\wL^\infty \Bl^{-\eps}}.
\]
The term $(\beta-T_\beta)\cdot \nabla_x\partial_zv$ can be treated in the same way.  Lastly, it holds that
\[
\Vert T_{\partial_zv}\gamma\Vert_{\wL^2 \Bl^{r-\mez}}\les \Vert\gamma\Vert_{\wL^2 \Bl^{r-\mez+\eps}}\Vert \partial_zv\Vert_{\wL^\infty \Cs^{-\eps}}\les K_{\eta,\eps_0}\Vert \partial_zv\Vert_{\wL^\infty \Bl^{-\eps}}
\]
and 
\[
\Vert R(\partial_zv, \gamma)\Vert_{\wL^2 \Bl^{r-\mez}}\les \Vert\gamma\Vert_{\wL^2 \Cs^{\mez+\eps_0}}\Vert \partial_zv\Vert_{\wL^\infty \Bl^{-\eps}}\les K_{\eta,\eps_0}\Vert \partial_zv\Vert_{L^\infty\Bl^{-\eps}}.
\]
Gathering the above estimates leads to 
\[
\Vert R_1v\Vert_{\wL^2(J; \Bl^{r-\mez})}\les K_{\eta, \eps_0}\Vert \nabla_{x,z}v\Vert_{\wL^\infty(J; \Bl^{-\eps})}.
\]
On the other hand, $R_2v$ satisfies (using \eqref{pest:CL})
\begin{align*}
\Vert R_2v\Vert_{\wL^2(J; \Bl^{r-\mez})}& =\Vert T_\gamma\partial_zv\Vert_{\wL^2(J; \Bl^{r-\mez})}\\
& \les \Vert\gamma\Vert_{\wL^2(J; L^\infty)}\Vert \partial_zv\Vert_{L^\infty(J; \Bl^{r-\mez})}\les K_{\eta,\eps_0}\Vert \partial_zv\Vert_{\wL^\infty(J; \Bl^{r-\mez})},
\end{align*}
which is finite due to the assumption \eqref{assume:C1}.\\
Next, noticing that (see Notation \ref{symbolnorm:z})
\[
\mathcal{M}^1_1(a^{(1)})+\mathcal{M}^1_1(A^{(1)})+\mathcal{M}^1_0(\partial_zA^{(1)})\les K_{\eta,\eps_0}
\]
we can apply Lemma \ref{sc:CL} to deduce that $R_3$ is of order $1$ and 
\[
\Vert R_3v\Vert_{\wL^2(J; \Bl^{r-\mez})}\le \Vert R_3v\Vert_{\wL^\infty(J; \Bl^{r-\mez})} \les K_{\eta, \eps_0} \Vert \nabla_xv\Vert_{\wL^\infty(J; \Bl^{r-\mez})}.
\]
In view of Lemma \ref{lemm:factor}, we have proved that 
\[
(\partial_z-T_{a^{(1)}})(\partial_z-T_{A^{(1)}})v=F
\]
with
\[
\Vert F\Vert_{\wL^2(J; \Bl^{r-\mez})} \les_{K_{\eta,\eps_0}} \Vert \nabla_{x,z}v\Vert_{\wL^\infty(J; \Bl^{r-\mez}\cap \Bl^{-\eps})}.
\]
{\bf Step 2.}  Fix $-1<z_0<z_1<0$ and introduce $\kappa$  a cut-off function satisfying $\kappa\arrowvert_{z<z_0}=0,~\kappa\arrowvert_{z>z_1}=1$. Setting $w=\kappa(z)(\partial_z-T_{A^{(1)}})v$, then 
\[
(\partial_z-T_{a^{(1)}})w=G:=\kappa(z)F+\kappa'(z)(\partial_z-T_{A^{(1)}})v.
\]
As $w\rvert_{z=z_0}=0$, applying Theorem \ref{parabolic:Hol} yields for  sufficiently large $\delta>0$ to be chosen, that $ w\in C([z_0, 0]; \Bl^r)$ and 
\begin{multline*}
\Vert w\Vert_{C([z_0, 0]; \Bl^{r})}\les \Vert \kappa(z)F\Vert_{\wL^2([z_0, 0]; \Bl^{r-\mez})}+\cr
\Vert\kappa'(z)(\partial_z-T_{A^{(1)}})v\Vert_{\wL^2([z_0, 0]; \Bl^{r-\mez})}
+\Vert w\Vert_{L^\infty([z_0, 0]; \Cs^{-r_0})}.
\end{multline*}
Choosing $r_0>\eps$ and using \eqref{CBC} we deduce 
\[
\Vert w\Vert_{C([z_0, 0]; \Bl^{r})}\les_{K_{\eta,\eps_0}} \Vert \nabla_{x,z}v\Vert_{\wL^\infty([z_0, 0]; \Bl^{r-\mez}\cap \Bl^{-\eps})}.
\]
Now, on $[z_1, 0]$, $v$ satisfies 
\[
(\partial_z-T_{A^{(1)}})\nabla_xv=\nabla w+T_{\nabla_x A^{(1)}}v, \quad\nabla_x v\rvert_{z=0}=\nabla f.
\]
After changing $z\mapsto -z$,  Theorem \ref{parabolic:Hol} gives for sufficiently large $\delta>0$
\bq\label{est:C0:1}
\begin{aligned}
\Vert \nabla_xv\Vert_{C([z_1, 0]; \Bl^r)}&\les_{K_{\eta,\eps_0}} \Vert \nabla f\Vert_{ \Bl^r}+\Vert \nabla w\Vert_{\wL^\infty([z_1, 0]; \Bl^{r-1})}\\
&\quad +\Vert T_{\nabla_x A^{(1)}}v\Vert_{\wL^\infty([z_1, 0]; \Bl^{r-1})}+\Vert \nabla_xv\Vert_{\wL^\infty([z_1, 0]; \Cs^{-\delta})}\\
&\les_{K_{\eta,\eps_0}}  \Vert \nabla f\Vert_{ \Bl^r}+\Vert \nabla_{x,z}v\Vert_{\wL^\infty([z_0, 0]; \Bl^{r-\mez+\eps}\cap \Bl^{-\eps})}.
\end{aligned}
\eq
Then, from the equation $\partial_z v=w+T_{A^{(1)}}v$ we see that $ \partial_zv\in C([z_1, 0];\Bl^r)$ with norm bounded by the right-hand side of \eqref{est:C0:1}. We split 
\[
\Vert \nabla_{x,z}v\Vert_{\wL^\infty([z_0, 0]; \Bl^{r-\mez}\cap \Bl^{-\eps})}
\]
into two norms, one is over $[z_0, z_1]$ and the other is over $[z_1, 0]$. The one over $[z_0, z_1]$ can be bounded by $\Vert f\Vert_{H^\mez}$ using the estimate \eqref{est:vari}. Indeed, the fluid domain corresponding to $[z_0, z_1]$ belongs to the interior of $\Omega_t$, where $\phi$ is analytic, and thus the result follows from the standard elliptic theory (see for instance the proof of Lemma 2.9, \cite{ABZ1}). On the other hand, by choosing a large $\delta>0$ and interpolating between $\Bl^{-\delta}$ and $\Bl^r$, the term 
\[
\Vert \nabla_{x,z}v\Vert_{\wL^\infty([z_1, 0];\Bl^{r-\mez}\cap \Bl^{-\eps})}
\]
 appearing on the right-hand side of \eqref{est:C0:1}, can be absorbed by $\Vert \nabla_{x,z}v\Vert_{\wL^\infty([z_1, 0]; \Bl^r)}$ on the left-hand side, leaving a term bounded by $\Vert \nabla_{x,z}v\Vert_{\wL^\infty([z_1, 0]; \Bl^{-\delta})}$. Finally, choosing $\delta>\frac d2+\mez$, we conclude by \eqref{CBC}, Sobolev's embedding and \eqref{est:base:Cx}-\eqref{est:base:Cz} that
\[
\Vert \nabla_{x,z}v\Vert_{\wL^\infty([z_1, 0]; \Bl^{-\delta})}\les \Vert \nabla_{x,z}v\Vert_{\wL^\infty([z_1, 0]; H^{-\mez})}\les K_{\eta, \eps_0}E(\eta, f).
\]
\end{proof}
\begin{coro}\label{est:v:C1}
Let $s>\tdm+\frac{d}{2}$, $\eps_0\in (0, s-\tdm-\frac{d}{2})$ and $r\in [0, 1+\eps_0)$. Assume that $\eta\in H^{s+\mez}$ and $f\in H^s$, $\nabla f\in \Bl^r$. Then for any $z\in (-1, 0)$,  we have $\nabla_{x,z}v\in C([z, 0]; \Bl^r)$ and
\[
\Vert \nabla_{x,z}v\Vert_{C([z, 0]; \Bl^r)}\les_{K_{\eta, \eps_0}}\Vert \nabla f\Vert_{\Bl^r}+E(\eta, f).
\]
\end{coro}
\begin{proof}
Under the assumptions on the Sobolev regularity of $\eta$ and $f$, we can apply Proposition \ref{elliptic:ABZ} in conjunction with \eqref{est:base:H} to get for any $z\in (-1, 0)$,
\[
\nabla_{x,z}v\in C([z, 0]; H^{s-1})\hookrightarrow  C([z, 0]; \Cs^{\mez+\eps_0})\hookrightarrow  C([z, 0]; \Bl^{\mez}).
\]
Notice that $\eta\in H^{s+\mez}\hookrightarrow \Cs^{2+\eps_0}$ and $\nabla f\in H^{s-1}\hookrightarrow \Bl^{\mez}$. Then the bootstrap provided by Proposition \ref{prop:est:v:C1:apriori} concludes the proof.
\end{proof}
Considering the case $r=-\mez$, we first establish an a priori estimate.
\begin{prop}\label{est:v:C-mez:a}
Assume that $\eta\in \Cs^{2+\eps_0}\cap L^2$ for some  $\eps_0>0$,   and $f\in L^\infty$, $\nabla f\in \Cs^{-\mez}$.  If $\nabla_{x,z}v\in C([z, 0]; \Cs^{-\mez})$ for some $z\in (-1, 0)$ then 
\bq\label{est:elliptic:Z}
\Vert \nabla_{x,z}v\Vert_{C([z, 0]; \Cs^{-\mez})}\les_{K_{\eta, \eps_0}} \Vert \nabla f\Vert_{\Cs^{-\mez}}+E(\eta, f).
\eq
\end{prop}
\begin{proof}
We follow the proof of Proposition \ref{est:v:C1}. The first step consists in estimating $R_jv$ in $\wL^2\Cs^{-1}$. Fix $0<\eps<\min\{\mez, \eps_0\}$. For $R_1v$, a typical term can be treated as 
\begin{align*}
\Vert (\alpha-T_\alpha)\Delta_xv\Vert_{\wL^2(J; \Cs^{-1})}&\les \Vert\alpha\Vert_{\wL^2(J; \Cs^{\tdm+\eps_0})}\Vert \Delta_xv\Vert_{\wL^\infty(J; \Cs^{-\tdm-\eps})}\\
&\les K_{\eta, \eps_0} \Vert \nabla_xv\Vert_{\wL^\infty(J; \Cs^{-\mez-\eps})}.
\end{align*}
On the other hand, $R_2v$ satisfies
\[
\Vert R_2v\Vert_{\wL^2(J; \Cs^{-1})} \les \Vert\gamma\Vert_{\wL^2(J; \Cs^{\mez+\eps})}\Vert \partial_zv\Vert_{\wL^\infty(J; \Cs^{-\mez+\eps})}\les K_{\eta, \eps_0}\Vert \partial_zv\Vert_{\wL^\infty(J; \Cs^{-\mez-\eps})}.
\]
Since $R_3$ is of order $1$ with norm bounded by $K_{\eta, \eps_0}$, it holds that 
\[
\Vert R_3v\Vert_{\wL^2(J; \Cs^{-1})} \les K_{\eta, \eps_0}\Vert \nabla_xv\Vert_{\wL^\infty(J; \Cs^{-1})}.
\]
Consequently, we obtain 
\[
(\partial_z-T_{a^{(1)}})(\partial_z-T_{A^{(1)}})v=F
\]
with 
\[
\Vert F\Vert_{\wL^2(J; \Cs^{-1})} \les K_{\eta, \eps_0}\Vert \nabla_{x,z}v\Vert_{\wL^\infty(J; \Cs^{-\mez-\eps})}.
\]
Now, arguing as in the proof of Proposition \ref{est:v:C1}, one concludes the proof by applying twice Theorem \ref{parabolic:Hol}, then interpolating $\Vert\nabla_{x,z}v\Vert_{\wL^\infty \Cs^{-\mez-\eps}}$ between $\Vert\nabla_{x,z}v\Vert_{\wL^\infty \Cs^{-\mez}}$ and $\Vert\nabla_{x,z}v\Vert_{\wL^\infty \Cs^{-\delta}}$ with large $\delta>0$, where the later can be controlled by $E(\eta; f)$ via Sobolev's embedding.
\end{proof}
Next, we prove a regularity result, assuming $1/2$ more derivative of $\eta$.
\begin{prop}\label{est:v:C-mez}
Assume that $\eta\in \Cs^{\frac{5}{2}+\eps_0}\cap L^2$ for some  $\eps_0>0$,   and $f\in L^\infty\cap H^\mez$, $\nabla f\in \Cs^{-\mez}$.  Then, for any $z\in (-1, 0)$ we have $\nabla_{x,z}v\in C([z, 0]; \Cs^{-\mez})$ and 
\bq\label{est:elliptic:Z}
\Vert \nabla_{x,z}v\Vert_{C([z, 0]; \Cs^{-\mez})}\le K_{\eta, \eps_0}\lB \Vert \nabla f\Vert_{\Cs^{-\mez}}+\big(1+\Vert \eta\Vert_{\Cs^{\frac{5}{2}+\eps_0}}\big)\Vert f\Vert_{L^\infty}\rB.
\eq
\end{prop}
\begin{proof}
We still follow the proof of Proposition \ref{est:v:C1}. The first step consists in estimating $R_jv$ in $\wL^2\Cs^{-1}$. For $R_1v$, a typical term can be treated as 
\begin{align*}
\Vert (\alpha-T_\alpha)\Delta_xv\Vert_{\wL^2(J; \Cs^{-1})}&\les \Vert\alpha\Vert_{\wL^2(J; \Cs^{2+\eps_0})}\Vert \Delta_xv\Vert_{\wL^\infty(J; \Cs^{-2})}\\
&\les K_{\eta, \eps_0} \big(1+\Vert \eta\Vert_{\Cs^{\frac{5}{2}+\eps_0}}\big)\Vert \nabla_xv\Vert_{\wL^\infty(J; \Cs^{-1})}.
\end{align*}
On the other hand, $R_2v$ satisfies
\[
\Vert R_2v\Vert_{\wL^2(J; \Cs^{-1})} \les \Vert\gamma\Vert_{\wL^2(J; L^\infty)}\Vert \partial_zv\Vert_{\wL^\infty(J; \Cs^{-1})}\les K_{\eta, \eps_0}\Vert \partial_zv\Vert_{\wL^\infty(J; \Cs^{-1})}.
\]
Since $R_3$ is of order $1$ with norm bounded by $K_{\eta, \eps_0}$, it holds that 
\[
\Vert R_3v\Vert_{\wL^2(J; \Cs^{-1})} \les K_{\eta, \eps_0}\Vert \nabla_xv\Vert_{\wL^\infty(J; \Cs^{-1})}.
\]
Consequently, we obtain 
\[
(\partial_z-T_{a^{(1)}})(\partial_z-T_{A^{(1)}})v=F
\]
with 
\[
\Vert F\Vert_{\wL^2(J; \Cs^{-1})} \les K_{\eta, \eps_0}\big(1+\Vert \eta\Vert_{\Cs^{\frac{5}{2}+\eps_0}}\big)\Vert \nabla_{x,z}v\Vert_{\wL^\infty(J; \Cs^{-1})}.
\]
Then, arguing as in the proof of Proposition \ref{est:v:C1}, one concludes the proof by applying twice Theorem \ref{parabolic:Hol}: once with  $q=2$, $\delta \gg 1$  and once with $q=1$ and $\delta=1$ so that  Proposition \ref{max:prin} can be invoked to have 
\[
\Vert \nabla_{x,z}v\Vert_{\wL^\infty(J; \Cs^{-1})}\les_{K_{\eta,\eps_0}}\Vert f\Vert_{L^\infty}.
\]
\end{proof} 
%%%%%%%%%%%%%%%%%%
\subsection{Estimates for the Dirichlet-Neumann operator}
 We now apply the elliptic estimates in the previous subsection to study the continuity  of the Dirichlet-Neumann operator. Put  
\[
\zeta_1=\frac{1+\la\nabla_x\rho\ra^2}{\partial_z\rho},\quad \zeta_2=\nabla_x\rho.
\] 
By the definition \eqref{def:DN}, the Dirichlet-Neumann operator is given by
	\bq\label{formula:DN}
\begin{aligned}
G(\eta)f&=\zeta_1\partial_zv-\zeta_2\cdot\nabla_xv\big\rvert_{z=0}\\
&=h^{-1}\partial_z v+(\zeta_1-h^{-1})\partial_zv-\zeta_2\cdot\nabla_xv\big\rvert_{z=0},
\end{aligned}
\eq
where $v$ is the solution to \eqref{elliptic:prob}.
\begin{prop}	\label{est:DN:Sob}
		Let $s>\tmez+\frac d2$, $\eta\in H^{s+\mez}$ and  $f\in H^s$. Then we have
\bq\label{est:DN}
\lA G(\eta)f\rA_{H^{s-1}}\leq_{K_{\eta, \eps_0}}\Vert f\Vert_{H^s}+\Vert \eta\Vert_{H^{s+\mez}}\lB\Vert \nabla f\Vert_{\Bl^{0}}+E(\eta, f)\rB.
\eq
\end{prop}
\begin{proof}
Notice first that by the Sobolev embedding, $\eta\in \Cs^{2+\eps_0}$. Using the formula \eqref{formula:DN} and the tame estimate \eqref{tame:H} we obtain 
\[
\Vert G(\eta)f\Vert_{H^{s-1}}\les K_{\eta,\eps_0}\Vert \nabla_{x,z}v\rvert_{z=0}\Vert_{H^{s-1}}+\Vert\eta\Vert_{H^s}\Vert\nabla_{x,z}v\rvert_{z=0}\Vert_{L^\infty}.\]
Under the hypotheses, Corollary \ref{est:v:C1} is applicable with $r=0$. Hence, in view of \eqref{CBC}, it holds that
\[
\forall z\in (-1, 0),\quad \Vert \nabla_{x,z} v \Vert_{ C^0([z, 0]; \Bl^0)}\les_{K_{\eta, \eps_0}}\Vert \nabla f\Vert_{\Bl^{0}}+E(\eta, f).
\]
Noticing embedding $\Bl^0\hookrightarrow L^\infty$, we deduce
\[
\Vert G(\eta)f\Vert_{H^{s-1}}\les_{K_{\eta, \eps_0}} \Vert f\Vert_{H^s}+\Vert \eta\Vert_{H^{s+\mez}}\lB\Vert \nabla f\Vert_{\Bl^{0}}+E(\eta, f)\rB,
\]
which is the desired estimate.
\end{proof}
\begin{prop}\label{prop:estDN:Z}
We have the following estimates for the Dirichlet-Neumann operator in Zygmund spaces.
\begin{enumerate}
\item[1.] Let $s>\tdm+\frac{d}{2}$, $\eps_0\in (0, s-\tdm-\frac{d}{2})$ and $r\in (0, 1+\eps_0)$. Assume that $\eta\in H^{s+\mez}$ and $f\in H^s$, $\nabla f\in \Bl^r$. Then we have 
\bq
\label{est:DN:C1}
\Vert G(\eta)f\Vert_{\Bl^r}\les_{K_{\eta, \eps_0}}\Vert \nabla f\Vert_{\Bl^r}+E(\eta, f),
\eq
where recall that $K_{\eta, \eps_0}$ is defined by \eqref{defi:Keps}.
\item[2.] Let $\eps_0>0$. Assume that $\eta\in \Cs^{\frac{5}{2}+\eps_0}$,  $f\in L^\infty\cap H^\mez$ and $\nabla f\in \Cs^{-\mez}$, then
\bq
\label{est:DN:C-mez} 
\Vert G(\eta)f\Vert_{\Cs^{-\mez}}\les_{K_{\eta, \eps_0}} \Vert \nabla f\Vert_{\Cs^{-\mez}}+\big(1+\Vert \eta\Vert_{\Cs^{\frac{5}{2}+\eps_0}}\big)\Vert f\Vert_{L^\infty}.
\eq
\item[3.] Let $\eps_0>0$. Assume that $\eta\in \Cs^{2+\eps_0}\cap H^{1+\frac{d}{2}+}$ and $f\in H^{\mez+\frac{d}{2}}$, $\nabla f\in \Cs^{-\mez}$, then 
\bq
\label{est:DN:C-mez:l} 
\Vert G(\eta)f\Vert_{\Cs^{-\mez}}\les_{K_{\eta, \eps_0}} \Vert \nabla f\Vert_{\Cs^{-\mez}}+E(\eta, f).
\eq
\end{enumerate}
\end{prop}
\begin{proof} We first notice that $\Vert \zeta_j\rvert_{z=0}\Vert_{\Cs^{1+\eps_0}}\les K_{\eta, \eps_0}$. \\
1. Using the Bony decomposition for the right-hand side of \eqref{formula:DN}, we see that \eqref{est:DN:C1} is a consequence of Corollary \ref{est:v:C1}, \eqref{pest:B}, \eqref{Bony:B} and the embedding $\Bl^0\hookrightarrow L^\infty$.\\
2. For \eqref{est:DN:C-mez} one applies the product rule \eqref{tame:H<0} and  Proposition \ref{est:v:C-mez}. \\
3. For \eqref{est:DN:C-mez:l}  we first remark that owing to Proposition \ref{elliptic:ABZ}, the assumptions $\eta\in H^{1+\frac d2+}$, $f\in H^{\mez+\frac{d}{2}}$ imply 
\[
z\in (-1, 0),\quad \nabla_{x,z}v\in C([z, 0]; H^{-\mez+\frac d2})\hookrightarrow C([z, 0]; \Cs^{-\mez}).
\]
Therefore, the a priori estimate of  Proposition \ref{est:v:C-mez:a} yields
\[
\Vert \nabla_{x,z} v\Vert_{C([z, 0]; \Cs^{-\mez})}\les_{K_{\eta, \eps_0}}\Vert f\Vert_{H^\mez} +E(\eta, f),
\]
which, combined with \eqref{tame:H<0}, concludes the proof.
\end{proof}
To conclude this section, let us recall the following result on the {\it shape derivative} of the Dirichlet-Neumann operator.
\begin{theo}[\protect{see \cite[Theorem~3.21]{LannesLivre}}]\label{shape}
Let  $s>\mez+\frac d2,~d\ge 1$ and $\psi\in H^{\tmez}$. Then the map 
\[
G(\cdot)\psi: H^{s+\mez}\to H^{\mez}
\]
is differentiable and  for any $f\in H^{s+\mez}$,
\[
d_\eta G(\eta)\psi\cdot f:=\lim_{\eps\to 0}\frac{1}{\eps}\big(G(\eta+\eps f)\psi-G(\eta)f\big)=-G(\eta)(Bf)-\div(Vf)
\]
where $B$ and $V$ are functions of $(\eta, \psi)$ as in \eqref{BV}.
\end{theo}
\section{Paralinearization and symmetrization of the system}\label{section:para}
%%%%%%%%%%%%%%%%%
Throughout this section,  we assume that $(\eta, \psi)$ is a solution to  \eqref{ww} on a time interval $I=[0, T]$ and
\bq\label{assu:reg}
\left\{
\begin{aligned}
& \eta\in L^\infty(I; H^{s+\mez})\cap L^1(I; \Cs^{\frac{5}{2}+\eps_*}),\\
&\psi\in L^\infty(I; H^s), \nabla_x\psi\in L^1(I; \Bl^1)\\
& s>\tdm+\frac{d}{2}, \eps_*>0\\
&\inf_{t\in I}\dist(\eta(t), \Gamma)\ge h>0.
\end{aligned}
\right.
\eq
%The Cauchy theory in Theorem $1.1$, \cite{ABZ1} ensures that such solutions exist provided 
%\[
%(\eta, \psi)\rvert_{t=0}\in H^{s+1}\times H^{s+\mez},\quad \dist(\eta\rvert_{t=0}\ge 2h.
%\]
We fix from now on
\[
0<\eps<\min\{\eps_*, \mez\}
\]
and define the quantities 
\bq
\begin{aligned}
&\A=\Vert \eta\Vert_{\Cs^{2+\eps_*}}+\Vert \eta\Vert_{L^2}+\Vert \nabla_x \psi\Vert_{\Bl^0}+E(\eta, \psi),\\
&\B=\Vert \eta\Vert_{\Cs^{\frac{5}{2}+\eps_*}}+\Vert \nabla_x\psi\Vert_{\Bl^1}+1.
\end{aligned}
\eq
%Since $\nabla\psi=B+V\nabla\eta$, it  holds that 
%\[
%\Vert \nabla\psi\Vert_{\Cs^{1+\eps_*}}\les (1+\A)\Vert(V, B)\Vert_{\Cs^{1+\eps_*}}.
%\]
 Our goal is to derive estimates for $(\eta, \psi)$ in $L^\infty(I; H^{s+\mez}\times H^s)$ by means of $\A$ and $\B$ and keep them {\it linear in $\B$}.
\subsection{Paralinearization of the Dirichlet-Neumann operator}
Our goal is to obtain error estimates for $G(\eta)\psi$ when expanding it in paradifferential operators. More precisely, as in Proposition $3.14$, \cite{ABZ1}, we will need such expansion in terms of the first two symbols defined by
\begin{equation} 
\begin{aligned}
&\lambda^{(1)}:=\sqrt{(1+\la\nabla\eta\ra^2)\la\xi\ra^2-(\nabla\eta\cdot\xi)^2},\\	
&\lambda^{(0)}:=\frac{1+\la\nabla\eta\ra^2}{2\lambda^{(1)}}\lb\dive(\alpha^{(1)}\nabla\eta)+i\partial_\xi\lambda^{(1)}\cdot\nabla\alpha^{(1)}\rb
\end{aligned}
\end{equation}
with 
\[
\alpha^{(1)}:=\frac{1}{1+\la\nabla\eta\ra^2}(\lambda^{(1)}+i\nabla\eta\cdot\xi).
\]
Set  $\lambda:=\lambda^{(1)}+\lambda^{(0)}$.\\
To study $G(\eta)\psi$, we reconsider  the elliptic problem \eqref{elliptic:prob}, {\it i.e.}, 
\bq\label{elliptic:DN}
\Delta_{x,y}\phi=0\text{ in }\Omega,\quad
\phi\arrowvert_{\Sigma}=\psi,\quad \partial_n \phi\arrowvert_{\Gamma}=0.
\eq
Let 
\[
v(x,z)=\phi(x, \rho(x,z)\quad (x,z)\in S=\xR^d\times (-1, 0)
\]
 as in Section \ref{section:elliptic}. Then, by \eqref{eq:v}, $v$ satisfies $\mathcal{L}v=0$ in $S$. Applying Proposition \ref{elliptic:Th} with $\sigma=s-1$ and Corollary \ref{est:v:C1} with $r=0$ we obtain for any $z\in (-1, 0)$
\bq\label{dv:H}
\begin{aligned}
\lA\nabla_{x,z}v\rA_{X^{s-1}([z ,0])}&\les_A \Vert \psi\Vert_{H^s}+\Vert \eta\Vert_{H^{s+\mez}}\lB\Vert \nabla \psi\Vert_{\Bl^0}+E(\eta, f)\rB,\\
&\les_\A\Vert \psi\Vert_{H^s}+\Vert \eta\Vert_{H^{s+\mez}}.
\end{aligned}
\eq
On the other hand, Corollary \ref{est:v:C1} with $r=1$ yields for any $z\in (-1, 0)$
\bq\label{dv:C1}
\Vert \nabla_{x,z}v\Vert_{C([z, 0]; \Bl^1)}\les_\A\Vert \nabla \psi\Vert_{\Bl^1}+E(\eta, f)\les_\A \B.
\eq 
\begin{lemm}\label{lemm:eq:v}
We have 
\[ 
\partial^2_zv+T_\alpha\Delta_xv+T_\beta\cdot\nabla_x\partial_zv-T_\gamma\partial_zv-T_{\partial_zv}\gamma=F_1,
\]
where, for all $I\Subset (-1, 0]$, $F_1$ satisfies
\[
\Vert F_1\Vert_{Y^{s+\mez}(I)}\les_\A\B\lB \Vert \eta\Vert_{H^{s+\mez}}+\Vert \psi\Vert_{H^{s}}\rB.
\]
\end{lemm}
\begin{proof}
From equation \eqref{eq:v} and the Bony decomposition, we see that 
\[
F_1=-R_1v=-(\alpha-T_\alpha)\Delta_xv-(\beta-T_\beta)\cdot\nabla\partial_zv+R(\gamma, \partial_z v).
\]
 Writing $(\alpha-T_\alpha)\Delta_xv=(\alpha-h^2-T_{\alpha-h^2})\Delta_xv+(h^2-T_{h^2})\Delta_x v$, we estimate using \eqref{dv:C1}
\begin{align*}
\Vert (\alpha-h^2-T_{\alpha-h^2})\Delta_xv\Vert_{L^2H^s}&\les \Vert T_{\Delta_x v}(\alpha-h^2)\Vert_{L^2 H^s}+\Vert R(T_{\Delta_x v}, \alpha-h^2)\Vert_{L^2H^s}\\
&\les \Vert \Delta_x v\Vert_{L^2L^\infty}\Vert (\alpha-h^2)\Vert_{L^2H^s}\\
&\les_\A \B\Vert \eta\Vert_{H^{s+\mez}}.
\end{align*}
Since $(h^2-T_{h^2})$ is a smoothing operator, there holds by Remark \ref{rema:E}
\[
\Vert (h^2-T_{h^2})\Delta_x v\Vert_{L^2H^s}\les \Vert \nabla_x v\Vert_{L^2L^2}\les (1+\Vert \eta\Vert_{W^{1,\infty}})\Vert \psi\Vert_{H^\mez}\les_A \Vert \psi\Vert_{H^s}
\]
The other terms of $F_1$ can be treated similarly.
\end{proof}
The next step consists in studying the paradifferential equation satisfied by the {\it good-unknown} (see \cite{AM} and the reference therein)
\[
u:=v-T_b\rho\quad\text{with}~b:=\frac{\partial_z v}{\partial_z\rho}.
\]
Notice that $b\rvert_{z=0}=B$. Estimates for $b$ is now provided.
\begin{lemm}\label{est:b}
For any $I\Subset (-1, 0]$, we have 
\begin{align}
\label{est:b:C0}
&\Vert b\Vert_{L^\infty(I; L^\infty)}\les_\A 1,\\
\label{est:db:C0}
&\Vert \nabla_{x,z}b\Vert_{L^\infty(I; L^\infty)}\les_\A\B,\\
\label{est:d2b:C-1}
&\Vert \nabla^2_{x,z}b\Vert_{L^\infty(I; \Cs^{-1})}\les_\A\B.
\end{align}
\end{lemm}
\begin{proof}
We first recall the lower bound \eqref{drho:lowerbound}
\bq\label{drho:lowerbound:1}
\partial_z\rho\ge \frac{h}{2}.
\eq
Observe that with respect to the $L^\infty$-norm in $z$, $\rho$ and $\eta$ have the same Zygmund regularity, hence 
\bq\label{est:drho}
\Vert \nabla_{x,z}^2\rho\Vert_{L^\infty(I; \Cs^{\eps_*})}+\Vert \nabla_{x,z}^3\rho\Vert_{L^\infty(I; \Cs^{-1+\eps_*})}\les_\A 1.
\eq

Next, applying Corollary \ref{est:v:C1} with $r=0$ yields
\bq\label{du:C0}
\Vert \nabla_{x,z}v\Vert_{C(I; \Bl^0)}\les_\A 1.
\eq
On the other hand, recall from \eqref{dv:C1} that 
\bq\label{du:C1:1}
\Vert \nabla_{x,z}v\Vert_{C(I; \Bl^1)}\les_\A\B.
\eq
Using equation \eqref{eq:v}, $\partial^2_zv$ can be expressed in terms of $(\alpha, \beta, \gamma)$ and $(\Delta_xv, \nabla_x\partial_zv, \partial_zv)$. It then follows from \eqref{du:C1:1}, \eqref{est:drho} and Lemma \ref{pr:B} that
\bq\label{d2u:C0}
\Vert \partial_z^2v\Vert_{C(I; \Bl^0)}\les_\A \B.
\eq
Let us now consider 
\[
\partial_z^3v=-\alpha\Delta_x\partial_zv-\partial_z\alpha\Delta_xv-\beta\cdot\nabla_x\partial_z^2v-\partial_z\beta\cdot\nabla_x\partial_zv+\gamma\partial_z^2v+\partial_z\gamma\partial_zv.
\]
We notice the following bounds 
\[
\Vert \partial_z\alpha\Vert_{C(I; \Cs^{\mez})}+\Vert \partial_z\beta\Vert_{C(I; \Cs^{\mez})}+\Vert \partial_z\gamma\Vert_{C(I; \Cs^{-\mez})}\les_\A 1,
\]
which can be proved along the same lines as the proof of \eqref{alpha:C:2}.\\ 
Then using the above estimates and \eqref{tame:H<0} one can derive
\bq\label{d3u:C-1}
\Vert \partial_z^3v\Vert_{C(I; \Cs^{-1})}\les_\A \B.
\eq
The estimates \eqref{est:b:C0}, \eqref{est:db:C0}, \eqref{est:d2b:C-1} are consequences of the above estimates and the Leibniz rule.
\end{proof}
\begin{lemm}\label{paralin:eq:u}
We have  
\[
Pu:=\partial^2_zu+T_\alpha\Delta_xu+T_\beta\cdot\nabla_x\partial_zu-T_\gamma\partial_zu=F_2,
\]
 where, for all $I\Subset (-1, 0]$, $F_2$ satisfies
\[
\Vert F_2\Vert_{L^2(I; H^s)}\les_\A\B \lB\Vert \eta\Vert_{H^{s+\mez}}+\Vert \psi\Vert_{H^{s}}\rB.
\]
\end{lemm}
\begin{rema}
Compared with the equation satisfied by $v$ in Lemma \ref{lemm:eq:v}, the introduction of the good-unknown $u$ helps eliminate the bad term $T_{\partial_zv}\gamma$, which is not controlled in $L^2H^s$. 
\end{rema}
\begin{proof}
We will write $A\sim B$ if 
\[
\Vert A-B\Vert_{L^2(I; H^s)}\les_\A\B \lB\Vert \eta\Vert_{H^{s+\mez}}+\Vert \psi\Vert_{H^{s}}\rB.
\]
From Lemma \ref{lemm:eq:v}, we see that
\bq\label{Pu:1}
Pu=Pv-PT_b\rho=T_{\partial_zv}\gamma-PT_b\rho+F_1
\eq
and $F_1\sim 0$. Therefore, it suffices to prove that $PT_b\rho\sim T_\gamma\partial_zv$.\\
In the expression of $PT_b\rho$, we observe that owing to Lemma \ref{est:b}, all the terms containing $\rho$ and $\nabla_{x,z}\rho$ are $\sim 0$, hence
\bq\label{PTrho}
PT_b\rho\sim T_b\partial^2_z\rho+T_\alpha T_b\Delta \rho+T_\beta\cdot T_b\nabla \partial_z\rho.
\eq
Next, we find an elliptic equation satisfied by $\rho$. Remark that $w(x,y):=y$ is  a harmonic function in $\Omega$. Then, under the change of variables $(x,z)\mapsto (x, \rho(x,z))$, $(x,z)\in S=\xR^d\times (-1, 0)$,
\[
\widetilde w(x,z):=w(x,\rho(x,z))=\rho(x,z)
\]
satisfies \[
\mathcal{L}\rho=(\partial_z^2+\alpha\Delta_x+\beta\cdot\nabla_x\partial_z-\gamma\partial_z)\rho=0.
\] 
Then, by paralinearizing as in Lemma \ref{lemm:eq:v} we obtain 
\[
\partial^2_z\rho+T_\alpha\Delta_x\rho+T_\beta\cdot\nabla_x\partial_z\rho-T_{\partial_z\rho}\gamma\sim 0,
\]
where we have used the fact that $T_\gamma \partial_z\rho\sim 0$. Consequently,
\[
T_b\partial^2_z\rho+T_bT_\alpha\Delta_x\rho+T_bT_\beta\cdot\nabla_x\partial_z\rho-T_bT_{\partial_z\rho}\gamma\sim 0.
\]
Comparing with \eqref{PTrho} leads to
\[
PT_b\rho\sim [T_\alpha, T_b]\Delta \rho+[T_\beta,T_b]\nabla \partial_z\rho+T_bT_{\partial_z\rho}\gamma.
\]
By Lemma \ref{est:b}, it is easy to check that $[T_\alpha, T_b]$ is of order $-1$ and
\[
\Vert [T_\alpha, T_b]\Delta \rho\Vert_{L^2H^s}\les_\A \B \Vert \Delta_x\rho\Vert_{L^2H^{s-1}}\les_\A \B \Vert \eta\Vert_{H^{s+\mez}}.
\]
In other words, $[T_\alpha, T_b]\Delta \rho\sim 0$. By the same argument, we get $[T_\beta,T_b]\nabla \partial_z\rho\sim 0$. Finally, since 
\[
T_bT_{\partial_z\rho}\gamma\sim T_{b\partial_z\rho}\gamma=T_{\partial_zv}\gamma
\]
we conclude that $PT_b\rho\sim T_\gamma\partial_zv$.
\end{proof}
Next, in the spirit of Lemma \ref{lemm:factor}, we factorize $P$ into two parabolic operators.
\begin{lemm}\label{factor2}
Define 
	\begin{equation*}
	\begin{aligned}
		a^{(0)}&=\frac1{A^{(1)}-a^{(1)}}\lp i\partial_\xi a^{(1)}\partial_xA^{(1)}-\gamma a^{(1)}\rp,\\
		A^{(0)}&=\frac1{a^{(1)}-A^{(1)}}\lp i\partial_\xi a^{(1)}\partial_xA^{(1)}-\gamma A^{(1)}\rp
	\end{aligned}
	\end{equation*}
so that 
\bq\label{rela:aA}
a^{(1)}+A^{(1)}=-i\beta \cdot \xi, \quad a^{(1)}A^{(1)}=-\alpha|\xi|^2.
\eq
Set $a=a^{(1)}+a^{(0)}$, $A=A^{(1)}+A^{(0)}$ and $R=T_aT_A-T_{\alpha}\Delta$. Then we have 
\[
P=(\partial_z-T_a)(\partial_z-T_A)+R
\]
and for any $I\Subset (-1, 0]$,
\[
\Vert Ru\Vert_{L^2(I; H^s)}\les_\A\B \lB\Vert \eta\Vert_{H^{s+\mez}}+\Vert \psi\Vert_{H^{s}}\rB.
\]
\end{lemm}
\begin{proof}
From the definitions of $a, A$, we can check that
\bq\label{relation:aA}
\begin{aligned}
&a^{(1)}A^{(1)}+\frac1i\partial_\xi a^{(1)}\cdot\partial_xA^{(1)}+a^{(1)}A^{(0)}+a^{(0)}A^{(1)}=-\alpha\la\xi\ra^2,\\
& a+A=-i\beta\cdot\xi+\gamma.
\end{aligned}
\eq
A direct computation shows that
\[
R=\lp T_aT_A-T_\alpha\Delta\rp+\lp(T_a+T_A)+(T_\beta\cdot\nabla-T_\gamma)\rp\partial_z=T_aT_A-T_\alpha\Delta
\]
by the second equation of \eqref{relation:aA}. Now, we write
\[
T_aT_A=T_{a^{(1)}}T_{A^{(1)}}+T_{a^{(1)}}T_{A^{(0)}}+T_{a^{(0)}}T_{A^{(1)}}+T_{a^{(0)}}T_{A^{(0)}}.
\]
We have the following bounds
\begin{align*}
&\cM^1_{\tdm}(a^{(1)})+\cM^1_{\tdm}(A^{(1)})\les \cF(\Vert \eta\Vert_{\Cs^{2+\eps}})(1+\Vert \eta\Vert_{\Cs^{\frac 52}}),\\
&\cM^1_{\mez}(a^{(1)})+\cM^1_{\mez}(A^{(1)})\les \cF(\Vert\eta\Vert_{\Cs^{2+\eps}}),\\
&\cM^0_{\mez}(a^{(0)})+\cM^0_{\mez}(A^{(0)})\les\cF(\Vert \eta\Vert_{\Cs^{2+\eps}})(1+\Vert \eta\Vert_{\Cs^{\frac 52}}),\\
&\cM^0_{0}(a^{(0)})+\cM^0_0(A^{(0)})\les \cF(\Vert\eta\Vert_{\Cs^{2+\eps}}).
\end{align*}
Then, applying Theorem \ref{theo:sc} (ii) we obtain
\bq\label{TaTA}
\begin{aligned}
&\Vert T_{a^{(0)}}T_{A^{(0)}}-T_{a^{(0)}A^{(0)}}\Vert_{H^{\mu-\mez}\to H^{\mu}}\les \Xi,\\
&\Vert T_{a^{(0)}}T_{A^{(1)}}-T_{a^{(0)}A^{(1)}}\Vert_{H^{\mu+\mez}\to H^{\mu}}\les \Xi,\\
&\Vert T_{a^{(1)}}T_{A^{(0)}}-T_{a^{(1)}A^{(0)}}\Vert_{H^{\mu+\mez}\to H^{\mu}}\les\Xi,\\
&\Vert T_{a^{(1)}}T_{A^{(1)}}-T_{a^{(1)}A^{(0)}}-T_{\frac1i\partial_\xi a^{(1)}\cdot\partial_xA^{(1)}}\Vert_{H^{\mu+\mez}\to H^{\mu}}\les\Xi,
\end{aligned}
\eq
where $\Xi$ denotes any constant of the form
\[
\cF(\Vert \eta\Vert_{\Cs^{2+\eps}})(1+\Vert \eta\Vert_{\Cs^{\frac 52}}).
\]
Therefore, the first equation of \eqref{relation:aA} implies 
\[
\Vert Ru\Vert_{L^2H^s}\les_\A \B\Vert\nabla_x u\Vert_{L^2H^{s-\mez}}
\]
where, we have replaced $\|u\|_{L^2H^{s+\mez}}$ by $\Vert\nabla_x u\Vert_{L^2H^{s-\mez}}$ according to Remark \ref{rema:low}. Finally, writing $\nabla_xu=\nabla_xv-T_{\nabla_xb}\rho-T_b\nabla_x\rho$ we conclude by means of \eqref{dv:H} and \eqref{est:b:C0} that
\bq
 \Vert\nabla_x u\Vert_{L^2H^{s-\mez}}\les_\A\B \lB  \Vert \psi\Vert_{H^s}+\Vert \eta\Vert_{H^{s+\mez}}\rB.
\eq
\end{proof}
\begin{prop}\label{prop:paralin:DN}
It holds that 
\[
G(\eta)\psi= T_\lambda (\psi-T_B\eta)+T_V\cdot\nabla \eta+F
\]
with $F$ satisfying
\[
\Vert F\Vert_{H^{s+\mez}}\les_\A\B \lB  \Vert \psi\Vert_{H^s}+\Vert \eta\Vert_{H^{s+\mez}}\rB.
\]
\end{prop}
\begin{proof} A combination of Lemma \ref{paralin:eq:u} and Lemma \ref{factor2} yields
\[
(\partial_z-T_a)(\partial_z-T_A)u=F_2,
\]
where, $F_2$ satisfies for all $I\Subset (-1, 0]$,
\bq\label{paralin:gain:0}
\Vert F_2\Vert_{L^2(I; H^s)}\les_\A \B\lB  \Vert \psi\Vert_{H^s}+\Vert \eta\Vert_{H^{s+\mez}}\rB.
\eq
The proof proceeds in two steps.\\
{\bf Step 1.} As in the proof of Proposition \ref{est:v:C1:apriori}, we fix $-1<z_0<z_1<0$ and introduce $\kappa$  a cut-off function satisfying $\kappa\arrowvert_{z<z_0}=0,~\kappa\arrowvert_{z>z_1}=1$ . Setting $w=\kappa(z)(\partial_z-T_A)u$, then 
\[
(\partial_z-T_a)w=G:=\kappa(z)F_2u+\kappa'(z)(\partial_z-T_A)u.
\]
We now bound $G$ in $L^2([z_0, 0]; H^s)$. First, it follows directly from \eqref{paralin:gain:0} that
\bq\label{paralin:gain:1}
\Vert \kappa(z)F_2u\Vert_{Y^{s+\mez}([z_0, 0])}\les \Vert \kappa(z)Ru\Vert_{L^2([z_0, 0]; H^s)} \les_\A\B \lB  \Vert \psi\Vert_{H^s}+\Vert \eta\Vert_{H^{s+\mez}}\rB=:\Pi.
\eq
Next, notice that $p:=\kappa'(z)(\partial_z-T_A)u$ is non vanishing only for $z\in I:=[z_0, z_1]$. In the light of Lemma \ref{est:b},
\[
\Vert \nabla_{x,z}u\Vert_{L^2(I; H^s)}\les_\A \B\Vert\eta\Vert_{H^{s+\mez}}+\Vert \nabla_{x,z}v\Vert_{L^2(I;H^s)}.
\]
Hence
\[
\Vert (\partial_z-T_A)u\Vert_{L^2(I; H^s)}\les_{\A} \Vert \nabla_{x,z}u\Vert_{L^2(I; H^s)}\les_{\A} \B\Vert\eta\Vert_{H^{s+\mez}}+\Vert \nabla_{x,z}v\Vert_{L^2(I;H^s)}.
\]
 The fluid domain corresponds to $[z_0, z_1]$ is a strip lying in the interior of $\Omega_h$, where the harmonic function $\phi$ is smooth by the standard elliptic theory. In particular, there holds (see for instance the proof of Lemma 2.9, \cite{ABZ1})
\[
\Vert \nabla_{x,z}v\Vert_{L^2(I;H^s)}\les_A\Vert \psi\Vert_{H^\mez}.
\]
Therefore, we can estimate
\begin{align*}
\Vert p\Vert_{L^2([z_0, 0]; H^s)}&\les \Vert (\partial_z-T_A)u\Vert_{L^2([z_0, z_1]; H^s)}\\
&\les_\A \B\Vert\eta\Vert_{H^{s+\mez}}+\Vert \nabla_{x,z}v\Vert_{L^2([z_0, z_1];H^s)},\\
&\les_\A \B\lB \Vert\eta\Vert_{H^{s+\mez}}+\Vert \psi\Vert_{H^\mez}\rB.
\end{align*}
This, combined with \eqref{paralin:gain:1} , yields
\bq\label{claim1}
\Vert G\Vert_{Y^{s+\mez}([z_0, 0])}\les_\A\B \lB  \Vert \psi\Vert_{H^s}+\Vert \eta\Vert_{H^{s+\mez}}\rB=:\Pi.
\eq
Consequently, as $w\rvert_{z=z_0}=0$,  we can apply Theorem \ref{parabolic:Sob} to have 
$\Vert w\Vert_{X^{s+\mez}([z_0, 0])}\les\Pi$, which implies
\bq\label{est:paralin:gain}
\Vert \partial_zu-T_Au\Vert_{X^{s+\mez}([z_1, 0])}\les\Pi.
\eq
{\bf Step 2.}  We will write ~$f_1\sim f_2$ provided $\lA f_1-f_2\rA_{X^{s+\mez}([z_1,0])}\leq \Pi$. By paralinearizing (using the Bony decomposition and Theorem \ref{paralin}) we have
\begin{multline*}
	\frac{1+\la\nabla\rho\ra^2}{\partial_z\rho}\partial_zv-\nabla\rho\cdot\nabla v\sim T_{\frac{1+\la\nabla\rho\ra^2}{\partial_z\rho}}\partial_zv+2T_{b\nabla\rho}\cdot\nabla\rho-T_{b\frac{1+\la\nabla\rho\ra^2}{\partial_z\rho}}\partial_z\rho-T_{\nabla\rho}\cdot\nabla v-T_{\nabla v}\cdot\nabla\rho.
\end{multline*}
Then replacing~$v$ with~$u+T_b\rho$ we obtain, after some computations, that
$$\frac{1+\la\nabla\rho\ra^2}{\partial_z\rho}\partial_zv-\nabla\rho\cdot\nabla v\sim T_{\frac{1+\la\nabla\rho\ra^2}{\partial_z\rho}}\partial_zu-T_{\nabla\rho}\cdot\nabla u+T_{b\nabla\rho-\nabla v}\cdot\nabla\rho.$$
Now, using \eqref{est:paralin:gain} allows us to replace the normal derivative $\partial_zv$ with the "tangential derivative" $T_Av$, leaving a remainder which is $\sim 0$. Therefore,
$$
T_{\frac{1+\la\nabla\rho\ra^2}{\partial_z\rho}}\partial_zu-T_{\nabla\rho}\cdot\nabla u\sim T_\Lambda u+T_{b\nabla\rho-\nabla v}\cdot\nabla\rho
$$
with
\[
\Lambda:=\frac{1+\la\nabla\rho\ra^2}{\partial_z\rho}A-i\nabla\rho\cdot\xi.
\]
One can check that $\Lambda\rvert_{z=0}=\lambda=\lambda^{(1)}+\lambda^{(0)}$ as announced. On the other hand, at $z=0$, 
\[
b\nabla\rho-\nabla v=B\nabla\eta-\nabla \psi=V,\quad u=\psi-T_B\eta.
\]
In conclusion, we have proved that 
\[
G(\eta)\psi\sim T_\lambda (\psi-T_B\eta)+T_V\cdot\nabla \eta.
\]
\end{proof}
\subsection{Paralinearization of the full system}
\begin{lemm}\label{lem:paracurv}
	There exists a nondecreasing function~$\cF$ such that
	$$H(\eta)=T_{\l}\eta+f,$$
	where~$\l=\l^{(2)}+\l^{(1)}$ with
	\begin{equation}	\label{eq:l}
	\begin{aligned}
\l^{(2)}&=\lp1+\la\nabla\eta\ra^2\rp^{-\mez}\lp\la\xi\ra^2-\frac{(\nabla\eta\cdot\xi)^2}{1+\la\nabla\eta\ra^2}\rp,\quad
\l^{(1)}&=-\frac i2(\partial_x\cdot\partial_\xi)\l^{(2)},
	\end{aligned}
	\end{equation}
	and~$f\in H^{s}$ satisfying
	$$\lA f\rA_{H^{s}}\leq\cF\lp\lA \eta\rA_{W^{1,\infty}}\rp\lA \eta\rA_{\Cs^\frac{5}{2}}\lA \nabla\eta\rA_{H^{s-\mez}}.$$
\end{lemm}
\begin{proof}
We first apply Theorem \ref{paralin} with  $u=\nabla \eta$, $\mu=s-\mez$ and $\rho= \tdm$  to have
\[
\frac{\nabla\eta}{\sqrt{1+|\nabla\eta|^2}}=T_{p}\nabla\eta+f_1, \quad p=\frac{1}{(1+|\nabla\eta|^2)^\mez}I-\frac{\nabla \eta\otimes\nabla\eta}{(1+|\nabla\eta|^2)^{\tdm}}
\]
with $f_1$ satisfying
\[
\lA f_1\rA_{H^{s-\mez+\mez}}\le \cF\lp\lA \nabla\eta\rA_{L^{\infty}}\rp\lA \nabla\eta\rA_{\Cs^{\tdm}}\lA \nabla\eta\rA_{H^{s-\mez}}.
\]
Hence,
\[
H(\eta)=-\cnx(T_{p}\nabla\eta+f_1)=T_{p\xi\cdot\xi-i\cnx p\xi}\eta-\cnx f_1.
\]
This gives the conclusion with $l^{(2)}=p\xi\cdot\xi,~l^{(1)}=-i\cnx p\xi,~ f=-\cnx f_1.$
\end{proof}
We next paralinearize the other nonlinear terms.
Recall the notations
$$B=\frac{\nabla\eta\cdot\nabla\psi+G(\eta)\psi}{1+\la\nabla\eta\ra^2},\quad V=\nabla\psi-B\nabla\eta.$$
For later estimates on $B$, we write 
\bq \label{decompose:B} 
\begin{aligned}B&=\frac{\nabla\eta}{1+\la\nabla\eta\ra^2}\cdot\nabla\psi+\frac1{1+\la\nabla\eta\ra^2}G(\eta)\psi\\
&=:K(\nabla\eta)\cdot\nabla\psi+L(\nabla\eta)G(\eta)\psi+G(\eta)\psi,
\end{aligned}
\eq
where~$K$ and~$L$ are smooth function in $L^\infty(\xR^d)$ and satisfy $K(0)=L(0)=0$. From this expression and the Bony decomposition, one can easily prove the following.
\begin{lemm}\label{est:VB}
We have 
\begin{align}
&\Vert (V, B)\Vert_{\Bl^1}\les_\A \B\label{est:VB:C1},\\
&\Vert (V, B)\Vert_{L^\infty}\les_\A 1. \label{est:VB:C0}
\end{align}
\end{lemm}
\begin{lemm}	\label{lem:paraquad}
	We have
	$$\mez\la\nabla\psi\ra^2-\mez\frac{(\nabla\eta\cdot\nabla\psi+G(\eta)\psi)^2}{1+\la\nabla\eta\ra^2}=
	T_V\cdot\nabla\psi-T_VT_B\cdot\nabla\eta-T_BG(\eta)\psi+f,$$
	with~$f\in H^{s}$ and
	$$\lA f\rA_{H^{s}}\les_\A \B \lB\Vert \eta\Vert_{H^{s}}+\Vert \psi\Vert_{H^s}\rB.$$
\end{lemm}
\begin{proof}
	Consider
$$
F(a, b, c)=\mez\frac{(ab+c)^2}{1+\la a\ra^2}, \quad (a, b, c)\in \R^d\times\R^d\times\R.
$$
We compute 
$$
\partial_aF=\frac{(ab+c)}{1+\la a\ra^2}\lp b-\frac{(ab+c)}{1+\la a\ra^2}a\rp,\quad \partial_bF=\frac{(ab+c)}{1+\la a\ra^2}a,\quad \partial_cF=\frac{(ab+c)}{1+\la a\ra^2}.
$$
Taking~$a=\nabla\eta$, $b=\nabla\psi$, and~$c=G(\eta)\psi$ gives
$$
\partial_aF=BV,\quad \partial_bF=B\nabla\eta,\quad \partial_cF=B.
$$
The estimate \eqref{est:DN:C1} with $r=0$ gives
\[
\Vert (a, b, c)\Vert_{L^\infty}\les_\A 1.
\]
Next, Proposition \ref{est:DN:Sob} implies
\[
\Vert (a,b,c)\Vert_{H^{s-1}}\les_\A \Vert \eta\Vert_{H^{s+\mez}}+\Vert \psi\Vert_{H^s}.
\]
On the other hand,  the estimate \eqref{est:DN:C1} with $r=1$ implies
\[
\Vert (a,b,c)\Vert_{\Cs^1}\les_A\B.
\]
Using the above estimates, we can apply Theorem~\ref{paralin} with $\rho=1$ to have
$$\mez\frac{(\nabla\eta\cdot\nabla\psi-G(\eta)\psi)^2}{1+\la\nabla\eta\ra^2}=T_{VB}\cdot\nabla\eta+T_{B\nabla\eta}\cdot\nabla\psi+T_BG(\eta)\psi+f_1,$$
with
$$\lA f_1\rA_{H^{s-1+1}}\les_\A \B \lB\Vert \eta\Vert_{H^{s+\mez}}+\Vert \psi\Vert_{H^s}\rB.$$
By the same theorem, there holds
\[ \mez\la\nabla\psi\ra^2=T_{\nabla\psi}\cdot\nabla\psi+f_2,\quad \lA f_2\rA_{H^{s-1+1}}\les_\A \B\lA\psi\rA_{H^s}.
\]
At last, we deduce  from Theorem \ref{theo:sc} (ii) (with $m=m'=0,~\rho=\mez$) and the estimates for~$(B,V)$ in Lemma ~\ref{est:VB} that
$$\lA(T_{BV}-T_VT_B)\cdot\nabla\eta\rA_{H^{s-\mez+\mez}}\les_\A \B\Vert \nabla\eta\Vert_{H^{s-\mez}}
$$
A combination of the above paralinearizations concludes the proof.
\end{proof}
\begin{lemm} \label{lem:derB}
We have
	$$\lA T_{\partial_tB}\eta\rA_{H^{s}}\les_\A \B\lA\eta\rA_{H^{s+\mez}}.$$
\end{lemm}
\begin{proof}
Applying the paraproduct rule \eqref{pest1} gives
\[
\lA T_{\partial_tB}\eta\rA_{H^{s}}\les \lA \partial_tB\rA_{\Cs^{-\mez}}\lA\eta\rA_{H^{s+\mez}}.
\]
The proof thus boils down to showing $\Vert\partial_t B\Vert_{\Cs^{-\mez}}\les_\A\B$. By Theorem \ref{shape} for the shape derivative of the Dirichlet-Neumann, we have 
$$\partial_t\lb G(\eta)\psi\rb=G(\eta)(\partial_t\psi-B\partial_t\eta)-\dive(V\partial_t\eta).
$$
From the formulas of $V, B$ and the definition of $G(\eta)\psi$,  the water waves system \eqref{ww} can be rewritten as
\bq\label{ww:1}
\left\{
\begin{aligned}
&\partial_t\eta=B-V\cdot\nabla\eta,\\
&\partial_t\psi=-V\cdot\nabla\psi-g\eta+\mez V^2+\mez B^2+H(\eta).
\end{aligned}
\right.
\eq
 We first estimate using Lemma \ref{est:VB} and \eqref{tame:H}
\[
\Vert \dive(V\partial_t\eta)\Vert_{\Cs^{-\mez}}\les \Vert VB-V(V\cdot \nabla\eta)\Vert_{\Cs^\mez}\les \Vert VB-V(V\cdot \nabla\eta)\Vert_{\Cs^1}\les_\A \B.
\]	
Similarly, we get 
\[
\Vert \partial_t\psi-B\partial_t\eta\Vert_{L^\infty}\les_\A,\quad \Vert \partial_t\psi-B\partial_t\eta\Vert_{\Cs^1}\les_\A\B.
\]
Consequently, the estimate \eqref{est:DN:C-mez} yields
\[
\Vert G(\eta)(\partial_t\psi-B\partial_t\eta)\Vert_{\Cs^{-\mez}}\les_\A \B,
\]
from which we conclude the proof. Remark that the estimate \eqref{est:DN:C-mez:l} is not applicable to $G(\eta)(\partial_t\psi-B\partial_t\eta)$ since under the assumption \eqref{assu:reg} we only have $\partial_t\psi-B\partial_t\eta\in H^{\frac{d}{2}+}$ (due to the bad term $H(\eta)$) and not $H^{\mez+\frac d2+}$.
\end{proof}
We now have all the ingredients needed to paralinearize \eqref{ww}.
\begin{prop}\label{prop:PL}
	There exists a nondecreasing function~$\cF$ such that with ~$U:=\psi-T_B\eta$ there holds
	\begin{equation}	\label{eq:PL}
	\lB\begin{aligned}
		   \partial_t\eta+T_V\cdot\nabla\eta-T_\lambda U=&f_1,\\
		   \partial_tU+T_V\cdot\nabla U+T_{\l}\eta=&f_2,
	    \end{aligned}
	\right.
	\end{equation}
	with~$(f_1,f_2)$ satisfying
\[
\lA(f_1,f_2)\rA_{H^{s+\mez}\times H^s}\les_\A \B\lB\lA\psi\rA_{H^s}+\lA\eta\rA_{H^{s+\mez}}\rB.
\]
\end{prop}
\begin{proof}
	The first equation is an immediate consequence of the equation $\partial_t\eta=G(\eta)\psi$ and Proposition~\ref{prop:paralin:DN}.
	For the second one, we use the second equation of \eqref{ww} and Lemmas~\ref{lem:paracurv}, \ref{lem:paraquad} to get
	$$\partial_t\psi-T_BG(\eta)\psi+T_V(\nabla\psi-T_B\cdot\nabla\eta)+T_{\l}\eta=R$$
	with
	$$\lA R\rA_{H^s}\les_\A \B\lB\lA\psi\rA_{H^s}+\lA\eta\rA_{H^{s+\mez}}\rB.$$
	Next, differentiating $U$ with respect to $t$ yields
	$$\partial_tU=\partial_t\psi-T_B\partial_t\eta-T_{\partial_tB}\eta=\partial_t\psi-T_BG(\eta)\psi-T_{\partial_tB}\eta,$$
where the $H^s$-norm of $T_{\partial_t B}\eta$ is controlled by means of Lemma~\ref{lem:derB}.\\
	On the other hand, 
\[
      \nabla\psi-T_B\nabla \eta=\nabla U+T_{\nabla B}\eta
\]
and by \eqref{est:VB:C1}
	$$\lA T_VT_{\nabla B}\eta\rA_{H^s}\les_\A \lA T_{\nabla B}\eta\rA_{H^s}\les_\A \lA \nabla B\rA_{L^\infty}\lA \eta\rA_{H^{s}}\les_\A\B \Vert \eta\Vert_{H^{s}}.
$$
The proof is complete.
\end{proof}

%%%%%%%%%%%%%%%%%%%%%%%%%%%%%%%%%%%%%%%%
\subsection{Symmetrization of the system}\label{section:sym}
As in \cite{ABZ1} we shall deal with a class of symbols having a special structure that we recall here .
\begin{defi}\label{structure:symbol} Given $m\in \xR$, $\Sigma^m$ denotes the class of symbols $a$ of the form  $a=a^{(m)}+a^{(m-1)}$ with 
\[
a^{(m)}( x, \xi)=F(\nabla\eta(x), \xi),\quad a^{(m-1)}( x, \xi)=\sum_{|\alpha|=2}F_{\alpha}(\nabla\eta(x), \xi)\partial_x^{\alpha}\eta( x)
\]
such that 
\begin{enumerate}
\item  $T_a$ maps real-valued functions to real-valued functions;
\item $F$ is a $C^{\infty}$ real-valued function of $(\zeta, \xi)\in \xR^d\times\xR^d\setminus\{0\}$, homogeneous of order $m$ in $\xi$, and there exists a function $K=K(\zeta)>0$ such that 
\[
F(\zeta, \xi)\ge K(\zeta)|\xi|^m, \quad\forall (\zeta, \xi)\in  \xR^d\times\xR^d\setminus\{0\};
\]
\item the $F_{\alpha}$s are complex-valued functions of  $(\zeta, \xi)\in\xR^d\times\xR^d\setminus\{0\}$, homogeneous of order $m-1$ in $\xi$.
\end{enumerate}
\end{defi}
\hk In what follows, we often need an estimate for $u$ from $T_au$. For this purpose, we prove the next proposition.
\begin{prop}\label{inverse}
Let $m,~\mu,~M \in \xR$. Then, for all ~$a\in\Sigma^m$, there exists a nondecreasing function $\cF$ such that
\begin{gather}
\Vert u\Vert_{H^{\mu+m}}\le \cF(\Vert \eta\Vert_{\Cs^{2}})\left(\Vert T_{a}u\Vert_{H^{\mu}}+\Vert u\Vert_{H^{-M}}\right)\label{est:inverse},\\
\Vert u\Vert_{C_*^{\mu+m}}\le \cF(\Vert \eta\Vert_{\Cs^{2}})\left(\Vert T_{a}u\Vert_{C_*^{\mu}}+\Vert u\Vert_{\Cs^{-M}}\right)\label{est:inverse:H},
\end{gather}
here $\cF$ depends only on $m,~\mu,~M$ and the functions $F,~F_\alpha$ given in Definition \ref{structure:symbol} of the class $\Sigma^m$.
\end{prop}
\begin{rema}
The same result was proved in Proposition $4.6$ of~\cite{ABZ1} where the constant in the right hand side reads $\cF(\Vert \eta(t)\Vert_{H^{s-1}})$, $s>2+\frac{d}{2}$. %Here, our estimates involves only the Holder norm of $\eta$.
\end{rema}
\begin{proof}
We give the proof for \eqref{est:inverse}, the proof for \eqref{est:inverse:H} follows similarly. We write $a=a^{(m)}+a^{(m-1)}$. Set $b=\frac{1}{a^{(m)}}$. Applying Theorem \ref{theo:sc} (ii) with $\rho=\eps$ gives $T_bT_{a^{(m)}}=I+r$ where~$r$ is of order~$-\eps$ and 
\bq\label{inverse:1}
\Vert ru\Vert_{H^{\mu+\eps}}\le \cF\lp\Vert \nabla\eta\Vert_{C^{\eps}}\rp\Vert u\Vert_{H^{\mu}}\le \cF\lp\Vert \eta\Vert_{C^{1+\eps}}\rp\Vert u\Vert_{H^{\mu}}.
\eq
Then, setting $R=-r-T_bT_{a^{(m-1)}}$ we have
\bq\label{eq:R:inverse}
(I-R)u=T_bT_au.
\eq
Let us consider the symbol $a^{(m-1)}$ having the structure given by Definition \ref{structure:symbol}. Applying \eqref{tame:H<0} and \eqref{F(u):C} yields for $|\alpha|=2$ and uniformly for $|\xi|=1$, 
\[
\Vert F_{\alpha}(\nabla\eta, \xi)\partial_x^{\alpha}\eta\Vert_{C^{-1+\eps}_*}
\le \Vert F_{\alpha}(\nabla\eta , \xi)\Vert_{\Cs^1}\Vert\partial_x^{\alpha}\eta\Vert_{\Cs^{-1+\eps}}\le \cF(\Vert \eta\Vert_{\Cs^{2}}).
\]
Similar estimates also hold when taking $\xi$-derivatives of $F_{\alpha}(\nabla\eta, \xi)\partial_x^{\alpha}\eta$. Consequently,  $a^{(m-1)}\in \dot{\Gamma}^{m-1}_{-1+\eps}$ and thus by Proposition \ref{regu<0},
\[
\Vert T_{a^{(m-1)}}u\Vert_{H^{\mu-m+\eps}}\le \cF(\Vert \eta\Vert_{\Cs^{2}})\Vert u\Vert_{H^{\mu}}.
\]
Because $b\in \Gamma^{-m}_0$ with semi-norm bounded by $\cF(\Vert \eta\Vert_{\Cs^{1+\eps}})$ we get
\bq
\label{inverse:2}
\Vert T_bT_{a^{(m-1)}}u\Vert_{H^{\mu+\eps}}\le \cF(\Vert \eta\Vert_{\Cs^{2}})\Vert u\Vert_{H^{\mu}}.
\eq
Combining \eqref{inverse:1} with \eqref{inverse:2} yields 
\[
\Vert Ru\Vert_{H^{\mu+\eps}}\le \cF(\Vert \eta\Vert_{\Cs^{2}})\Vert u\Vert_{H^{\mu}}.
\]
In other words, $R$ is a smoothing operator of order $-\eps$. Now, multiplying both sides of \eqref{eq:R:inverse} by $1+R+...+R^N$ leads to 
\[
u-R^Nu=(1+R+...+R^N)T_bT_au.
\]
On the one hand, using the fact that $R$ is of order $0$, we get
\begin{align*}
\Vert (1+R+...+R^N)T_bT_au\Vert_{H^{\mu+m}}&\le \cF(\Vert \eta\Vert_{\Cs^{2}})\Vert T_bT_au\Vert_{H^{\mu+m}}\\
&\le \cF(\Vert \eta\Vert_{\Cs^{2}})\Vert T_au\Vert_{H^{\mu}}.
\end{align*}
On the other hand, that $R$ is of order $-\eps$ implies
\[
\Vert R^Nu\Vert_{H^{\mu+m}}\le \cF(\Vert \eta\Vert_{\Cs^{2}})\Vert u\Vert_{H^{\mu+m-N\eps}}.
\]
Therefore, by choosing $N$ sufficiently large we conclude the proof.
\end{proof}
For the sake of conciseness, we give the following definition.
\begin{defi}\label{equi:operators}
Let $m\in \xR$ and consider two families of operators of order~$m$,
\[
\{ A(t): t\in [0, T]\},\quad \{ B(t): t\in [0, T]\}.
\]
We write $A\sim B$ (in $\Sigma^m$) if $A-B$ is of order $m-\tdm$ and the following condition is fulfilled: for all $\mu\in \xR$, there exists a nondecreasing function $\cF$ such that for a.e. $t\in [0, T]$, 
\[
\lA A(t)-B(t)\rA_{H^{\mu}\to H^{\mu-m+\tdm}}\le \cF(\lA \eta(t)\rA_{\Cs^2}\big)\big(1+\lA \eta(t)\rA_{C^{\frac{5}{2}}}\big).
\]
\end{defi}
\begin{prop}\label{rema:class}
For any $a\in \Sigma^m$ and  $b\in \Sigma^{m'}$, it holds that
\[
T_aT_b\sim T_{c}
\]
(in $\Sigma^{m+m'}$) with 
\[
c=a^{(m)}b^{(m')}+a^{(m-1)}b^{(m')}+a^{(m-1)}b^{(m')}+\frac{1}{i}\partial_\xi a^{(m)}\partial_x b^{(m')}.
\]
\end{prop}
\begin{proof}
1. Since the principal symbol $a^{(m)}(t)$ contains only the first order derivatives of $\eta$, applying the nonlinear estimate \eqref{F(u):H} we obtain
\[
M^m_{3/2}(a^{(m)}(t))\le  \cF(\lA \eta(t)\rA_{\Cs^{1+\eps}}\big)\big(1+\lA \eta(t)\rA_{C^{\frac{5}{2}}}\big).
\]
On the other hand, 
\[
M^m_{1/2}(a^{(m)}(t))\le  \cF(\lA \eta(t)\rA_{\Cs^{\tdm}})
\]
and 
\[
M^m_0(a^{(m)}(t))\le \cF(\lA \eta(t)\rA_{\Cs^{1+\eps}}\big).
\]
2. The subprincipal symbol $a^{(m-1)}(t)$ depends linearly on $\partial^{\alpha}\eta~, |\alpha|=2$ and nonlinearly on $\nabla\eta$. Hence $a^{(m-1)}\in \Gamma^{m-1}_{1/2}$ and by \eqref{tame:H} and \eqref{F(u):H} we have uniformly for $|\xi|=1$,
\begin{align*}
&\lA F_{\alpha}(\nabla\eta(t,x), \xi)\partial_x^{\alpha}\eta(t, x) \rA_{\Cs^{\mez}}\\
&\le \lA [F_{\alpha}(\nabla\eta(t,\cdot), \xi)-F_{\alpha}(0, \xi)]\partial_x^{\alpha}\eta(t,\cdot) \rA_{\Cs^\mez}+\la F_{\alpha}(0, \xi)\ra \lA \partial_x^{\alpha}\eta(t, \cdot) \rA_{\Cs^\mez}\\
&\le \cF(\lA\eta(t)\rA_{\Cs^{\tdm}})\lA \eta(t) \rA_{\Cs^{\frac{5}{2}}}.
\end{align*}
The same estimates hold when taking $\xi$-derivatives, consequently
\[
M^{m-1}_{1/2}(a^{(m-1)}(t))\le  \cF(\lA\eta(t)\rA_{\Cs^{\tdm}})\lA \eta(t) \rA_{\Cs^{\frac{5}{2}}}.
\]
On the other hand, 
\[
M^{m-1}_0(a^{(m-1)}(t))\le \cF(\lA\eta(t)\rA_{\Cs^{2}}).
\]
3. We now write
\[
T_aT_b=T_{a^{(m)}}T_{b^{(m')}}+T_{a^{(m-1)}}T_{b^{(m')}}+T_{a^{(m)}}T_{b^{(m'-1)}}+T_{a^{(m-1)}}T_{b^{(m'-1)}}.\]
 Using 1. and 2., we deduce by virtue of Theorem \ref{theo:sc} (ii) with $\rho=3/2$ that 
\begin{multline*}
\Vert T_{a^{(m)}}T_{b^{(m')}}- T_{a^{(m)}b^{(m')}+\frac{1}{i}\partial_\xi a^{(m)}\partial_x b^{(m')}}\Vert_{H^{\mu}\to H^{\mu-(m+m')+\tmez}}\cr
\le  \cF(\lA \eta(t)\rA_{\Cs^{1+\eps}}\big)\big(1+\lA \eta(t)\rA_{C^{\frac{5}{2}}}\big).
\end{multline*}
 The same theorem, applied with $\rho=1/2$, yields
\begin{align*}
&\Vert T_{a^{(m-1)}}T_{b^{(m')}}-T_{a^{(m-1)}b^{(m')}}\Vert_{H^{\mu}\to H^{\mu-(m+m')+\tmez}}\le  \cF(\lA \eta(t)\rA_{\Cs^2}\big)\big(1+\lA \eta(t)\rA_{C^{\frac{5}{2}}}\big),\\
&\Vert T_{a^{(m)}}T_{b^{(m'-1)}}- T_{a^{(m-1)}b^{(m')}}\Vert_{H^{\mu}\to H^{\mu-(m+m')+\tmez}}\le \cF(\lA \eta(t)\rA_{\Cs^2}\big)\big(1+\lA \eta(t)\rA_{C^{\frac{5}{2}}}\big).
\end{align*}
Finally, applying Theorem \ref{theo:sc} (i) leads to
\[
\Vert T_{a^{(m-1)}}T_{b^{(m'-1)}}\Vert_{H^{\mu}\to H^{\mu-(m+m')+2}}\le  \cF(\lA \eta(t)\rA_{\Cs^2}\big).
\]
Putting  the above estimates together we conclude that $T_aT_b\sim T_c$ in $\Sigma^{m+m'}$.
\end{proof}
Using the preceding Proposition, one can easily verify  that  Proposition $4.8$ in \cite{ABZ1} is still valid:
\begin{prop}\label{pq}
Let $q\in \Sigma^0,~p\in \Sigma^{\mez},~\gamma\in \Sigma^{\tdm}$ defined by 
\[
\begin{aligned}
q&=(1+|\partial_x\eta|^2)^{-\mez},\\
p&=(1+|\partial_x\eta|^2)^{-\frac{5}{4}}\sqrt{\lambda^{(1)}}+p^{(-1/2)},\\
\gamma&=\sqrt{\ell^{(2)}\lambda^{(1)}}+\sqrt{\frac{\ell^{(2)}}{\lambda^{(1)}}}\frac{\Re \lambda^{(0)}}{2}-\frac{i}{2}(\partial_{\xi}\cdot\partial_x)\sqrt{\ell^{(2)}\lambda^{(1)}},
\end{aligned}
\]
where 
\[
p^{(-1/2)}=\frac{1}{\gamma^{(3/2)}}\left\{q^{(0)}\ell^{(1)}-\gamma^{(1/2)}p^{(1/2)}+i\partial_{\xi}\gamma^{(3/2)}\partial_xp^{(1/2)} \right\}.
\]
Then, it holds that
\[
T_pT_{\lambda}\sim T_{\gamma}T_q,\quad T_qT_{\ell}\sim T_{\gamma}T_p,\quad (T_{\gamma})^*\sim T_\gamma.
\]
\end{prop}
We are now in position to perform the symmetrization.
\begin{prop}\label{sym:prop}
Introduce two new unknowns
\[
\Phi_1=T_p\eta,\quad \Phi_2=T_qU.
\]
Then $\Phi_1,~\Phi_2\in L^\infty ([0, T], H^s)$ and 
\bq
\left\{
\begin{aligned}
&\partial_t\Phi_1+T_V\cdot\nabla\Phi_1- T_{\gamma}\Phi_2=F_1,\\
&\partial_t\Phi_2+T_V\cdot\nabla\Phi_2+T_{\gamma}\Phi_2=F_2,
\end{aligned}
\right.
\eq
where, there exists a nondecreasing function $\cF$ independent of $\eta, \psi$ such that for {\it a.e.} $t\in [0, T]$, there holds
\bq\label{SS:Fi}
\lA (F_1, F_2)\rA_{ H^s\times H^s}\les_\A\B\lB\lA\eta\rA_{H^{s+\mez}}+\lA  \psi\rA_{H^{s}}\rB.
\eq
\end{prop}
\begin{proof}
It follows directly from system \eqref{eq:PL} that $\Phi_1,~\Phi_2$ satisfy
\bq
\left\{
\begin{aligned}
&\partial_t\Phi_1+T_V\cdot\nabla\Phi_1- T_{\gamma}\Phi_2=T_pf_1+T_{\partial_tp}\eta+[T_V\cdot\nabla,T_p]\eta+R_1,\\
&\partial_t\Phi_2+T_V\cdot\nabla\Phi_2+T_{\gamma}\Phi_2=T_qf_2+T_{\partial_tq}U+[T_V\cdot\nabla,T_q]U+R_2,
\end{aligned}
\right.
\eq
where 
\[
R_1=(T_pT_{\lambda}-T_{\gamma}T_q)\psi,\quad R_2=-(T_qT_{\ell}- T_{\gamma}T_p)\eta.
\]
Let $\Pi$ denote the right-hand side of \eqref{SS:Fi}. According to Proposition \ref{pq},
\[
\Vert R_1\Vert_{H^s}+\Vert R_2\Vert_{H^s}\les \Pi.
\]
On the other hand, Proposition \ref{prop:PL} implies
\[
\Vert T_pf_1\Vert_{H^s}+\Vert T_qf_2\Vert_{H^s}\les \Pi.
\]
Owing to Lemma \ref{est:VB} and the norm estimates for symbols in Proposition \ref{rema:class}, the composition rule of Theorem \ref{theo:sc} (ii) (with  $\rho=1$) yields
\[
\lA [T_V\cdot\nabla,T_p]\eta\rA_{H^s}+\lA [T_V\cdot\nabla,T_q]U\rA_{H^s}\les \Pi.
\]
It remains to prove
\[
\lA T_{\partial_tp}\rA_{H^{s+\mez}\to H^s}+ \lA T_{\partial_tq}\rA_{H^s\to H^s}\les_\A \B.
\]
To this end, we first recall from the first equation of \eqref{ww:1} that $\partial_t\eta=B-V\cdot\nabla\eta$. Hence $\|\partial_t\eta\|_{W^{1,\infty}}\les_\A \B$ and 
\[
\cM^{1/2}_0(\partial_tp^{(1/2)})+\cM^{0}_0(\partial_tq)\les_\A\B,
\]
 which, combined with Theorem \ref{theo:sc} (i), yields
\[
\Vert T_{\partial_tp^{(1/2)}}\Vert_{H^{s+\mez}\to H^s}+ \lA T_{\partial_tq}\rA_{H^s\to H^s}\les_\A\B.
\]
We are thus left with the estimate of $\Vert T_{\partial_tp^{(-1/2)}}\Vert_{H^{s+\mez}\to H^s}$. According to Proposition \ref{regu<0}, it suffices to show 
\bq\label{norm:dtp}
\cM^{-1/2}_{-1}(\partial_tp^{(-1/2)})\les_\A \B.
\eq
Recall that $p^{(-1/2)}$ is of the form 
\[
p^{(-1/2)}=\sum_{|\alpha|=2}F_{\alpha}(\nabla \eta, \xi)\partial_x^{\alpha}\eta,
\]
where the~$F_{\alpha}$ are smooth functions in $\xi\ne 0$ and homogeneous of order $-1/2$. Hence,
\[
\partial_tp^{(-1/2)}=\sum_{|\alpha|=2}[\partial_tF_{\alpha}(\nabla \eta, \xi)]\partial_x^{\alpha}\eta+\sum_{|\alpha|=2}F_{\alpha}(\nabla \eta, \xi)\partial_t\partial_x^{\alpha}\eta.
\]
It is easy to see that 
\[
\cM^{-\mez}_0\big([\partial_tF_{\alpha}(\nabla \eta, \xi)]\partial_x^{\alpha}\eta\big)\les_\A 1.
\]
For the main term $F_{\alpha}(\nabla \eta, \xi)\partial_t\partial_x^{\alpha}\eta$ we use the first equation of \eqref{ww:1} to have
\[\partial_t\partial_x^{\alpha}\eta=\partial_x^\alpha(B-V\nabla_x\eta).
\]
Hence 
\[
\Vert \partial_t\partial_x^{\alpha}\eta\Vert_{\Cs^{-1}}\le \Vert B-V\nabla_x\eta\Vert_{\Cs^1}\les_\A \B.
\]
The product rule \eqref{tame:H<0} then implies
\[
\cM^{-\mez}_{-1}\big(F_{\alpha}(\nabla \eta, \xi)\partial_t\partial_x^{\alpha}\eta)\les_\A \B,
\]
which concludes the proof of \eqref{norm:dtp} and hence of the proposition.
\end{proof}
%%%%%%%%%%%%%%%%%%%%%%%%%%%%
\section{A priori estimates and blow-up criteria}\label{section:apriori}
\subsection{A priori estimates}
First of all, it follows straightforwardly from Proposition \ref{sym:prop} that  the water waves system can be reduced to a single equation of a complex-valued unknown as follows.
\begin{prop}\label{singleeq:Phi}
Assume that $(\eta, \psi)$ is a solution to \eqref{ww} and satisfies \eqref{assu:reg}. Let $\Phi_1, \Phi_2$ be as in Proposition \ref{sym:prop}, then 
\[
\Phi\defn \Phi_1+i\Phi_2=T_p\eta+iT_qU
\]
 satisfies
\begin{gather}
\left( \partial_t+T_V\cdot\nabla+iT_{\gamma}\right)\Phi=F,\label{Phi}\\
\lA F(t)\rA_{ H^s}\les_\A\B\lB\lA\eta\rA_{H^{s+\mez}}+\lA  \psi\rA_{H^{s}}\rB. \label{F-tame}
\end{gather}
\end{prop}
\hk In order to obtain $H^s$ estimate for $\Phi$, we shall commute equation \eqref{Phi} with an elliptic operator $\wp$ of order $s$ and then perform an $L^2$-energy estimate. Since $\gamma^{(3/2)}$ is of order $3/2>1$, we need to choose $\wp$ as a function of $\gamma^{(3/2)}$ as in \cite{ABZ1}:
\bq\label{def:wp}
\wp\defn (\gamma^{(3/2)})^{2s/3},
\eq
and take $\ph=T_{\wp}\Phi$. To obtain energy estimates in terms of the original variables~$\eta$ and~$\psi$, it is necessary to link them with this new unknown~$\ph$. 
\begin{lemm}
	We have 
\begin{align}
&\lA\ph\rA_{L^2}\les_\A\lA\eta\rA_{H^{s+\mez}}+\lA\psi\rA_{H^s}, \label{new-old}\\
&\lA\eta\rA_{H^{s+\mez}}+\lA\psi\rA_{H^s}\les_\A \lA\ph\rA_{L^2}+\Vert \eta\Vert_{L^2}+\Vert \psi\Vert_{L^2}. \label{old-new} 
\end{align}
\end{lemm}
\begin{proof}
	Recall that~$p\in\Sigma^s$, $q\in\Sigma^0$, and~$\wp\in\Sigma^s$ since~$\gamma^{(\frac32)}\in\Sigma^{\frac32}$. The estimate \eqref{new-old} is then a direct consequence of Theorem \ref{theo:sc} (i). To prove \eqref{old-new} we apply Proposition~\ref{inverse} twice to get
	\begin{align*}
	&\lA\eta\rA_{H^{s+\mez}}\les_A\lA T_\wp T_p\eta\rA_{L^2}+\lA\eta\rA_{L^2},\\
	&\lA\psi\rA_{H^s}\les_\A\lA T_\wp T_q\psi\rA_{L^2}+\lA\psi\rA_{L^2}.
	\end{align*}
	Clearly, $\Vert T_\wp T_p\eta\Vert_{L^2}\le \lA \varphi\rA_{L^2}$ , hence
\[
\lA\eta\rA_{H^{s+\mez}}\les_\A \Vert \varphi\Vert_{L^2}+\Vert \eta\Vert_{L^2}.
\]
 On the other hand,
\begin{align*}
\lA T_\wp T_q\psi\rA_{L^2}&\leq\lA T_\wp T_qU\rA_{L^2}+\lA T_\wp T_qT_B\eta\rA_{L^2}\\
&\le \lA \varphi\rA_{L^2}+\lA T_\wp T_qT_B\eta\rA_{L^2}\\
&\les_\A \lA \varphi\rA_{L^2}+\Vert \eta\Vert_{H^{s+\mez}}\\
&\les_\A \lA \varphi\rA_{L^2}+\Vert \eta\Vert_{L^2}.
\end{align*}
This completes the proof of \eqref{old-new}.
\end{proof}
\begin{prop}\label{L2varphi}
There exists a nondecreasing function $\cF:\xR^+\to \xR^+$ depending only on $s, \eps_*, h$ such that for any $t\in [0, T]$, 
\begin{equation}	\label{eq:enerL^2}
		\frac{\d}{\d t}\lA\ph\rA_{L^2}^2\le \cF(\A)\B (\Vert \eta\Vert_{L^2}+\Vert \psi\Vert_{L^2}+\lA\ph\rA_{L^2})\lA\ph\rA_{L^2}.
	\end{equation}
\end{prop}
\begin{proof}
We see from~\eqref{Phi} that~$\ph$ solves the equation
	\bq\label{varphi}
	\left( \partial_t+T_V\cdot\nabla+iT_{\gamma}\right)\varphi=T_{\wp}F+G
	\eq
	where 
	\[
	G=T_{\partial_t\wp}\Phi+[T_V\cdot\nabla, T_{\wp}]\Phi+i[T_{\gamma}, T_{\wp}]\Phi.
	\]
	 First, remark that since $\partial_{\xi}\wp\cdot\partial_x\gamma^{(3/2)}=\partial_{\xi}\gamma^{(3/2)}\cdot\partial_x\wp$ we can apply Lemma \ref{rema:class} twice: once with $m=s,~ m'=\tdm,~\rho=\tdm$ and once with $m=\tdm,~ m'=s,~\rho=\tdm$ to find 
\[
\Vert [T_{\wp}, T_{\gamma}]\Vert_{H^s\to L^2}\les_\A\B.
\]
On the other hand, Theorem \ref{theo:sc} (ii) applied with $\rho=1$ gives
\[
\Vert [T_V\cdot\nabla, T_{\wp}]\Vert_{H^s\to L^2}\les_\A\B.
\]
Next, we write $\partial_t\wp=L(\nabla\eta, \xi) \partial_t\nabla \eta$ for some smooth function $L$ homogeneous of order $s$ in $\xi$, where by the first equation of \eqref{ww:1} $\Vert  \partial_t\nabla \eta\Vert_{L^\infty}\les_\A \B$. Hence
\[
\Vert T_{\partial_t\wp}\Vert_{H^s\to L^2}\les_\A\B.
\]
Putting the above estimates together leads to
$$\lA G\rA_{L^2}\les_\A\B\lA\Phi\rA_{H^s}.$$
On the other hand,  Proposition~\ref{inverse} applied to $u=\Phi,~a=\wp\in \Sigma^s$ yields
$$\lA\Phi\rA_{H^s}\les_\A\lA\ph\rA_{L^2}+\Vert \eta\Vert_{L^2}+\Vert \psi\Vert_{L^2}.$$
Therefore,
$$\lA G\rA_{L^2}\les_\A \B(\lA\ph\rA_{L^2}+\Vert \eta\Vert_{L^2}+\Vert \psi\Vert_{L^2}).$$
On the other hand, \eqref{F-tame} together with \eqref{old-new} implies
\bq\lA T_\wp F\rA_{L^2}\les_\A \B(\lA\ph\rA_{L^2}+\Vert \eta\Vert_{L^2}+\Vert \psi\Vert_{L^2}).
\eq
Now, using Theorem \ref{theo:sc} (iii) and he proof of Lemma \ref{rema:class} we easily  find that 
\bq\label{conjugate:V}
\Vert (T_V\cdot\nabla)+(T_V\cdot\nabla)^*\Vert_{L^2\to L^2}\les_\A\B.
\eq
On the other hand, according to Proposition \ref{pq}, $(T_\gamma)^*\sim T_\gamma$, so
\bq\label{conjugate:gamma}
\Vert (T_{\gamma})-(T_{\gamma})^*\Vert_{L^2\to L^2}\les_\A\B.
\eq
Therefore, by an $L^2$-energy estimate for  \eqref{varphi} we end up with \eqref{eq:enerL^2}.
\end{proof}
\begin{prop}\label{energy:W}
Set $W=(\eta, \psi)$, $\cH^{r}=H^{r+\mez}\times H^r$. Then, there exists a nondecreasing function $\cF:\xR^+\to \xR^+$ depending only on $s, \eps_*, h$ such that for {\it a.e.} $t\in [0, T]$, 
\[
\Vert W(t)\Vert^2_{\cH^s}\le \cF(P^1(t))\Vert W(0)\Vert^2_{\cH^s}+\cF(P^1(t))\int_0^t \B(r) \Vert W(r)\Vert_{\cH^s}^2\d r
\]
with 
\[
P^1(t):=\sup_{r\in [0, t]}\A(r).
%\sup_{r\in [0, t]}\big(\Vert \eta(r)\Vert_{\Cs^{2+\eps_*}}+\Vert \nabla \psi(r)\Vert_{\Bl^0}+\Vert \eta(r)\Vert_{L^2}+E(\eta(r), \psi(r))\big).
\]
\end{prop}
\begin{proof}  Integrating \eqref{eq:enerL^2} over $[0, t]$ and using \eqref{new-old}-\eqref{old-new}, we obtain
\bq\label{dt:W1}
\begin{aligned}
\Vert W(t)\Vert^2_{\cH^s}&\les_\A \Vert W(t)\Vert^2_{L^2\times L^2}+\lA\ph\rA^2_{L^2}\\
&\les_\A \Vert W(t)\Vert^2_{L^2\times L^2}+\Vert W(0)\Vert^2_{\cH^s}+\int_0^t \cF(\A(r))\B(r) \Vert W(r)\Vert_{\cH^s}^2\d r.
\end{aligned}
\eq
Recall the system \eqref{ww:1} satisfied by $W$:
\[
\left\{
\begin{aligned}
&\partial_t\eta=B-V\cdot\nabla\eta,\\
&\partial_t\psi=-V\cdot\nabla\psi-g\eta+\mez V^2+\mez B^2+H(\eta).
\end{aligned}
\right.
\]
A standard $L^2$ estimate for each equation gives
\[
\frac{\d}{\d t}\Vert W(t)\Vert^2_{L^2\times L^2}\les_\A\B \Vert W(t)\Vert^2_{\cH^s}.
\]
Hence 
\[
\Vert W(t)\Vert^2_{L^2\times L^2}\le \Vert W(0)\Vert^2_{L^2\times L^2}+\int_0^t \cF(\A(r))\B(r) \Vert W(r)\Vert_{\cH^s}^2\d r.
\]
Plugging this into \eqref{dt:W1} we conclude the proof.
\end{proof}
Let us denote the Sobolev norm and the "Strichartz norm" of the solution by
\bq\label{nota:MZ}
\begin{aligned}
&M_{\sigma,T}=\Vert (\eta, \psi)\Vert_{L^{\infty}([0, T]; H^{\sigma+\mez}\times H^\sigma)},\\ &M_{\sigma,0}=\Vert (\eta, \psi)\arrowvert_{t=0}\Vert_{H^{\sigma+\mez}\times H^\sigma},\\
&N_{r, T}=\Vert (\eta, \nabla\psi)\Vert_{L^1([0, T]; W^{r+\mez}\times \Bl^1)}.
\end{aligned}
\eq
We next derive from Proposition \ref{L2varphi} an a priori estimate for t$M_{s,T}$ using the control of $N_{r,T}$.s
\begin{theo}\label{theo:aprioriestimate}
	Let~$d\geq 1$, $h>0$ and 
	$$s>\frac32+\frac d2,\quad r>2.$$
	Then there exists a nondecreasing function~$\cF:\xR^+\to \xR^+$ depending only on $(s, r, h, d)$ such that for all  $T\in [0, 1)$ and all $(\eta,\psi)$ solution to \eqref{ww} with
\begin{align*}
&(\eta,\psi)\in L^{\infty}\lp[0, T]; H^{s+\mez}\times H^s\rp,\\
&(\eta, \nabla\psi)\in L^1\lp[0, T]; W^{r+\mez, \infty}\times \Bl^1\rp.\\
&\inf_{t\in [0,T]}\dist(\eta(t), \Gamma)>h,
\end{align*}
there holds
\bq\label{aprioriestimate} M_{s, T}\leq\cF\lp M_{s,0}+T\cF\big(M_{s, T}\big)+N_{r,T}\rp.
\eq
\end{theo}
\begin{proof}
Pick 
\[
0<\eps< \mez\min\big\{\mez, r-1, s-\tdm-\frac d2\big\}.
\]
By Remark \ref{rema:E}, $E(\eta, \psi)\le \cF(\Vert \eta\Vert_{\Cs^{1+\eps}})\Vert \psi\Vert_{H^\mez}$.  Therefore, by applying Proposition \ref{energy:W} we obtain 
\[
M_{s,T}\le M_{s, 0}K(T)\exp\lp K(T)\int_0^T\Big(\Vert (\eta, \nabla\psi)(t)\Vert_{\Cs^{\frac 52+\eps}\times \Bl^1}+1\Big)dt\rp
\]
with 
\[
K(T):=\cF\Big(\sup_{t\in [0, T]}\big(\Vert (\eta, \psi)(t)\Vert_{\Cs^{2+\eps}\times \Cs^{\eps}}+\Vert (\eta, \psi)\Vert_{L^2\times H^\mez}\big)\Big).
\]
%By H\"older's inequality,
%\[
%\int_0^T\Vert (\eta, \psi)(\tau)\Vert_{\Cs^{\frac 52+\eps}\times \Cs^{2+\eps}}d\tau\le T^{1-1/p}Z_{r,T}.
%\]
 Therefore, it suffices to show for all $t\le T$
\[\Vert (\eta, \psi)(t)\Vert_{\Cs^{2+\eps}\times \Cs^{\eps}}+\Vert (\eta, \psi)(t)\Vert_{L^2\times H^\mez} \le \cF\big( M_{s,0}+TM_{s, T}\big).
\]
By Sobolev's embeddings, this reduces to 
\[\Vert (\eta, \psi)(t)\Vert_{H^{s+\mez-\eps}\times H^{s-\eps}} \le \cF\big( M_{s,0}+TM_{s, T}\big)\quad\forall t\le T.
\]
Using the Sobolev estimate for the Dirichlet-Neumann  in Proposition \ref{est:DN:Sob} in conjunction with Remark \ref{rema:E} we get
\bq\label{eta:s-1}
\Vert\eta(t)-\eta(0)\Vert_{H^{s-1}}\le\int_0^t\Vert\partial_t\eta(\tau)\Vert_{H^{s-1}}d\tau= \int_0^t\Vert G(\eta(\tau))\psi(\tau)\Vert_{H^{s-1}}d\tau\le T\cF (M_{s,T}).
\eq
Consequently, it follows by interpolation that
\bq\label{eta:s-mez}
\begin{aligned}
\Vert\eta(t)\Vert_{H^{s+\mez-\eps}}&\le \Vert\eta(0)\Vert_{H^{s+\mez-\eps}}+\Vert\eta(t)-\eta(0)\Vert_{H^{s+\mez-\eps}}
\\   
&\le M_{s,0}+\Vert\eta(t)-\eta(0)\Vert_{H^{s-1}}^\theta\Vert\eta(t)-\eta(0)\Vert_{H^{s+\mez}}^{1-\theta}\quad\theta\in (0, 1)\\&\le M_{s, 0}+T^\theta \cF(M_s(T)).
\end{aligned}
\eq
The estimate for $\Vert \psi(t)\Vert_{H^{s-\eps}}$ follows along the same lines using the second equation of \eqref{ww} (or \eqref{ww:1}) and interpolation.
\end{proof}
\subsection{Blow-up criteria}
Taking~$\sigma>2+\frac d2$ and 
\bq\label{condition:data}
(\eta_0, \psi_0)\in  H^{\sigma+\mez}\times H^\sigma,\quad \dist(\eta_0, \Gamma)>h>0,
\eq
 we know from Theorem $1.1$ in \cite{ABZ1} that there exists a time $ T\in (0, \infty)$ such that the Cauchy problem for system \eqref{ww} with initial data $(\eta_0, \psi_0)$ has a {\it unique solution} 
\[
(\eta, \psi)\in C\lp[0, T];  H^{\sigma+\mez}\times H^\sigma\rp
\]
satisfying 
\[
\sup_{t\in [0, T]}\dist(\eta(t), \Gamma)>\frac{h}{2}.
\]
 The {\it maximal time of existence} $T^*>0$ then can be defined as
\begin{multline}\label{def:T*}
T^*=T^*(\eta_0, \psi_0, \sigma, h):=\sup\left\{ T'>0:  ~\text {the Cauchy problem for \eqref{ww} with data}\right.\\
\left.  ~(\eta_0, \psi_0) ~\text{satisfying}~ \eqref{condition:data}~\text{has a solution}~(\eta, \psi)\in C([0, T'];  H^{\sigma+\mez}\times H^\sigma) \right.\\
\left. ~\text{satisfying}~ \inf_{[0, T']}\dist(\eta(t), \Gamma)>0\right\}.
\end{multline}
It should be emphasized that $T^*$ depends not only on $(\eta_0, \psi_0)$ and $\sigma$ but also on the initial depth $h$. By the uniqueness statement of Proposition $6.4$, \cite{ABZ1} (it is because of this Proposition that we require the separation condition in the definition \eqref{def:T*}) the solution $(\eta, \psi)$ is defined for all $t<T^*$ and 
\[
(\eta, \psi)\in C\lp[0, T^*);  H^{\sigma+\mez}\times H^\sigma\rp,
\]
which will be called the {\it maximal solution}.\\
We recall the following lemma from \cite{Wa2} (see Lemma $9.20$ there).
\begin{lemm}\label{B-H}
Let $\mu>1+\frac d2$. Then, there exists a constant $C>0$ such that 
\[
\Vert u\Vert_{\Bl^1}\le C\big(1+\Vert u\Vert_{\Cs^1}\big)\ln\big(e+\Vert u\Vert^2_{H^\mu}\big)
\]
provided the right-hand side is finite.
\end{lemm}
\begin{proof}
For the sake of completeness, we present the proof of this lemma, taken from \cite{Wa2}. Given an integer $N$, we have by the Berstein inequality
\[
\begin{aligned}
\Vert u\Vert_{\Bl^1}&=\sum_{j=0}^N2^j\Vert \Delta_ju\Vert_{L^\infty}+\sum_{j>N}2^j\Vert \Delta_ju\Vert_{L^\infty}\\
&\le (N+1)\Vert u\Vert_{\Cs^1}+\sum_{j>N}2^{j(1+\frac d2-\mu)}2^{j\mu}\Vert \Delta_ju\Vert_{L^2}.
\end{aligned}
\]
As $1+\frac d2-\mu<0$, it follows by H\"older's inequality for sequence that there exists $C>$ independent of $N$ such that
\[
\Vert u\Vert_{\Bl^1}\le (N+1)\Vert u\Vert_{\Cs^1}+C2^{-N(\mu-1-\frac d2)}(\Vert u\Vert_{H^\mu}+e).
\] 
Choosing $N\sim \ln(e+\Vert u\Vert_{H^\mu})$ so that $2^{-N(\mu-1-\frac d2)}(\Vert u\Vert_{H^\mu}+e)\sim 1$, we obtain the desired inequality.
\end{proof}
\begin{prop}\label{prop:grow}
Let~$d\geq 1, h>0$, $\sigma>2+\frac{d}{2}$, $T>0$. Let
\[
(\eta, \psi)\in  C([0, T]; H^{\sigma+\mez}\times H^\sigma),\quad \inf_{t\in [0, T]}\dist(\eta(t), \Gamma)>h>0
\]
be a solution to \eqref{ww}. Fix $\eps_*\in (0, \sigma-\tdm-\frac d2)$. Then there exists a nondecreasing function~$\cF:\xR^+\to \xR^+$ depending only on $(\sigma, \eps_*, h, d)$ such that
\[
M^2_{\sigma, T}\le \cF(P^2(t))\big(M^2_{\sigma, 0}+2e\big)e^{e^{\int_0^tQ(r)dr}}-2e
\]
with 
\begin{align*}
&Q(r):=1+\Vert \nabla\psi(r)\Vert_{\Cs^1}+\Vert \eta(r)\Vert_{\Cs^{2+\eps_*}},\\
&P^2(t):=\sup_{r\in [0, t]}\big(\Vert \eta(r)\Vert_{\Cs^{2+\eps_*}}+\Vert \nabla \psi(r)\Vert_{\Bl^0}+\cH(0)\big).
\end{align*}
\end{prop}
\begin{proof}
 Recall the definition of $\A(t)$:
\[
\A(t)=\Vert \eta\Vert_{\Cs^{2+\eps_*}}+\Vert \nabla \psi\Vert_{\Bl^0}+\Vert \eta\Vert_{L^2}+E(\eta, \psi).
\]
 Proposition \ref{identity} tells us 
\[
E(\eta(t) , \psi(t))\le \int_{\xR^d}\psi G(\eta)\psi,
\]
hence
\[
\Vert \eta\Vert_{L^2}+E(\eta, \psi)\le \cH(t)=\cH(0),
\]
$\cH(t)$ being the total energy \eqref{Hamiltonian} at time $t$.
 Here, we remark that the conservation of $\cH$ follows by proving $\frac{\d}{\d t}\cH(t)=0$, which can be justified under our the regularity $\cH^s$. Therefore, Proposition \ref{energy:W} applied with $s=\sigma>\tdm+\frac d2$ yields
\bq\label{energy:W:2}
\Vert W(t)\Vert^2_{\cH^\sigma}\le \cF(P^2(t))\Vert W(0)\Vert^2_{\cH^\sigma}+\cF(P^2(t))\int_0^t \B(r) \Vert W(r)\Vert_{\cH^\sigma}^2\d r
\eq
with 
\[
P^2(t):=\sup_{r\in [0, t]}\big(\Vert \eta(r)\Vert_{\Cs^{2+\eps_*}}+\Vert \nabla \psi(r)\Vert_{\Bl^0}+\cH(0)\big).
\]
Next, as $\nabla\psi\in H^{s-1}$ with $s-1>1+\frac{d}{2}$, we can apply Lemma \ref{B-H} to have
\[
\Vert \nabla\psi\Vert_{\Bl^1}\le C\big(1+\Vert \nabla\psi\Vert_{\Cs^1}\big)\ln\big(e+\Vert \psi\Vert_{H^{\sigma}}\big)\le C\big(1+\Vert \nabla\psi\Vert_{\Cs^1}\big) \ln\big(2e+\Vert \psi\Vert^2_{H^{\sigma}}\big).
\]
Consequently,
\[
\B(r)\le C\big(1+\Vert \nabla\psi(r)\Vert_{\Cs^1}+\Vert \eta(r)\Vert_{\Cs^{2+\eps_*}}\big)\ln\big(2e+\Vert W(r)\Vert^2_{\cH^{\sigma}}\big).
\]
In view of \eqref{energy:W:2}, this implies
\begin{align*}
\Vert W(t)\Vert^2_{\cH^\sigma}&\le \cF(P^2(t))\Vert W(0)\Vert^2_{\cH^\sigma}+\\
&\cF(P^2(t))\int_0^t  Q(r)\ln\big(2e+\Vert W(r)\Vert^2_{\cH^{\sigma}}\big)\Vert W(r)\Vert_{\cH^s}^2\d r
\end{align*}
with $Q(r):=1+\Vert \nabla\psi(r)\Vert_{\Cs^1}+\Vert \eta(r)\Vert_{\Cs^{2+\eps_*}}$.
Finally, using a Gr\"onwall type argument as in \cite{BKM} we conclude that
\[
\Vert W(t)\Vert^2_{\cH^\sigma}\le \cF(P^2(t))\big(\Vert W(0)\Vert^2_{\cH^\sigma}+2e\big)\exp\big(e^{\cF(P^2(t))\int_0^tQ(r)dr}\big)-2e.
\]
\end{proof}
\begin{rema}\label{rema:expo}
Using \eqref{energy:W:2} and Gr\"onwall's lemma we obtain the exponential bound
\[
\Vert W(t)\Vert^2_{\cH^\sigma}\le \cF(P^2(t))\Vert W(0)\Vert^2_{\cH^\sigma}\exp\big(\cF(P^2(t))\int_0^t\B(r)dr\big)
\]
provided $\sigma>\tdm+\frac{d}{2}$ only.
\end{rema}
\begin{theo}\label{theo:blowup}
	Let~$d\geq 1, h>0$, and $\sigma>2+\frac{d}{2}$. Let
\[
(\eta_0, \psi_0)\in  H^{\sigma+\mez}\times H^\sigma,\quad \dist(\eta_0, \Gamma)>h>0.
\]
Let $T^*=T^*(\eta_0, \psi_0, \sigma, h)$ be the maximal time of existence defined by \eqref{def:T*} and
 \bq
(\eta,\psi)\in L^{\infty}\lp[0, T^*); H^{\sigma+\mez}\times H^\sigma\rp
\eq
be the maximal solution of ~\eqref{ww} with prescribed data $(\eta_0, \psi_0)$. If ~$T^*$ is finite, then for all $\eps>0$, it holds that
\bq\label{criterion:blowup} P_\eps(T^*)+\int_0^{T^*}Q_\eps(t)dt+\frac{1}{h(T^*)}=+\infty,
\eq
where 
\[
\begin{aligned}
&P_\eps(T^*)=\sup_{t\in [0, T^*)}\Vert \eta(t)\Vert_{\Cs^{2+\eps}}+\Vert \nabla \psi(t)\Vert_{\Bl^0},\\
&Q_\eps(t)=\Vert \eta(t)\Vert_{\Cs^{\frac{5}{2}+\eps}}+\Vert \nabla\psi(t)\Vert_{\Cs^1},\\
&h(T^*)=\inf_{t\in [0, T^*)}\dist(\eta(t),\Gamma).
\end{aligned}
\]
Consequently, if $T^*$ is finite then for all $\eps>0$
\bq\label{criter2}
 P^0_\eps(T^*)+\int_0^{T^*}Q^0_\eps(t)dt+\frac{1}{h(T^*)}=+\infty,
\eq
where
\[
\begin{aligned}
&P^0_\eps(T)=\sup_{t\in [0, T]}\Vert \eta(t)\Vert_{\Cs^{2+\eps}}+\Vert (V, B)(t)\Vert_{\Bl^0},\\
&Q^0_\eps(t)=\Vert \eta(t)\Vert_{\Cs^{\frac{5}{2}+\eps}}+\Vert (V, B)(t)\Vert_{\Cs^1}.
\end{aligned}
\]
\end{theo}
\begin{proof}
Suppose that $T^*<+\infty$ and for some $\eps>0$
\[
K:=P_\eps(T^*)+\int_0^{T^*}Q_\eps(t)dt+\frac{1}{h(T^*)}<+\infty.
\]
Let $T\in [0, T^*)$ be arbitrary then $h(T)\ge h(T^*)\ge 1/K>0$. As $\sigma>2+\frac d2$, it follows from Proposition \ref{prop:grow} that
\bq\label{estforblowup}
 M_{\sigma,T}\le \cF\lp  M_{\sigma, 0}+\cH(0)+P_\eps(T) +\int_0^TQ_\eps(t)dt\rp=:L
\eq
for some increasing function $\cF:\xR^+\to \xR^+$ depending on $1/K$. On the other hand, from the a priori estimate in Proposition $5.2$, \cite{ABZ1} we deduce that the existence time for local solutions can be chosen uniformly for data lying in a bounded subset of $H^{\sigma+\mez}\times H^\sigma$ and satisfy uniformly the separation condition $(H_0)$. In particular, call $T_1$ be the time of existence for data in the ball $B(0, L)$ of  $H^{\sigma+\mez}\times H^\sigma$ whose surface is away from the bottom a distance (at least) $1/K$. Choosing $\eta(T^*-\frac{T_1}{2})$ as such a datum we can prolong the solution up to the time $T^*+\frac{T_1}{2}$. This contradicts the maximality of $T^*$ and thus the blow-up criterion \eqref{criterion:blowup} is proved.\\
 Finally, \eqref{criter2} is a consequence of \eqref{criterion:blowup} and the facts that
\bq\label{psi-VB}
\begin{aligned}
&\Vert \nabla \psi\Vert_{\Bl^0}\le C\Vert V\Vert_{\Bl^0}+C\Vert B\Vert_{\Bl^0}\Vert \nabla \eta\Vert_{\Cs^{\eps_*}},\\
&\Vert \nabla \psi\Vert_{\Bl^1}\le C\Vert V\Vert_{\Bl^1}+C\Vert B\Vert_{\Bl^1}\Vert \nabla \eta\Vert_{\Cs^{1+\eps_*}}.
\end{aligned}
\eq
\end{proof}
\hk Now we give the proof of Corollary \ref{intro:prop:regularity} which is stated again for the reader's convenience.
\begin{coro}\label{theo:regularity}
Let $T\in (0, +\infty)$ and $(\eta, \psi)$ be a distributional solution to system \eqref{ww} on the time interval $[0, T]$  such that  $\inf_{[0, T]}\dist(\eta(t), \Gamma)>0.$ Then the following property holds: if one knows a priori that for some $\eps_0>0$
\bq\label{condition:reg}
\sup_{[0, T]}\| (\eta(t), \nabla\psi(t))\|_{\Cs^{\frac{5}{2}+\eps_0}\times \Cs^{1}}<+\infty
\eq
 then $(\eta(0), \psi(0))\in  H^\infty(\xR^d)^2$ implies that $(\eta, \psi)\in  L^\infty([0, T]; H^\infty(\xR^d))^2$.
\end{coro}
\begin{proof}
Take $\sigma>2+\frac d2$ be arbitrary, it suffices to prove that if $(\eta(0), \psi(0))\in  H^{\sigma+\mez}\times H^{\sigma}$ then $(\eta, \psi)\in  L^\infty([0, T]; H^{\sigma+\mez}\times H^{\sigma})$. Since $\sigma>2+\frac{d}{2}$, according to the Cauchy theory in \cite{ABZ1} one has a maximal solution 
\[
(\eta, \psi)\in  L^\infty([0, T_\sigma); H^{\sigma+\mez}\times H^{\sigma}).
\]
By the uniqueness statement of this  Cauchy theory, we only need to show that $T_\sigma> T$. Suppose that $T_\sigma\le T<+\infty$ we get by applying \eqref{criter2} that for all $\eps>0$
\[
\sup_{t\in [0, T_\sigma)}\Vert(\eta, \nabla\psi)(t)\Vert_{\Cs^{\frac 52+\eps}\times \Cs^{1}}+\frac{1}{h(T_\sigma)}=+\infty.
\]
On the other hand, by our assumption, $h(T')\ge h(T_\sigma)\ge h(T)>0$ (by the assumption) for all $T'<T_\sigma$, hence for all $\eps>0$,
\[
\sup_{t\in [0, T_s)}\Vert(\eta, \nabla\psi)(t)\Vert_{\Cs^{\frac 52+\eps}\times \Cs^{1}}=+\infty,
\]
which contradicts \eqref{condition:reg}.
\end{proof}

%%%%%%%%%%%%%%%%%%%%%%%
\section{Contraction of the solution map}\label{section:contraction}
Our goal in this section is to prove a contraction estimate for two solutions to \eqref{ww} in  weaker norms. This will be used in the proof of the convergence of the approximate scheme and in establishing uniqueness for the Cauchy theory in our companion paper \cite{NgPo}. On the other hand, the proof will make use of the Strichartz estimate in the same paper. To get started, we have by straightforward computations the following assertion: $(\eta, \psi)$ is a solution to system \eqref{ww} if and only if 
\[
(\partial_t+T_V\cdot\nabla+\mathcal{L})\begin{pmatrix}\eta \\ \psi\end{pmatrix} =f(\eta, \psi)
\]
with
\bq\label{defi:L}
\mathcal{L}:=
\begin{pmatrix}
I &  0\\
T_B& I
\end{pmatrix}
\begin{pmatrix}
0 &  -T_\lambda\\
T_{\l}& 0
\end{pmatrix}
\begin{pmatrix}
I &  0\\
-T_B& I
\end{pmatrix},\quad
f(\eta, \psi):=\begin{pmatrix}
I &  0\\
T_B& I
\end{pmatrix}
\begin{pmatrix}
f^1\\
f^2
\end{pmatrix}
.
\eq
where
\bq\label{f1f2}
\begin{aligned}
f^1(\eta, \psi)&=G(\eta)\psi-\lp T_{\lambda}(\psi-T_B\eta)-T_V\cdot\nabla\eta \rp,\\
f^2(\eta, \psi)&=-\mez|\nabla\psi|^2+\mez\frac{(\nabla \eta \cdot \nabla\psi + G(\eta)\psi)^2}{1+ \vert \nabla \eta \vert^2}\\
\quad & +T_V\cdot\nabla \psi-T_BT_V\cdot\nabla\eta-T_BG(\eta)\psi- H(\eta)+T_{\l}\eta-g\eta.
\end{aligned}
\eq
Assume that $(\eta_1, \psi_1)$ and $(\eta_2, \psi_2)$ are two solutions of system \eqref{ww} on $[0, T]$ and satisfy 
\[
(\eta_j, \psi_j)\in L^\infty\lp[0, T]; H^{s+\mez}\times H^s\rp\cap L^p\lp[0 ,T]; W^{r+\mez,\infty}\times W^{r,\infty}\rp,\quad j=1,2
\]
with
\[
s>\tmez+\frac d2,~\quad r>2.
\]
Assume in addition that there exists $h>0$ such that
\[
\sup_{t\in [0, T]}\dist(\eta_j(t), \Gamma)\ge h\quad j=1,2.
\]
Denote for $j=1, 2$
\bq\label{MZj}
\begin{aligned}
&M^j_{\sigma,T}=\Vert (\eta^j, \psi^j)\Vert_{L^{\infty}([0, T]; H^{\sigma+\mez}\times H^\sigma)},\\ 
&M^j_{\sigma,0}=\Vert (\eta^j, \psi^j)\arrowvert_{t=0}\Vert_{H^{\sigma+\mez}\times H^\sigma},\\
&Z^j_{r, T}=\Vert (\eta^j, \psi^j)\Vert_{L^p([0, T]; W^{r+\mez, \infty}\times W^{r,\infty})}.
\end{aligned}
\eq
Set 
\[
\delta \eta=\eta_1-\eta_2,\quad \delta\psi=\psi_1-\psi_2, \quad \delta B=B_1-B_2,\quad \delta V=V_1-V_2.
\]
Define the following quantities
\bq\label{quantities:contraction}
\begin{aligned}
&P_S(t)=\lA \delta \eta(t)\rA_{H^{s-1}}+\lA \delta \psi(t)\rA_{H^{s-\tmez}},\\
&P_H(t)=\lA \delta\eta(t)\rA_{\Cs^{r-1}}+\lA \delta\psi(t)\rA_{\Cs^{r-\tmez}},\\
&P_{S, T}=\lA P_S\rA_{L^\infty([0, T])}, \quad P_{H,T}=\lA P_S\rA_{L^p([0, T])},\\ &P(t)=P_S(t)+P_H(t), \quad P_T=P_{S,T}+P_{H,T}.
\end{aligned}
\eq
\begin{nota}
Throughout this section, we write $A\les B$ if there exists a non-decreasing function $\cF:\xR^+\to \xR^+$ such that $A\le \cF(M^1_{s, T}, M^2_{s,T})B$.
\end{nota}
\subsection{Contraction estimate for $f^2$}
Recall that we consider $B,~V$ as functions of $(\eta, \psi)$ defined by \eqref{BV}.
\begin{lemm}\label{dB,dV}
We have for a.e. $t\in [0, T]$
\[
\lA \delta B(t)\rA_{\Cs^{-\mez}}+\lA \delta V(t)\rA_{\Cs^{-\mez}}\les P(t).
\]
\end{lemm}
\begin{proof}
Assume the estimate for $\delta B$.  We have
\[
\delta V=\nabla \delta \psi-\delta B\nabla \eta_1-B_2\nabla\delta\eta.
\]
Obviously, 
\[
\|\nabla \delta \psi(t)\|_{\Cs^{-\mez}}\le \| \delta \psi(t)\|_{\Cs^{\mez}}\le \| \delta \psi(t)\|_{\C^{r-\tmez}}\le P_H(t).
\]
On the other hand, 
\[
\lA B_2\nabla\delta\eta(t)\rA_{\Cs^{-\mez}}\le \lA B_2\nabla\delta\eta(t)\rA_{L^\infty}\les \lA \delta\eta(t)\rA_{W^{1,\infty}}\les P_H(t)
\]
From the product rule \eqref{tame:H<0} for negative H\"older indices , we deduce
\[
\lA \delta B\nabla \eta_1(t)\rA_{\Cs^{-\mez}}\les \lA \delta B(t)\rA_{\Cs^{-\mez}}\lA \nabla\eta_1(t)\rA_{\Cs^{\mez+\eps}}\les P(t)
\]
with $\eps>0$ sufficiently small so that $\lA \nabla\eta_1(t)\rA_{\Cs^{\mez+\eps}}\le C\lA \nabla\eta_1(t)\rA_{H^{s-\mez}}.$
Therefore, we are left with the estimate for $\delta B$, for which we use again the formula \eqref{decompose:B}
\[
B=K(\nabla\eta)\cdot\nabla\psi+L(\nabla\eta)G(\eta)\psi+G(\eta)\psi
\]
where~$K$ and~$L$ are smooth functions, vanishing at $0$. Observe that $G(\eta)$ has order $1$, hence these three terms have the same regularity structure. We give the proof for the second one since it is a product with the  Dirichlet-Neumann operator
\begin{multline*}
L(\nabla\eta_1)G(\eta_1)\psi_1-L(\nabla\eta_2)G(\eta_2)\psi_2=[L(\nabla\eta_1)-L(\nabla\eta_2)]G(\eta_1)\psi_1\\
+L(\nabla\eta_2)[G(\eta_1)\psi_1-G(\eta_2)\psi_2].
\end{multline*}
Let us consider the more difficult term $L(\nabla\eta_2)[G(\eta_1)\psi_1-G(\eta_2)\psi_2$. By means of the product rule \eqref{tame:H<0} it suffices to estimate the $\Cs^{-\mez}$ norm of 
\[
G(\eta_1)\psi_1-G(\eta_2)\psi_2=G(\eta_1)\delta\psi-[G(\eta_1)-G(\eta_2)]\psi_2.
\]
The H\"older estimate \eqref{est:DN:C-mez:l} together with Remark \ref{rema:E} implies  that 
\[
\lA G(\eta_1)\delta\psi\rA_{\Cs^{-\mez}}\les \Vert \delta \psi\Vert_{\Cs^\mez}+\Vert \delta\psi\Vert_{H^\mez}\les \Vert \delta \psi\Vert_{\Cs^\mez}+\Vert \delta\psi\Vert_{H^{s-\tdm}},
\]
where we have used the fact that $s>2$.\\
 For the second term on the right-hand side, we apply Proposition \ref{shape} to have
\bq\label{diff:G}
 [G(\eta_1)-G(\eta_2)]\psi_2=\int_0^1 \left\{ G(\widetilde\eta(m))\big(\widetilde B\delta\eta(t)\big)+\div\big( \widetilde V(m)\delta\eta(t)\big) \right\}dm
\eq
where $\widetilde \eta(m)=\eta_1+m\delta\eta,~\widetilde B(m)=B(\widetilde \eta(m), \psi_2), ~\widetilde V(m)=V(\widetilde\eta(m), \psi_2)$. Theorem \ref{DN:ABZ} applied with $\sigma=s-2$ then yields
\bq\label{diff:DN:s-2}
\Vert [G(\eta_1)-G(\eta_2)]\psi_2\Vert_{H^{s-2}}\les \Vert \delta \eta\Vert_{H^{s-1}}.
\eq
The embedding $H^{s-2}\hookrightarrow \Cs^{-\mez}$ then concludes the proof.
\end{proof}
We introduce the following notation.
\begin{nota}
Let $f:\xR^d\to \xC^d$ be a  function of $u$, we set 
\[
\d_uf(u)\dot u=\lim_{\eps\to 0}\{f(u+\eps \dot u)-f(u)\}.
\]
\end{nota}
\begin{prop}\label{contraction:f2}
With $f^2$ defined in \eqref{f1f2}, it holds for a.e. $t\in [0, T]$ that
\[
\lA f^2(\eta_1,\psi_1)(t)-f^2(\eta_2, \psi_2)(t)\rA_{H^{s-\tmez}}\les P(t).
\]
\end{prop}
\begin{proof}
It suffices to prove that 
\bq\label{df2}
\lA d_\eta f^2(\eta, \psi)\dot\eta+ \d_\psi f^2(\eta, \psi)\dot\psi \rA_{H^{s-\tmez}}\les \|\dot\eta\|_{H^{s-1}}+\| \dot\eta\|_{\Cs^{r-1}}+\| \dot\psi\|_{H^{s-\tmez}}+\| \dot\psi\|_{\Cs^{r-\tmez}}.
\eq
We have $f^2(\eta, \psi)=I_1+I_2+I_3$ with
\begin{align*}
I_1&:=H(\eta)+T_l\eta,\\
I_2&:=-\mez|\nabla\psi|^2+\mez\frac{(\nabla \eta \cdot \nabla\psi + G(\eta)\psi)^2}{1+ \vert \nabla \eta \vert^2}+T_V\cdot\nabla \psi-T_BT_V\cdot\nabla\eta-T_BG(\eta)\psi,\\
I_3&:=-g\eta.
\end{align*}
Observe that $\d_\psi I_1=\d_\psi I_3=0$. The estimate for $\d_\eta I_3\dot\eta=-g\dot\eta$ is obvious. Observe that $I_1$ and $I_2$ are the remainder of the paralinearization of nonlinear functions in Lemmas \ref{lem:paracurv} and  \ref{lem:paraquad}, respectively. Putting $f(x)=x(1+|x|^2)^{-1/2},~x\in \xR^d$, we have $-H(\eta)=\div f(\nabla \eta)$. Since
\[
-\d_\eta f(\nabla \eta)\dot\eta=f'(\nabla\eta)\nabla\dot\eta,
\]
it follows that 
\[
-\d_\eta H(\eta)\dot \eta=\div(f'(\nabla\eta)\nabla)\dot\eta+f'(\nabla\eta)\nabla\cdot\nabla\dot\eta.
\]
U sing the Bony decomposition we get 
\[
-\d_\eta H(\eta)\dot \eta=T_{i\div(f'(\nabla\eta)\xi)}\dot\eta+T_{-f'(\nabla\eta)\xi\cdot\xi}\dot\eta+R=T_{-\l}\dot\eta+R
\]
with $\| R\|_{H^{s-3/2}}\les  \lA \dot\eta\rA_{H^{s-1}}+\lA \dot\eta\rA_{\Cs^{r-1}}$.  The Leibnitz rule then implies
\[
\d_\eta I_1(\eta)\dot\eta=T_{\dot \l}\eta+R,\quad \dot \l:=d_\eta \l\dot\eta,
\]
 so we only need to show that $\| T_{\dot \l}\eta\|_{H^{s-3/2}}\les \lA \dot\eta\rA_{H^{s-1}}+\lA \dot\eta\rA_{\Cs^{r-1}}$. Indeed, observe that $\dot \l$ is of the form 
\[
\dot \l=F_1(\nabla\eta, \xi)\nabla\dot\eta+F_2(\nabla\eta, \xi)\nabla^2\dot\eta+F_3(\nabla\eta, \xi)\nabla\dot\eta\nabla^2\eta=:\sum_{j=1}^3G_j(x, \xi),
\]
where $F_j,~j=1,2,3$ are smooth in $\xR^d\times \xR^d\setminus\{0\}$, $F_1$ is homogeneous of order $2$ in $\xi$ and $F_2,~F_3$ are homogeneous of order $1$ in $\xi$. By virtue of Theorem \ref{theo:sc} $(i)$ and Proposition \ref{regu<0} we see that to obtain the desired bound for $\| T_{\dot \l}\eta\|_{H^{s-3/2}}$ it suffices to prove for $j=1,2, 3$
\[
\sup_{|\xi|=1}\|\partial_\xi^\alpha G_1(\cdot, \xi)\|_{L^\infty}+\sup_{|\xi|=1}\|\partial_\xi^\alpha G_j(\cdot, \xi)\|_{\Cs^{-1}}\les C_\alpha\lA \dot\eta\rA_{\Cs^{r-1}}\quad\forall \alpha\in \xN^d.
\]
This is true because (assuming without loss of generality that $F_j(0, \xi)=0$, for all $\xi$) uniformly in $|\xi|=1$,
\begin{align*}
&\lA F_1(\nabla\eta)\nabla\dot\eta\rA_{L^\infty}\les \lA \dot\eta\rA_{W^{1,\infty}}\les \lA \dot\eta\rA_{\Cs^{r-1}},\\
&\lA F_2(\nabla\eta)\nabla^2\dot\eta\rA_{\Cs^{-1}}\les \lA F_2(\nabla\eta)\rA_{\Cs^{1+\eps}}\lA \nabla^2\dot\eta\rA_{\Cs^{-1}}\les \lA \dot\eta\rA_{\Cs^{r-1}}\quad \eps\in(0, s-\tdm-\frac{d}{2}),\\
&\lA F_3(\nabla\eta)\nabla\dot\eta\nabla^2\eta\rA_{\Cs^{-1}}\les \lA F_3(\nabla\eta)\nabla\dot\eta\nabla^2\eta\rA_{L^\infty}\les \lA \dot\eta\rA_{W^{1,\infty}}\les \lA \dot\eta\rA_{\Cs^{r-1}}.
\end{align*}
We have shown the desired estimate for $I_1$. By inspecting the proof of Lemma  \ref{lem:paraquad}, the estimate for the $H^{s-3/2}$-norm of  $\d_\eta I_2\dot\eta+\d_\psi I_2\dot\psi$ can be obtained by the same method.
\end{proof}
\subsection{Contraction estimate for $f^1$}
Our goal in this subsection is to derive the following estimate.
\begin{prop}\label{contraction:f1}
With $f^1$ defined as in \eqref{f1f2}, it holds for a.e. $t\in [0, T]$ that
\[
\lA f^1(\eta_1,\psi_1)(t)-f^1(\eta_2, \psi_2)(t)\rA_{H^{s-1}}\les P_H(t)+P_S(t)Q(t)
\]
with 
\bq\label{Q(t)}
Q(t):=1+\sum_{j=1}^2\lA \eta_j(t)\rA_{\Cs^{r+\mez}}+\sum_{j=1}^2\lA \psi_j(t)\rA_{\Cs^r}.
\eq
\end{prop}
The key point is that the preceding estimate is tame with respect to the highest H\"older norms. Proposition \ref{contraction:f1} will be a consequence of 
\bq\label{est:f:deta}
\|\d_\eta f^1(\eta, \psi)\dot\eta\|_{H^{s-1}} \les \| \dot\eta\|_{H^{s-1}}\big(1+\lA \eta\rA_{\Cs^{r+\mez}}+\lA \psi\rA_{\Cs^r} \big)+\lA \dot\eta\rA_{\Cs^{r-1}}
\eq
for all $\dot \eta\in H^{s+\mez}\cap \Cs^{r+\mez}$, and
\bq\label{est:f:dpsi}
\| \d_\psi f^1(\eta, \psi)\dot\psi\|_{H^{s-1}}\les \| \dot\psi\|_{H^{s-\tmez}}\big(1+\lA \eta\rA_{\Cs^{r+\mez}} \big)+\| \dot\psi\|_{\Cs^{r-\tmez}}
\eq
for all $\dot\psi\in H^s\cap \Cs^r$.
\begin{lemm}
The estimate \eqref{est:f:deta} holds.
\end{lemm}
\begin{proof}
From the  definition of $f^1$ and Proposition \ref{shape} we have
\begin{align*}
\d_\eta f^1(\eta, \psi)\dot\eta&=-G(B\dot\eta)-\div(V\dot\eta)\\
&\quad -\left\{ T_{\dot\lambda}(\psi-T_B\eta)-T_{\lambda}T_{\dot B}\eta-T_{\lambda}T_B\dot\eta-T_{\dot V}\nabla\eta-T_V\nabla\dot\eta \right\}\\
&=\sum_{j=1}^5 I_j,
\end{align*}
where $\dot B:=\d_\eta B(\eta, \psi)\dot\eta$ (similarly for $\dot V,~\dot\lambda$)  and 
\begin{align*}
&I_1:=T_{\dot V}\nabla\eta,\quad I_2:=-V\nabla\dot\eta+T_V\nabla\dot\eta,\quad I_3:=-T_{\dot\lambda}(\psi-T_B\eta),\\
& I_4:=T_{\lambda}T_{\dot B}\eta,\quad I_5:=-G(B\dot\eta)-(\div V)\dot\eta+T_{\lambda}T_B\dot\eta.
\end{align*}
1. For $I_2$ we write $I_2=-T_{\nabla\dot\eta}V-R(\nabla\dot\eta,  V)$ and use  \eqref{Bony1}, \eqref{pest3} to estimate 
\[
\lA I_2\rA_{H^{s-1}}\les \lA V\rA_{H^{s-1}}\|\nabla\dot\eta\|_{L^\infty}\les \lA \dot\eta\rA_{\Cs^{r-1}}.
\]
2. Let us study $\dot B$ and $\dot V$. For the former, the only nontrivial point is 
\bq\label{dshape}
\d_\eta [G(\eta)\psi]\dot \eta=-G(\eta)(B\dot\eta)-\div(V\dot\eta).
\eq
It holds that
\[
\lA \d_\eta [G(\eta)\psi]\dot \eta\rA_{H^{s-2}}\les \lA \dot\eta\rA_{H^{s-1}}+\lA V\dot\eta\rA_{H^{s-1}}\les \lA \dot\eta\rA_{H^{s-1}}.
\]
Therefore, $\| \dot B\|_{H^{s-2}}\les \|\dot\eta\|_{H^{s-1}}$. This, together with the relation  $V=\nabla\psi-B\nabla\eta$, implies
\[
\| \dot B\|_{H^{s-2}}+\| \dot V\|_{H^{s-2}}\les \| \dot \eta\|_{H^{s-1}}.
\]
A a consequence, the paraproduct rule \eqref{boundpara} gives (keep in mind that $s>\tdm+\frac d2$)
\[
\lA I_1\rA_{H^{s-1}}\les \| \dot V\|_{H^{s-2}}\| \nabla \eta\|_{H^{s-\mez}}\les \| \dot\eta\|_{H^{s-1}}.
\]
Similarly, 
\[
\lA I_4\rA_{H^{s-1}}\les \|T_{\dot B}\eta\|_{H^s}\les \| \dot B\|_{H^{s-2}}\|\eta\|_{H^{s+\mez}}\les \|\dot \eta\|_{H^{s-1}}.
\]
3. For $I_3$ one estimates $\dot \lambda$ exactly as for $\dot l$ in the proof of Proposition \ref{contraction:f2}.\\
4. For $I_5$ we follow \cite{ABZ1} using the following cancellation in Lemma $2.12$, \cite{ABZ1} whose proof applies equally at our regularity level:
\[
G(\eta)B=-\div V+R,~\|R\|_{H^{s-1}}\les 1.
\]
On the other hand, applying Proposition $3.13$ in \cite{ABZ3} with $\eps=\mez$ $\sigma=s-\mez$ we obtain the following paralinearizations 
\[
G(\eta)(B\dot\eta)=T_{\lambda^{(1)}}B\dot\eta+F(\eta, B\dot\eta),\quad G(\eta)(B)=T_{\lambda^{(1)}}B+F(\eta, B)
\]
with 
\[
\lA F(\eta, B\dot\eta)\rA_{H^{s-1}}\les \lA \dot\eta\rA_{H^{s-1}},\quad \lA F(\eta, B)\rA_{H^{s-1}}\les 1.
\]
Then plugging these paralinearizations into the expression of $I_5$ gives  (see \cite{ABZ1} pages $482-483$ for details) $I_5=J_1+J_2$ with
\begin{gather*}
J_1=-T_{\lambda^{(1)}}\left(B\dot\eta-T_B\dot\eta-T_{\dot\eta}B\right),\\
J_2=T_{\lambda^{(0)}}T_B\dot\eta+[T_{\dot\eta}, T_{\lambda^{(1)}}]B+T_{\dot\eta}F(\eta, B)+(\dot\eta-T_{\dot\eta})\div V-F(\eta, B\dot\eta)-T_{\dot\eta}R.
\end{gather*}
Using \eqref{Bony2} we estimate
\[
\lA J_1\rA_{H^{s-1}}\les \lA R(B, \dot\eta)\rA_{H^s}\les \lA \dot\eta\rA_{H^{s-1}}\lA B\rA_{\Cs^{1}}\les  \| \dot\eta\|_{H^{s-1}}\Big(1+\lA \eta\rA_{\Cs^{r+\mez}}+\lA \psi\rA_{\Cs^r} \Big). 
\]
For $J_2$ we only need to take care of the commutator $[T_{\dot\eta}, T_{\lambda^{(1)}}]B$. Since $\| B\|_{H^{s-1}}\les 1$ it suffices to prove that $[T_{\dot\eta}, T_{\lambda^{(1)}}]$ has order $0$ and map  $H^{s-1}\to H^{s-1}$ with norm bounded by the right hand side of \eqref{est:f:deta}. This is in turn a consequence of Theorem \ref{theo:sc} (ii) and the fact that $r-1>1$.
\end{proof}
Finally, we prove
\begin{lemm}\label{lemm:est:f:dpsi}
The estimate \eqref{est:f:dpsi} holds.
\end{lemm}
Write $B=B(\eta, \psi),~V=V(\eta, \psi)$. Since $G(\eta)\psi$ is linear with respect to $\psi$ we get
\[
\d_\psi f^1(\eta, \psi)\dot\psi=G(\eta)\dot\psi -T_\lambda(\dot\psi-T_{B(\eta, \dot\psi)}\eta)-T_{V(\eta, \dot\psi)}\cdot\nabla\eta=:R(\eta, \dot\psi).
\]
Estimate \eqref{est:f:dpsi} means that $R$ has order $-1/2$ with respect to $\dot\psi$ and map  $H^{s-3/2}$ to $H^{s-1}$. In fact,  Proposition \ref{prop:paralin:DN} shows that $R$ maps $H^s$ to $H^{s+1/2}$. Here, we will follow the proof of Proposition \ref{prop:paralin:DN} except that the good unknown will not be invoked. Lemma \ref{lemm:est:f:dpsi} is a consequence of the following.
\begin{lemm}
Let $d\ge 1,~h>0$ and 
\[
s>\tmez+\frac d2,\quad r>2.
\]
Then there exists a nondecreasing function $\cF:\xR^+\times \xR^+\to \xR^+$ such that for any $\eta \in H^{s+\mez}$ satisfying $\dist(\eta, \Gamma)\ge h>0$ and $\psi\in H^s\cap C^r$, there holds
\bq\label{DN:low}
\begin{aligned}
&\lA G(\eta)\psi-T_{\lambda}(\psi-T_B\eta)-T_V\cdot\nabla \eta\rA_{H^{s-1}}\\
&\quad \le \cF\big(\lA(\eta, \psi)\rA_{H^{s+\mez}\times H^s}\big)\lB \| \psi\|_{H^{s-\tmez}}\big(1+\lA \eta\rA_{\Cs^{r+\mez}}+\lA \psi\rA_{\Cs^r} \big)+\| \psi\|_{\Cs^{r-\tmez}}\rB.
\end{aligned}
\eq
\end{lemm}
\begin{proof}
We first apply Theorem \ref{theo:sc} (i) to have
\[
\lA T_{\lambda} T_B\eta\rA_{H^{s-1}}+\lA T_V\cdot\nabla\eta\rA_{H^{s-1}}\le \cF\big(\lA(\eta, \psi)\rA_{H^{s+\mez}\times H^s}\big)\lB\lA B\rA_{\Cs^{-\mez}}+\lA V\rA_{\Cs^{-\mez}}\rB.
\]
On the other hand, as in Lemma \ref{dB,dV} it holds that
\[
 \lA B\rA_{\Cs^{-\mez}}+\lA V\rA_{\Cs^{-\mez}}\le \cF\big(\lA(\eta, \psi)\rA_{H^{s+\mez}\times H^s}\big)\lB \| \psi\|_{H^{s-\tmez}}+\| \psi\|_{\Cs^{r-\tmez}}\rB.
\]
Therefore, the proof of \eqref{DN:low} reduces to showing 
\bq\label{DN:low1}
\begin{aligned}
&\lA G(\eta)\psi-T_{\lambda}\psi\rA_{H^{s-1}}\\
&\quad\le \cF\big(\lA(\eta, \psi)\rA_{H^{s+\mez}\times H^s}\big)\lB \| \psi\|_{H^{s-\tmez}}\big(1+\lA \eta\rA_{\Cs^{r+\mez}}+\lA \psi\rA_{\Cs^r} \big)+\| \psi\|_{\Cs^{r-\tmez}}\rB.
\end{aligned}
\eq
To this end, let $\phi$ be the solution to \eqref{elliptic:DN}. Let $v$ be as in \eqref{defi:v}, which satisfies equation~\eqref{eq:v}.\\
  Let $z_0\in (-1, 0)$ and denote $J=[z_0, 0]$. Let $\Pi$ denote the right-hand side of \eqref{DN:low1}. Again, to alleviate notations, we will write $A\les B$ provided 
\[
A\le \cF\big(\lA(\eta, \psi)\rA_{H^{s+\mez}\times H^s}\big)B.
\]
 According to Proposition \ref{est:v:C-mez:a} and Remark \ref{rema:E} (notice that $s-\tdm>\mez$), there holds
\bq\label{v:C}
\lA \nabla_{x,z}v\rA_{C(J; \Cs^{-\mez})}\les \| \psi\|_{H^{s-\tmez}}+\| \psi\|_{\Cs^{r-\tmez}}.
\eq
On the other hand, applying Proposition \ref{elliptic:ABZ} with $\sigma=s-\frac 52\ge -\mez$ gives
\bq\label{v:S}
\lA \nabla_{x,z}v\rA_{X^{s-\frac 52}(I)}\les \Vert \psi\Vert_{H^{s-\tdm}}.
\eq
 We will write $g_1\sim_E g_2$ if the $E$-norm of $g_1-g_2$ is bounded by $\Pi$. As in Proposition \ref{paralin:eq:u}, we set
\[
P:= \partial_z^2+T_\alpha\Delta_x +T_\beta\cdot\nabla_x\partial_z-T_{\gamma}\partial_z.
\]
 In view of equation \eqref{eq:v}, we have 
\[
0=(\partial_z^2+\alpha\Delta_x+\beta\cdot\nabla_x\partial_z-\gamma\partial_z)v=Pv+Qv
\]
with 
\begin{align*}
Qv&:=[T_{\Delta v}(\alpha-h^2)+R(\Delta v, \alpha-h^2)]+[(h^2-T_{h^2})\Delta_xv]+\\
&\quad [T_{\nabla\partial_z v}\beta+R(\nabla\partial_zv, \beta)]-[T_{\partial_z v}\gamma+R(\partial_z v, \gamma)].
\end{align*}
For the first bracket, we use \eqref{pest1}, \eqref{Bony2} and \eqref{v:C} to have
\[
\lA T_{\Delta v}(\alpha-h^2)\rA_{L^2H^{s-\tmez}}+\lA R(\Delta v, \alpha-h^2)\rA_{L^2H^{s-\tmez}}\les \lA \Delta v\rA_{L^\infty \Cs^{-\frac 32}}\lA \alpha-h^2\rA_{L^2H^s}\les \Pi.
\]
The other terms can be estimated along the same lines. Consequently,
\[
Pu\sim_{L^2(J; H^{s-\tdm})} 0.
\]
 Next, with two symbols $a, A$ defined in Lemma \ref{factor2}, the proof of Lemma \ref{factor2} shows that
\[
Pv=(\partial_z-T_a)(\partial_z-T_A)v+(T_aT_A-T_{\alpha}\Delta_x)v.
\]
According to the symbolic estimate \eqref{TaTA} and Theorem \ref{theo:sc} (ii), $(T_aT_A-T_{\alpha}\Delta_x)$ has order $\mez$ and 
\[
\Vert (T_aT_A-T_{\alpha}\Delta_x)v\Vert_{L^2H^{s-\tdm})}\les (1+\Vert \eta\Vert_{\Cs^{\frac 52}})\Vert \nabla v\Vert_{L^2H^{s-2}}\les (1+\Vert \eta\Vert_{\Cs^{\frac 52}})\Vert \psi\Vert_{L^2H^{s-\tdm}},
\]
owing to \eqref{v:S}. We have proved that
\[
(\partial_z-T_a)(\partial_z-T_A)v\sim_{L^2(J; H^{s-\tdm})} 0,
\]
which implies
\[
(\partial_z-T_a)(\partial_z-T_A)v\sim_{Y^{s-1}(J)} 0.
\]
With this result, one can follow exactly Step 1. of the proof of Proposition \ref{prop:paralin:DN} and obtain for some $I=[z_1, 0)$, $z_1\in (z_0, 0)$ that
\bq\label{gain:DN:low}
\lA \partial_zv-T_Av\rA_{X^{s-1}(I)}\les \Pi.
\eq
This allow us  to replace the normal derivative $\partial_zv$ with the "tangential derivative" $T_Av$, leaving a term $\sim_{X^{s-1}(I)} 0$.  Therefore, we deduce by using  the Bony decomposition and the estimates \eqref{v:C}-\eqref{v:S} that
\begin{align*}
\frac{1+\la\nabla\rho\ra^2}{\partial_z\rho}\partial_zv-\nabla\rho\cdot\nabla v&\sim_{X^{s-1}(I)} T_{\frac{1+\la\nabla\rho\ra^2}{\partial_z\rho}}\partial_zv-T_{\nabla \rho}\nabla v\\
&\sim_{X^{s-1}(I)} T_{\frac{1+\la\nabla\rho\ra^2}{\partial_z\rho}}T_{A}v-T_{\nabla\rho}\nabla v\\
&\sim_{X^{s-1}(I)}\sim T_{\frac{1+\la\nabla\rho\ra^2}{\partial_z\rho}A}v-T_{\nabla \rho}\nabla v\\
&\sim_{X^{s-1}(I)} T_{\Lambda}v
\end{align*}
with $\Lambda=\frac{1+\la\nabla\rho\ra^2}{\partial_z\rho}A-i\nabla\eta\cdot\xi$ satisfying  $\Lambda\arrowvert_{z=0}=\lambda$.  The proof of \eqref{DN:low1} is complete.
\end{proof}
\subsection{Contraction estimate for the solution map}
In views of the notations \eqref{f1f2}, \eqref{quantities:contraction} and \eqref{Q(t)} , we have proved in subsections~$5.1$, $5.2$ the following result for a.e. $t\in [0, T]$,
\[
\lA f(\eta_1, \psi_1)(t)-f(\eta_2, \psi_2)(t)\rA_{H^{s-1}\times H^{s-\tmez}}\le  \cF\lp M^1_{s, T}, M^2_{s,T}\rp\lp P_H(t)+P_S(t)Q(t)\rp.
\]
Consequently, this together with Lemma \ref{dB,dV} implies that the difference of two solutions satisfies
\bq\label{eq:difference}
(\partial_t+T_{V_1}\cdot\nabla+\mathcal{L}_1)\begin{pmatrix}\delta\eta \\ \delta\psi\end{pmatrix} =\begin{pmatrix}g_1 \\ g_2\end{pmatrix} 
\eq
\bq\label{contraction:g1g2}
\lA (g_1(t), g_2(t))\rA_{H^{s-1}\times H^{s-\tmez}}\le  \cF\lp M^1_{s, T}, M^2_{s,T}\rp\Big( P_H(t)+P_S(t)Q(t)\Big).
\eq
\subsubsection{Symmetrization}
Now, as in Section \ref{section:sym} we symmetrize \eqref{eq:difference} using the symmetrizer 
\[
S_1=
\begin{pmatrix}
T_{p_1} &0\\
0&T_{q_1}
\end{pmatrix}
\begin{pmatrix}
I &0\\
-T_{B_1}&I
\end{pmatrix}.
\]
{\it The dispersive part $\mathcal{L}$}. Recall the Definition \ref{equi:operators} on the equivalence of two families of operators $A(t)$ and $B(t)$,~$t\in [0, T]$: $A\sim B$ if
\[
\lA A(t)-B(t)\rA_{H^{\mu}\to H^{\mu-m+\tdm}}\le \cF\lp\lA \eta(t)\rA_{H^{s+\mez}}\rp\lp1+\lA \eta(t)\rA_{C^{r+\mez}}\rp.
\]
By virtue of Proposition \ref{pq} we obtain 
\begin{align*}
&\begin{pmatrix} %%%%%%%%%%%%
T_{p_1} &0\\
0&T_{q_1}
\end{pmatrix}
\begin{pmatrix}
I &0\\
-T_{B_1}&I
\end{pmatrix}
\begin{pmatrix}
I &  0\\
T_{B_1}& I
\end{pmatrix}
\begin{pmatrix}
0 &  -T_{\lambda_1}\\
T_{\l_1}& 0
\end{pmatrix}\\
&=\begin{pmatrix} %%%%%%%%%%%%
T_{p_1} &0\\
0&T_{q_1}
\end{pmatrix}
\begin{pmatrix}
0 &  -T_{\lambda_1}\\
T_{\l_1}& 0
\end{pmatrix}
= \begin{pmatrix} %%%%%%%%%%%%
0 & -T_{p_1}T_{\lambda_1}\\
T_{q_1}T_{\l_1}& 0
\end{pmatrix}\\ 
&\sim \begin{pmatrix} %%%%%%%%%%%%
0 & -T_{\gamma_1} T_{q_1}\\
T_\gamma T_{p_1}& 0
\end{pmatrix}
=\begin{pmatrix} %%%%%%%%%%%%
0 & -T_{\gamma_1} \\
T_{\gamma_1} & 0
\end{pmatrix}
\begin{pmatrix}
T_{p_1} & 0\\
0 & T_{q_1}
\end{pmatrix}.
\end{align*}
Consequently (see \eqref{defi:L} for the definition of $\mathcal{L}$)
\[
S_1\mathcal{L}_1\sim 
\begin{pmatrix}
0 & -T_{\gamma_1} \\
T_{\gamma_1} & 0
\end{pmatrix}
\begin{pmatrix}
T_{p_1} & 0\\
0 & T_{q_1}
\end{pmatrix}
\begin{pmatrix}
I & 0\\
-T_{B_1} & I
\end{pmatrix}.
\]
Therefore, if we set 
\[
\Phi_1:=T_{p_1}\delta \eta,\quad \Phi_2:=T_{q_1}(\delta \psi-T_{B_1}\delta \eta),
\]
then $\Phi_1,~\Phi_2$ satisfy 
\[
S_1\mathcal{L}_1
\begin{pmatrix}
\delta\eta\\
\delta\psi
\end{pmatrix}
\sim 
\begin{pmatrix}
-T_{\gamma_1}\Phi_2\\
T_{\gamma_1}\Phi_1
\end{pmatrix}
,
\]
meaning that 
\[
\lA  S_1\mathcal{L}_1
\begin{pmatrix}
\delta\eta\\
\delta\psi
\end{pmatrix}
-
\begin{pmatrix}
-T_{\gamma_1}\Phi_2\\
T_{\gamma_1}\Phi_1
\end{pmatrix}
\rA_{H^{s-\tmez}}(t)
%%%%%%%%%
\le \cF\lp M_{s,T}^1, M_{s,T}^2\rp\left(1+\lA  \eta_1(t)\rA_{\Cs^{r+\mez}}\right)P_{S}(t).
\]
{\it The convection part $\partial_t+T_{V_1}\nabla$}:  one proceeds as in the proof Proposition \ref{sym:prop} and obtain
\[
S_1\lp \partial_t+T_{V_1}\cdot\nabla\rp \begin{pmatrix}\delta\eta \\ \delta\psi\end{pmatrix} 
=\lp \partial_t+T_{V_1}\cdot\nabla\rp S_1\begin{pmatrix}\delta\eta \\ \delta\psi\end{pmatrix} +R
=\lp \partial_t+T_{V_1}\cdot\nabla\rp \begin{pmatrix}\Phi_1 \\ \Phi_2\end{pmatrix} +R
\]
where the remainder $R$ verifies
\[
\lA R(t)\rA_{H^{s-\mez}\times H^{s-\tmez}}\le \cF\lp M_{s,T}^1, M_{s,T}^2\rp\left(1+\lA  \eta_1\rA_{C^{r+\mez}}+\lA  \psi_1\rA_{C^{r}}\right)P_{S}(t).
\]
A combination of two parts yields
\bq\label{contraction:eq:sym}
\left\{
\begin{aligned}
&\partial_t\Phi_1+T_{V_1}\cdot\nabla\Phi_1- T_{\gamma_1}\Phi_2=F_1+G_1,\\
&\partial_t\Phi_2+T_{V_1}\cdot\nabla\Phi_2+T_{\gamma_1}\Phi_2=F_2+G_2
\end{aligned}
\right.
\eq
where for a.e. $t\in [0, T]$, 
\bq\label{contraction:F1F2}
\begin{aligned}
\lA (F_1, F_2)\rA_{ H^{s-\tmez}\times H^{s-\tmez}}&\le \cF\lp M_{s,T}^1, M_{s,T}^2\rp\Big(1+\lA  \eta_1\rA_{C^{r+\mez}}+\lA  \psi_1\rA_{C^{r}}\Big)P_{S}(t)\\
& \le \cF\lp M^1_{s, T}, M^2_{s,T}\rp Q(t)P_S(t),
\end{aligned}
\eq
and from \eqref{eq:difference}
\[
\begin{pmatrix}G_1 \\ G_2\end{pmatrix} =\begin{pmatrix}T_{p_1}g_1 \\ T_{q_1}(g_2-T_{B_1}g_1)\end{pmatrix}.
\]
It follows from \eqref{contraction:g1g2} that $(G_1, G_2)$ also satisfy
\bq\label{contraction:G1G2}
\lA (G_1, G_2)\rA_{ H^{s-\tmez}\times H^{s-\tmez}}\le \cF\lp M^1_{s, T}, M^2_{s,T}\rp\Big(P_H(t)+P_S(t)Q(t)\Big).
\eq
\subsubsection{Contraction estimates}
Put $\Phi:=\Phi_1+i\Phi_2$, then 
\bq\label{contraction:Phi}
\partial_t\Phi+T_{V_1}\cdot\nabla\Phi+iT_{\gamma_1}\Phi=F+G:=(F_1+iF_2)+(G_1+iG_2).
\eq
We are now back to the situation of Proposition \ref{singleeq:Phi}: we shall conjugate \eqref{contraction:Phi} with an operator of order $s-3/2$ and then perform an $L^2$-energy estimate. As in \eqref{def:wp}, we choose 
\[
\wp_1=(\gamma_1^{(3/2)})^{2(s-\tmez)/3},\quad \varphi=T_{\wp_1}\Phi.
\]
After conjugating with $T_{\wp_1}$, one obtains
\bq\label{contraction:eq:varphi}
\lp \partial_t+T_{V_1}\cdot\nabla+iT_{\gamma_1}\rp\varphi=T_{\wp_1}(F+G)+H
\eq
with
\[
	H:=T_{\partial_t\wp_1}\Phi+[T_{V_1}\cdot\nabla, T_{\wp_1}]\Phi+i[T_{\gamma_1}, T_{\wp_1}]\Phi.
\]
It is easy to see as in the proof of Proposition \ref{L2varphi} that
\bq\label{contraction:H}
\begin{aligned}
\lA H(t)\rA_{H^{s-\tmez}}&\le  \cF\lp M^1_{s, T}, M^2_{s,T}\rp Q(t)\lA \Phi(t)\rA_{H^{s-\tmez}}\\
&\le \cF\lp M_{s,T}^1, M_{s,T}^2\rp Q(t)\big(\lA\ph(t)\rA_{L^2}+\lA \Phi(t)\rA_{L^2}\big),
\end{aligned}
\eq
where we have applied Lemma \ref{inverse} in the second line.\\
On the other hand, from the  estimates \eqref{contraction:F1F2}, \eqref{contraction:G1G2} for $F,~G$ we get
\bq\label{contraction:F+G}
\lA T_{\wp_1}(F+G)\rA_{L^2}\le \cF\lp M^1_{s, T}, M^2_{s,T}\rp\Big( P_H(t)+P_S(t)Q(t)\Big).
\eq
Now, multiplying both sides of \eqref{contraction:eq:varphi} by $\varphi$ and using \eqref{contraction:H}, \eqref{contraction:F+G}, \eqref{conjugate:V}, \eqref{conjugate:gamma} lead to
\[
\begin{aligned}
&\frac{\d}{\d t}\lA\ph(t)\rA_{L^2}^2\leq \cF\lp M_{s,T}^1, M_{s,T}^2\rp\times \\
&\left\{ \Big[ P_H(t)+Q(t)P_S(t)+Q(t)\lA \Phi(t)\rA_{L^2}\Big]\lA\ph(t)\rA_{L^2}
+ Q(t)\lA\ph(t)\rA^2_{L^2}\right\}.
\end{aligned}
\]
Notice that
\[
\lA \Phi(t)\rA_{L^2}\le \cF\lp M_{s,T}^1, M_{s,T}^2\rp P_S(t),\quad \int_0^T Q(t)dt\le T+ T^{\frac{1}{p'}}\big(Z^1_{r,T}+Z^2_{r,T}\big),
\]
with $\frac{1}{p'}=1-\frac{1}{p}>0$ (recall the  notation $Z^j_{r,T}$ in \eqref{MZj}). Gr\"onwall's lemma  then implies (see Notation \ref{quantities:contraction}) for all $t\le T\le 1$
\bq\label{contraction:est:energy}
\begin{aligned}
\lA \varphi(t)\rA_{L^2}&\le \cF(...)\lp \lA \varphi(0)\rA_{L^2}+\int_0^t\lb Q(m) P_S(m)+P_H(m)\rb dm\rp\\
& \le \cF(...)\lp \lA \varphi(0)\rA_{L^2}+T^{\frac{1}{p'}} \lb (1+Z^1_{r,T}+Z^2_{r,T}) P_{S,T}+P_{H,T}\rb \rp\\
 &\le \cF(...)\lp P_S(0)+T^{\frac{1}{p'}}P_T \rp
\end{aligned}
\eq
where 
\[
\cF(...)=\cF\lp M_{s,T}^1, M_{s,T}^2, Z^1_{r,T}, Z^2_{r,T}\rp.
\]
We want to show that $\Vert (\delta \eta, \delta \psi)\Vert_{H^{s-1}\times H^{s-\tdm}}$ is also controlled by the right-hand side of \eqref{contraction:est:energy}. To this end, one uses again  Proposition~\ref{inverse} to have
	\begin{gather*}
	\lA\delta\eta\rA_{H^{s-1}}\les \lA T_\wp T_p\delta\eta\rA_{L^2}+\lA\delta\eta\rA_{H^{-\mez}},\\
	\lA\delta\psi\rA_{H^{s-\tmez}}\les\lA T_\wp T_q\delta\psi\rA_{L^2}+\lA\delta\psi\rA_{H^{-\mez}}.
	\end{gather*}
Then, in view of \eqref{contraction:est:energy} it remains to estimate $\|\delta\eta\|_{H^{-\mez}}$ and $\|\delta\psi\|_{H^{-\mez}}$. Indeed, we write
\begin{align*}
\lA \delta\eta(t)\rA_{H^{-\mez}}&\le \lA \delta\eta(0)\rA_{H^{-\mez}}+\lA \delta\eta(t)-\delta\eta(0)\rA_{H^{-\mez}}\\
&\le \lA \delta\eta(0)\rA_{H^{-\mez}}+\lA \int_0^t\frac{d}{dt}\delta\eta(m)dm\rA_{H^{-\mez}} \\
&\le \lA \delta\eta(0)\rA_{H^{-\mez}}+T\sup_{t\in [0, T]}\lA \frac{d}{dt}\delta\eta(t)\rA_{H^{-\mez}}.
\end{align*}
The last term can be written as
\[
\frac{d}{dt}\delta\eta(t)=G(\eta_1(t))\psi_1(t)-G(\eta_2(t))\psi_2(t)=G(\eta_1)\delta\psi+[G(\eta_1(t))-G(\eta_2(t))]\psi_2(t).
\]
The Sobolev estimate for the Ditichlet-Neumann operator in Theorem \ref{DN:ABZ} applied with $\sigma=\mez$ gives
\[
\lA G(\eta_1)\delta\psi\rA_{H^{-\mez}}\les \lA\delta\psi\rA_{H^{\mez}}\les \lA\delta\psi\rA_{H^{s- \frac 32}}.
\]
On the other hand, according to \eqref{diff:DN:s-2}
\[
\Vert [G(\eta_1)-G(\eta_2)]\psi_2\Vert_{L^2}\le \Vert [G(\eta_1)-G(\eta_2)]\psi_2\Vert_{H^{s-2}}\les \Vert \delta \eta\Vert_{H^{s-1}}.
\]
Therefore, 
\[
\Vert \delta\eta\Vert_{H^{-\mez}}\les \Vert \delta\eta(0)\Vert_{H^{-\mez}}+TP_{S,T}.
\]
Using the second equation of \eqref{ww:1} and arguing as above, we find that 
\[
\Vert \delta\psi\Vert_{H^{-\mez}}\les \Vert \delta\psi(0)\Vert_{H^{-\mez}}+TP_{S,T}.
\]
Putting the above estimates together, we end up with 
\[
\lA (\delta\eta(t), \delta\psi(t)) \rA_{H^{s-1}\times H^{s-\tmez}} \le \cF\lp M_{s,T}^1, M_{s,T}^2, Z_{r,T}^1, Z_{r,T}^2\rp\lp P_S(0)+T^{\frac{1}{p'}}P(t)\rp,
\]
which implies (recall that we are assuming $s>\tdm+\frac d2$, $r>2$)
\bq\label{contraction:Sobolev}
P_{S,T}\le \cF\lp M_{s,T}^1, M_{s,T}^2, Z_{r,T}^1, Z_{r,T}^2\rp\lp P_S(0)+T^{\frac{1}{p'}}  P_{T}\rp.
\eq
Observe that \eqref{contraction:Sobolev} is an a priori estimate for the Sobolev norm of the difference of two solutions. To close this estimate, we seek a similar estimate for the H\"older norm, {\it i.e.}, for $P_{H, T}$. This is achieved by applying the Strichartz estimates in our companion paper \cite{NgPo} to the dispersive equation \eqref{contraction:Phi}. According to Theorem $1.1$ of \cite{NgPo}, if $u$ is a solution to the problem 
\[
\big(\partial_t+T_{V_1}\cdot\nabla+iT_{\gamma_1}\big)u=f
\]
 with $f\in L^\infty([0, T]; H^\sigma)$, $\sigma\in \xR$, then it holds that
\bq\label{Strichartz}
\lA u\rA_{L^pW^{\sigma-\frac d2+\mu}}
\le \cF(Z_{r,T}^1)\lp \lA f\rA_{L^pH^{\sigma}}+\lA u\rA_{L^\infty H^{\sigma}}\rp
\eq
where,
\bq\label{mu,p}
\begin{cases}
\mu=\frac{3}{20},~p=4\quad \text{when}~d=1,\\
\mu=\frac{3}{10},~p=2\quad \text{when}~d\ge 2.
\end{cases}
\eq
Applying this result to $u=\Phi$ with $\sigma=s-\tdm$ leads to
\begin{align*}
\lA \Phi\rA_{L^pW^{s-\tmez-\frac d2+\mu}}
&\le \cF(Z_{r,T}^1)\lp \lA F+G\rA_{L^pH^{s-\tmez}}+\lA \Phi\rA_{L^\infty H^{s-\tmez}}\rp,
\end{align*}
This, combined with \eqref{contraction:F1F2} and \eqref{contraction:G1G2}, implies for any $2<r<r'<s-\frac d2+\mu$
\begin{align*}
\lA \Phi\rA_{L^pW^{r'-\tmez}}&\le\cF\lp M_{s,T}^1, M_{s,T}^2, Z_{r,T}^1, Z_{r,T}^2\rp\lp P_T+\lA \Phi\rA_{L^\infty H^{s-\tmez}}\rp\\
&\le  \cF\lp M_{s,T}^1, M_{s,T}^2, Z_{r,T}^1, Z_{r,T}^2\rp P_T.
\end{align*}
By interpolating between $r'$ and some lower index, we gain a multiplication factor of the form $T^\delta$, $\delta>0$ on the right-hand side. Then using the symbolic calculus in Theorem \ref{theo:sc} to go back from $\Phi$ to $\delta\eta$, $\delta\psi$ we obtain
\bq\label{apriori:Holder}
P_{H, T}\le \cF\lp M_{s,T}^1, M_{s,T}^2, Z_{r,T}^1, Z_{r,T}^2\rp T^{\delta}  P_{T}.
\eq
Combining \eqref{contraction:Sobolev} and \eqref{apriori:Holder} we end up with a closed a priori estimate for the difference of two solutions of \eqref{ww} in terms of Sobolev norm and Strichartz norm: for any $T\le 1$, there holds
\[
P_T\le \cF\lp M^1_{s,T}, M^2_{s,T}, Z^1_{r,T}, Z^2_{r,T}\rp \lp P_{S}(0)+T^\delta P_T\rp.
\]
This implies $P_{T_1}\le\cF(...) P_S(0)$ for some $T_1>0$ sufficiently small, depending only on $\cF(...)$.
Then iterating this estimate between $[T_1, 2T_1],...,[T-T_1, T]$ we obtain the following result.
\begin{theo}\label{theo:contraction}
Let $(\eta_j, \psi_j)$,~$j=1,2$ be two solutions to \eqref{ww} on $I=[0, T],~0<T\le 1$ such that 
\[
(\eta_j, \psi_j)\in L^\infty(I; H^{s+\mez}(\xR^d)\times H^s(\xR^d))\cap L^p(I; W^{r+\mez}(\xR^d)\times W^{r,\infty}(\xR^d))
\]
with 
\bq\label{gain:flowmap}
s>\tmez+\frac d2,~\quad 2<r<s-\frac d2+\mu;
\eq
where $\mu,~p$ are given by \eqref{mu,p} and such that $\inf_{t\in [0, T]}\dist(\eta_j(t), \Gamma)>h>0.$ Set
\[
M^j_{s,T}:=\Vert (\eta_j, \psi_j)\Vert_{L^{\infty}([0, T]; H^{s+\mez}\times H^s)}, \quad  Z^j_{r,T}:=\Vert (\eta_j, \psi_j)\Vert_{L^p([0, T]; W^{r+\mez, \infty}\times W^{r, \infty})}.
\]
Consider the differences $\delta\eta:=\eta_1-\eta_2,~ \delta\psi:=\psi_1-\psi_2$ and their norms in Sobolev space and H\"older space:
\[
P_{T}:=\lA (\delta\eta, \delta\psi)\rA_{L^\infty(I; H^{s-1}\times H^{s-\tmez})}+\lA (\delta\eta, \delta\psi)\rA_{L^p(I; W^{r-1,\infty}\times W^{r-\tmez,\infty})}.
\]
Then there exists a non-decreasing  function $\cF:\xR^+\times \xR^+\to\xR^+$ depending only on $d,~r,~s,~p,~\mu,~h$ such that
\[
P_T\le \cF_h\lp M^1_{s,T}, M^2_{s,T}, Z^1_{r,T}, Z^2_{r,T}\rp\lA (\delta\eta, \delta\psi)\arrowvert_{t=0}\rA_{H^{s-1}\times H^{s-\tmez}}.
\]
\end{theo}
\begin{rema}\label{rema:contraction}
If the Strichartz estimate \eqref{Strichartz} had been proved with a gain of $\mu'$ derivative, $\mu'\in (0, \mez]$, then Theorem \ref{theo:contraction} would have held with $\mu=\mu'$ in \eqref{gain:flowmap}.
\end{rema}
%%%%%%%%%%%%%%%%%%%%%%%%%%%%%%%%%%%%%%%%%%%%%%%5
\section{Appendix: Paradifferential calculus and technical results}\label{Appendix}
\subsection{Paradifferential operators}
\begin{defi}\label{defi:PL}
1. (Littlewood-Paley decomposition) Let~$\kappa\in C^\infty_0({\mathbf{R}}^d)$ be such that
$$
\kappa(\theta)=1\quad \text{for }\la \theta\ra\le 1.1,\qquad
\kappa(\theta)=0\quad \text{for }\la\theta\ra\ge 1.9.
$$
Define
\begin{equation*}
\kappa_k(\theta)=\kappa(2^{-k}\theta)\quad\text{for }k\in \xZ,
\qquad \varphi_0=\kappa_0,\quad\text{ and }
\quad \varphi_k=\kappa_k-\kappa_{k-1} \quad\text{for }k\ge 1.
\end{equation*}
Given a temperate distribution $u$, we  introduce 
\begin{align*}
&S_k u=\kappa_k(D_x)u\quad\text{for}~ k\in \xZ,\\
&\Delta_0u=S_0u,\quad \Delta_k u=S_k u-S_{k-1}u\quad \text{for}~k\ge 1.
\end{align*}
 Then we have the formal dyadic partition of unity
$$
u=\sum_{k=0}^{\infty}\Delta_k u.
$$
2. (Zygmund spaces) Let $s\in \xR$ and $p,~q\in [1, \infty]$. The Besov  space~$B^s_{p, q}(\xR^d)$ is defined as the space of all the tempered distributions~$u$ satisfying
$$
\lA u\rA_{B^s_{p,q}}\defn \Vert \big(2^{js}\lA \Delta_j u\rA_{L^p(\xR^d)}\big)_{j=0}^\infty\Vert_{\ell^q}<+\infty.
$$
When $p=q=\infty$, $B^s_{p,q}$ becomes the Zygumd space denoted by $\Cs^s$.\\
3. (H\"older spaces) For~$k\in\xN$, we denote by $W^{k,\infty}({\mathbf{R}}^d)$ the usual Sobolev spaces.
For $\rho= k + \sigma$ with $k\in \xN$ and $\sigma \in (0,1)$, $W^{\rho,\infty}({\mathbf{R}}^d)$
denotes the space of all function $u\in W^{k, \infty}(\xR^d)$ such that all the $k^{th}$ derivatives of $u$ are  $\sigma$-H\"older continuous on $\xR^d$.
\end{defi}
Let us review notations and results about Bony's paradifferential calculus (see \cite{Bony,Hormander,MePise}). Here we follow the presentation of M\'etivier in \cite{MePise} (se also \cite{ABZ3}, \cite{ABZ4}).
\begin{defi}
1. (Symbols) Given~$\rho\in [0, \infty)$ and~$m\in\xR$,~$\Gamma_{\rho}^{m}({\mathbf{R}}^d)$ denotes the space of
locally bounded functions~$a(x,\xi)$
on~${\mathbf{R}}^d\times({\mathbf{R}}^d\setminus 0)$,
which are~$C^\infty$ with respect to~$\xi$ for~$\xi\neq 0$ and
such that, for all~$\alpha\in\xN^d$ and all~$\xi\neq 0$, the function
$x\mapsto \partial_\xi^\alpha a(x,\xi)$ belongs to~$W^{\rho,\infty}({\mathbf{R}}^d)$ and there exists a constant
$C_\alpha$ such that,
\begin{equation*}%\label{para:10}
\forall\la \xi\ra\ge \mez,\quad 
\lA \partial_\xi^\alpha a(\cdot,\xi)\rA_{W^{\rho,\infty}(\xR^d)}\le C_\alpha
(1+\la\xi\ra)^{m-\la\alpha\ra}.
\end{equation*}
Let $a\in \Gamma_{\rho}^{m}({\mathbf{R}}^d)$, we define the semi-norm
\begin{equation}\label{defi:norms}
M_{\rho}^{m}(a)= 
\sup_{\la\alpha\ra\le d/2+1+\rho ~}\sup_{\la\xi\ra \ge 1/2~}
\lA (1+\la\xi\ra)^{\la\alpha\ra-m}\partial_\xi^\alpha a(\cdot,\xi)\rA_{W^{\rho,\infty}({\mathbf{R}}^d)}.
\end{equation}
2. (Paradifferential operators) Given a symbol~$a$, we define
the paradifferential operator~$T_a$ by
\begin{equation}\label{eq.para}
\widehat{T_a u}(\xi)=(2\pi)^{-d}\int \chi(\xi-\eta,\eta)\widehat{a}(\xi-\eta,\eta)\psi(\eta)\widehat{u}(\eta)
\, d\eta,
\end{equation}
where
$\widehat{a}(\theta,\xi)=\int e^{-ix\cdot\theta}a(x,\xi)\, dx$
is the Fourier transform of~$a$ with respect to the first variable; 
$\chi$ and~$\psi$ are two fixed~$C^\infty$ functions such that:
\begin{equation}\label{cond.psi}
\psi(\eta)=0\quad \text{for } \la\eta\ra\le \frac{1}{5},\qquad
\psi(\eta)=1\quad \text{for }\la\eta\ra\geq \frac{1}{4},
\end{equation}
and~$\chi(\theta,\eta)$  is defined by
$
\chi(\theta,\eta)=\sum_{k=0}^{+\infty} \kappa_{k-3}(\theta) \varphi_k(\eta).
$
\end{defi}
\begin{rema}\label{TwT} We make the following remarks on the preceding definition.\\ 
1. The cut-off $\chi$ satisfies the following localization property (see \cite{MePise}, page $73$) for some $0<\eps_1<\eps_2<1$
\[
\left\{ 
\begin{aligned}
&\chi(\theta, \eta)=1\quad\text{for}~|\theta|\le \eps_1(1+|\eta|)\\
&\chi(\theta, \eta)=0\quad\text{for}~|\theta|\ge \eps_2(1+|\eta|).
\end{aligned}
\right.
\]
Therefore, in the definition of $T_au$, on the Fourier side, $T_au$ keeps only the regime where $u$ has higher frequency then $a$. In particular, when $a=1$, we have $T_1u=\psi(D_x)u$, hence 
\[
(T_1-1):~H^{-\infty}\to H^\infty,\quad \Cs^{-\infty}\to \Cs^\infty.
\]
2. As usual, the paraproduct $\widetilde T_au$ is defined by 
\[
\widetilde T_au=\sum_{k=0}^{+\infty} S_{k-3}a \Delta_ku.
\]
On the Fourier side, $\widetilde T_au$ is thus given by the formula \eqref{eq.para} with $\psi\equiv 1$. Consequently
\[
(T_a-\widetilde T_a)u=\widetilde T_a((1-\psi(D_x))u)
\]
and thus using the fact that for any $m>0$ (see Theorem $2.82$, \cite{BCD}), 
\[
\Vert \widetilde T_a v\Vert_{H^s}\le C\Vert a\Vert_{\Cs^{-m}}\Vert v\Vert_{H^{s+m}}
\]
we obtain
\[
(T_a-\widetilde T_a): ~H^{-\infty}\to H^\infty,\quad \Cs^{-\infty}\to \Cs^\infty
\]
provided $a\in \Cs^{-\infty}$. For this reason, we do not distinguish $T_au$ and $\widetilde T_au$ in this paper.
\end{rema}
\begin{defi}\label{defi:order}
Let~$m\in\xR$.
An operator~$T$ is said to be of  order~$\leo m$ if, for all~$\mu\in\xR$,
it is bounded from~$H^{\mu}$ to~$H^{\mu-m}$. 
\end{defi}
Symbolic calculus for paradifferential operators is summarized in the following theorem.
\begin{theo}\label{theo:sc}(Symbolic calculus)
Let~$m\in\xR$ and~$\rho\in [0, \infty)$. \\
$(i)$ If~$a \in \Gamma^m_0({\mathbf{R}}^d)$, then~$T_a$ is of order~$\leo m$. 
Moreover, for all~$\mu\in\xR$ there exists a constant~$K$ such that
\begin{equation}\label{esti:quant1}\lA T_a \rA_{H^{\mu}\rightarrow H^{\mu-m}}\le K M_{0}^{m}(a).
\end{equation}
$(ii)$ If~$a\in \Gamma^{m}_{\rho}({\mathbf{R}}^d), b\in \Gamma^{m'}_{\rho}({\mathbf{R}}^d)$ with $\rho>0$. Then 
$T_a T_b -T_{a \sharp b}$ is of order~$\leo m+m'-\rho$ where
\[
a\sharp b:=\sum_{|\alpha|<\rho}\frac{(-i)^{\alpha}}{\alpha !}\partial_{\xi}^{\alpha}a(x, \xi)\partial_x^{\alpha}b(x, \xi).
\] 
Moreover, for all~$\mu\in\xR$ there exists a constant~$K$ such that
\begin{equation}\label{esti:quant2}
\lA T_a T_b  - T_{a \sharp b}   \rA_{H^{\mu}\rightarrow H^{\mu-m-m'+\rho}}%&
\le 
K M_{\rho}^{m}(a)M_{0}^{m'}(b)+K M_{0}^{m}(a)M_{\rho}^{m'}(b).
\end{equation}
$(iii)$ Let~$a\in \Gamma^{m}_{\rho}({\mathbf{R}}^d)$ with $\rho >0$. Denote by 
$(T_a)^*$ the adjoint operator of~$T_a$ and by~$\overline{a}$ the complex conjugate of~$a$. Then 
$(T_a)^* -T_{a^*}$ is of order~$\leo m-\rho$ where
\[
a^*=\sum_{|\alpha|<\rho}\frac{1}{i^{|\alpha|}\alpha!}\partial_{\xi}^{\alpha}\partial_x^{\alpha}\overline{a}.
\]
Moreover, for all~$\mu$ there exists a constant~$K$ such that
\begin{equation}\label{esti:quant3}
\lA (T_a)^*   - T_{a^*}   \rA_{H^{\mu}\rightarrow H^{\mu-m+\rho}}\le 
K M_{\rho}^{m}(a).
\end{equation}
\end{theo}
We also need the following definition for symbols with negative regularity.
\begin{defi}\label{defi:symbol<0}
For~$m\in \xR$ and~$\rho\in (-\infty, 0)$,~$\Gamma^m_\rho(\xR^d)$ denotes the space of
distributions~$a(x,\xi)$
on~${\mathbf{R}}^d\times(\xR^d\setminus 0)$,
which are~$C^\infty$ with respect to~$\xi$ and 
such that, for all~$\alpha\in\xN^d$ and all~$\xi\neq 0$, the function
$x\mapsto \partial_\xi^\alpha a(x,\xi)$ belongs to $C^\rho_*({\mathbf{R}}^d)$ and there exists a constant
$C_\alpha$ such that,
\begin{equation}
\forall\la \xi\ra\ge \mez,\quad \lA \partial_\xi^\alpha a(\cdot,\xi)\rA_{C^\rho_*}\le C_\alpha
(1+\la\xi\ra)^{m-\la\alpha\ra}.
\end{equation}
For~$a\in \Gamma^m_\rho$, we define 
\begin{equation}
M_{\rho}^{m}(a)= 
\sup_{\la\alpha\ra\le 2(d +2)+\vert \rho \vert   ~}\sup_{\la\xi\ra \ge 1/2~}
\lA (1+\la\xi\ra)^{\la\alpha\ra-m}\partial_\xi^\alpha a(\cdot,\xi)\rA_{C^{\rho}_*({\mathbf{R}}^d)}.
\end{equation}
\end{defi}
\begin{prop}[see \protect{\cite[Proposition~2.12]{ABZ3}}]\label{regu<0}
Let~$\rho<0$,~$m\in \xR$ and~$a\in \dot{ \Gamma}^m_\rho$. Then the operator~$T_a$ is of order~$m-\rho$:
\begin{equation}\label{niSbis}
\| T_a \|_{H^s \rightarrow H^{s-(m- \rho)}}\leq C M_{\rho}^{m}(a),\qquad
\| T_a \|_{C^s_* \rightarrow C^{s-(m- \rho)}_*}\leq C M_{\rho}^{m}(a).
\end{equation}
\end{prop}
\begin{rema}\label{rema:low}
In the definition \eqref{eq.para} of paradifferential operators, the cut-off $\psi$ removes the low frequency part of $u$. Therefore, estimates pertaining to $T_au$ can be relaxed, for example, when $a\in \Gamma^m_0$ and $u\in \mathcal{S}'$ such that $\nabla u\in H^{\sigma+m-1}$ we have
\[
\Vert T_a u\Vert_{H^\sigma}\le CM_0^m(a)\Vert \nabla u\Vert_{H^{\sigma+m-1}}. 
\]
\end{rema}
\begin{nota} \label{symbolnorm:z}
Let $I\subset \xR$ and $a(z, x, \xi):I\times \xR^d\times \xR^d\to \xC$ be a family of symbols parametrized by $z\in I$. We denote
\[
\cM(a)=\sup_{z\in I}M^m(\rho)(a(z, \cdot, \cdot)).
\]
The set of such $a$ with $\cM(a)<\infty$ is denoted by $\Gamma^m_\rho(\xR^d\times I)$.
\end{nota}
\subsection{Paraproducts}
Given two functions~$a,b$ defined on~${\mathbf{R}}^d$ we define the remainder 
\bq\label{Bony:dep}
R(a,u)=au-T_a u-T_u a.
\eq
We shall use frequently various estimates about paraproducts (see chapter 2 in~\cite{BCD},~\cite{BaCh} and \cite{ABZ3}) which are recalled here.
\begin{theo}\label{pproduct}
\begin{enumerate}
\item  Let~$\alpha,\beta\in \xR$. If~$\alpha+\beta>0$ then
\begin{align}
&\lA R(a,u) \rA _{H^{\alpha + \beta-\frac{d}{2}}({\mathbf{R}}^d)}
\leq K \lA a \rA _{H^{\alpha}({\mathbf{R}}^d)}\lA u\rA _{H^{\beta}({\mathbf{R}}^d)},\label{Bony1} \\ 
&\lA R(a,u) \rA _{H^{\alpha + \beta}({\mathbf{R}}^d)} \leq K \lA a \rA _{C^{\alpha}_*({\mathbf{R}}^d)}\lA u\rA _{H^{\beta}({\mathbf{R}}^d)}\label{Bony2},\\
&\lA R(a,u) \rA _{C_*^{\alpha + \beta}({\mathbf{R}}^d)}\leq K \lA a \rA _{C^{\alpha}_*({\mathbf{R}}^d)}\lA u\rA _{C_*^{\beta}({\mathbf{R}}^d)}.\label{Bony3}
\end{align}
\item Let~$s_0,s_1,s_2$ be such that 
$s_0\le s_2$ and~$s_0 < s_1 +s_2 -\frac{d}{2}$, 
then
\begin{equation}\label{boundpara}
\lA T_a u\rA_{H^{s_0}}\le K \lA a\rA_{H^{s_1}}\lA u\rA_{H^{s_2}}.
\end{equation}
If in addition to the conditions above, $s_1+s_2>0$ then 
\bq\label{boundpara2}
\lA au-T_ua\rA_{H^{s_0}}\le K \lA a\rA_{H^{s_1}}\lA u\rA_{H^{s_2}}.
\eq
\item  Let~$m>0$ and~$s\in \xR$. Then
\begin{align}
&\lA T_a u\rA_{H^{s-m}}\le K \lA a\rA_{C^{-m}_*}\lA u\rA_{H^{s}}\label{pest1},\\ 
&\lA T_a u\rA_{C_*^{s-m}}\le K \lA a\rA_{C^{-m}_*}\lA u\rA_{C_*^{s}}\label{pest2},\\
&\lA T_a u\rA_{C_*^s}\le K \lA a\rA_{L^\infty}\lA u\rA_{C_*^{s}}\label{pest3}.
\end{align}
\end{enumerate}
\end{theo}
\begin{prop}\label{tame}
\begin{enumerate}
\item  If $s_0, s_1, s_2\in \xR$ satisfying $s_1+s_2> 0$, $s_0\le s_1$, $s_0\le s_2$ and $s_0< s_1+s_2-\frac{d}{2}$, then 
\begin{equation}\label{pr}
\lA u_1 u_2 \rA_{H^{s_0}}\le K \lA u_1\rA_{H^{s_1}}\lA u_2\rA_{H^{s_2}}.
\end{equation}
\item If $s\ge 0$ then 
\bq\label{tame:S}
\lA u_1u_2\rA_{H^s}\le K(\lA u_1\rA_{H^s}\lA u_2\rA_{L^{\infty}}+\lA u_2\rA_{H^s}\lA u_1\rA_{L^{\infty}}).
\eq
\item If $s\ge 0$ then 
\bq\label{tame:H}
\lA u_1u_2\rA_{C_*^s}\le K(\lA u_1\rA_{C_*^s}\lA u_2\rA_{L^{\infty}}+\lA u_2\rA_{C_*^s}\lA u_1\rA_{L^{\infty}}).
\eq
\item Let $\beta>\alpha>0$. Then
\bq\label{tame:H<0}
\lA u_1u_2\rA_{C_*^{-\alpha}}\le K\lA u_1\rA_{C_*^{\beta}}\lA u_2\rA_{C_*^{-\alpha}}.
\eq
\item Let~$s\ge 0$ and $F\in C^\infty(\xC^N)$ satisfying~$F(0)=0$. 
Then there exists a nondecreasing function~$\mathcal{F}\colon\xR_+\rightarrow\xR_+$ 
such that, for all~$U\in H^s({\mathbf{R}}^d)^N\cap L^\infty(\xR^d)^N$,
\begin{equation}\label{F(u):H}
\lA F(U)\rA_{H^s}\le \mathcal{F}\bigl(\lA U\rA_{L^\infty}\bigr)\lA U\rA_{H^s}.
\end{equation}
\item Let $F\in C^\infty(\xC^N)$ satisfying~$F(0)=0$, $s>0$,  and $p, r\in [1, \infty]$. Then there exists a nondecreasing function~$\mathcal{F}\colon\xR_+\rightarrow\xR_+$ 
such that for all $u\in B^s_{p,r}(\xR^d)^N\cap L^\infty(\xR^d)^N$,
\bq\label{F(u):C}
\Vert F\circ u\Vert_{B^s_{p,r}}\le \cF(\Vert u\Vert_{L^\infty})\Vert u\Vert_{B^s_{p,r}}.
\eq
% Let~$s>0 $ and consider~$F\in C^\infty(\xC^N)$ such that~$F(0)=0$. 
%Then there exists a nondecreasing function~$\mathcal{F}\colon\xR_+\rightarrow\xR_+$ 
%such that, for any~$U\in C_*^s({\mathbf{R}}^d)^N$,
%\begin{equation}\label{F(u):C}
%\lA F(U)\rA_{C_*^s}\le \mathcal{F}\bigl(\lA U\rA_{L^\infty}\bigr)\lA U\rA_{C_*^s}.
%\end{equation}
\end{enumerate}
\end{prop}
\begin{theo}[see \protect{\cite[Theorem~2.92]{BCD}}]\label{paralin}(Paralinearization)
Let $r,~\rho$ be positive real numbers and $F$ be a $C^{\infty}$ function on $\xR$ such that $F(0)=0$. Assume that $\rho$ is not an integer. For any $u\in H^{\mu}(\xR^d)\cap C_*^{\rho}(\xR^d)$ we have
\[
\lA F(u)-T_{F'(u)}u\rA_{ H^{\mu+\rho}(\xR^d)}\le C(\lA u\rA_{L^{\infty}(\xR^d)})\lA u\rA_{C_*^{\rho}(\xR^d)}\lA u\rA_{H^{\mu}(\xR^d)}.
\]
\end{theo}
\begin{rema}
In Theorem  $2.92$, \cite{BCD}, there is a restriction that $\rho$ is not an integer. In fact, by following the proof of  the same result (but qualitative) in Theorem $5.2.4$, \cite{MePise} one can check that this restriction can be dropped.
\end{rema}
\begin{lemm}\label{pest:n}
 Let  $s, r, \alpha\in \xR$ satisfy 
\[
\text{either}\quad r\le 0,~s<\alpha+r\quad\text{or}\quad r>0,~s<\alpha.
\]
 Then there exists $C>0$ such that 
\[
\Vert T_au\Vert_{H^s}\le C\Vert a\Vert_{L^r}\Vert u\Vert_{\Cs^\alpha}.
\]
\end{lemm}
\begin{proof}
We have by the definition of paraproducts  (see  Definition \ref{defi:PL} and Remark \ref{TwT}),
\begin{align*}
\Vert T_au\Vert_{H^s}^2&\les\sum_{k\ge 0}2^{2sk}\Vert S_{k-3}a\Delta_ku\Vert_{L^2}^2\les\sum_{k\ge 0}2^{2sk}\Vert S_{k-3}a\Vert^2_{L^2}\Vert\Delta_ku\Vert^2_{L^\infty}.
%&\les\Vert u\Vert^2_{\Cs^\alpha}\Vert a\Vert^2_{L^2}\sum_{k\ge 0}2^{2k(s-\alpha)}\les\Vert u\Vert^2_{\Cs^\alpha}\Vert a\Vert^2_{L^2}.
\end{align*}
For small $k$, we have the easy estimate 
\[
\sum_{k=0}^3 2^{2sk}\Vert S_{k-3}a\Vert^2_{L^2}\Vert\Delta_ku\Vert^2_{L^\infty}\les \Vert a\Vert^2_{L^2}\Vert u\Vert^2_{\Cs^\alpha}.
\]
Consider the case $r\le 0$ and $s<\alpha+r$.  Pick $\eps\in (0, \alpha+r-s)$. For $k\ge 4$, using $S_{k-3}a= \sum_{j=0}^{k-3}\Delta_j a$ we can apply the H\"older inequality to estimate (notice that $r\le 0$)
\begin{align*}
\sum_{k\ge 4}2^{2s k}\Vert S_{k-3}a\Vert^2_{L^2}\Vert\Delta_ku\Vert^2_{L^\infty}& \les \Vert u\Vert_{\Cs^\alpha}^2\sum_{k\ge 4}2^{2(s-\alpha)k}\Big(\sum_{j=0}^{k-3}\Vert \Delta_ja\Vert_{L^2}\Big)^2\\
&\les \Vert u\Vert_{\Cs^\alpha}^2\sum_{k\ge 4}2^{2(s-\alpha)k}\sum_{j=0}^{k-3}2^{2rj}\Vert \Delta_j a\Vert^2_{L^2}\sum_{j=0}^{k-3}2^{-2rj}\\
&\les \Vert u\Vert_{\Cs^\alpha}^2\sum_{k\ge 4}2^{2(s-\alpha-r+\eps)k}\sum_{j=0}^{k-3}2^{2rj}\Vert \Delta_j a\Vert^2_{L^2}\\
&\les \Vert u\Vert_{\Cs^{\alpha}}^2\Vert a\Vert^2_{H^r}.
\end{align*}
Now, if $r>0$ and $s<\alpha$, in the second line of the preceding estimate we observe that the series
\[
\sum_{j\ge 0}2^{-2rj},\quad \sum_{k\ge 4}2^{2(s-\alpha)k} 
\]
converge. This concludes the proof in the second case.
\end{proof}
\subsection{Paradifferential calculus in Besov spaces}
Concerning the symbolic calculus in Besov spaces, we have the following results.
\begin{lemm}[see \protect{\cite[Lemma 2.6]{WaZh}}]\label{sc:B}
Let $s, m, m'\in \xR$, $q\in [1, \infty]$ and $\rho\in [0, 1]$.
\begin{itemize}
\item[(i)] If $a\in \Gamma^m_0(\xR^d)$ then 
\[
\Vert T_a\Vert_{B^s_{\infty, q}\to B^{s-m}_{q, \infty}}\le C M^m_0(a).
\]
\item[(ii)] If $a\in \Gamma^m_0(\xR^d)$, $b\in \Gamma^{m'}_0(\xR^d)$ then 
\[
\Vert T_aT_b-T_{ab}\Vert_{B^s_{\infty, q}\to B^{s-(m+m')+\rho}_{q, \infty}}\le C M^m_\rho(a)M^{m'}_0(b)+CM^m_0(a)M^{m'}_\rho(b).
\]
\end{itemize}
\end{lemm}
\begin{lemm}[see \protect{\cite[Lemma 2.10]{Wa2}}]\label{pr:B}
1. Let $s\in \xR$ and $p, q\in [1, \infty]$. Then for any $\sigma>0$ we have
\bq
\Vert T_au\Vert_{B^s_{p,q}}\le K\min \big(\Vert a\Vert_{L^\infty}\Vert u\Vert_{B^{s}_{p,q}}, \Vert a\Vert_{\Cs^{-\sigma}}\Vert u\Vert_{B^{s+\sigma}_{p,q}} \big) \label{pest:B}.
\eq
2. Let $s>0$ and $p, q\in [1, \infty]$. The for any $\sigma\in \xR$, we have
\bq
\Vert R(a, u)\Vert_{B^s_{p,q}}\le K\Vert a\Vert_{\Cs^{\sigma}}\Vert u\Vert_{B^{s-\sigma}_{p,q}}  \label{Bony:B}.
\eq
\end{lemm}
To deal with time-dependent distributions, we use the Chemin-Lerner spaces defined as follows (see Chapter 2, \cite{BCD}).
\begin{defi}\label{Chemin-Lerner}
For $T>0$, $s\in \xR$, and $p, q, r\in [1, \infty]$, we set
\bq
\Vert u\Vert_{\wL^q([0, T]; B^s_{p, r})}\defn \Vert \big(2^{js}\lA \Delta_j u\rA_{L^q([0, T]; L^p(\xR^d))}\big)_{j=0}^\infty\Vert_{\ell^r}.
\eq
Again, when $p=r=\infty$, we denote $\wL^q([0, T]; B^s_{p, r})=\wL^q([0, T]; \Cs^s)$. Notice that $\Vert u\Vert_{\wL^\infty(I; \Cs^s)}=\Vert u\Vert_{L^\infty(I; \Cs^s)}$.
\end{defi}
The next lemma then follows easily from the proof Lemma \ref{sc:B}.
\begin{lemm}[see \protect{\cite[Lemma 2.6]{WaZh}}]\label{sc:CL}
Let $s, m, m'\in \xR$, $p, q\in [1, \infty]$, $\rho\in [0, 1]$ and $I=[0, T]$.
\begin{itemize}
\item[(i)] If $a\in \Gamma^m_0(I\times\xR^d)$ then 
\[
\Vert T_a\Vert_{\wL^p(I; B^s_{\infty, q})\to \wL^p(I; B^{s-m}_{\infty, q})}\le C \cM^m_0(a).
\]
\item[(ii)] If $a\in \Gamma^m_0(I\times \xR^d)$, $b\in \Gamma^{m'}_0(I\times \xR^d)$ then 
\[
\Vert T_aT_b-T_{ab}\Vert_{\wL^p(I; B^s_{\infty, q})\to \wL^p(I; B^{s-(m+m')+\rho}_{\infty, q})}\le C \cM^m_\rho(a)\cM^{m'}_0(b)+C\cM^m_0(a)\cM^{m'}_\rho(b).
\]
\end{itemize}
\end{lemm}
Finally, the following lemma is a direct consequence of Lemma \ref{pest:B}.
\begin{lemm}[see \protect{\cite[Lemmas 2.17, 218]{Wa2}}] 
Let $I=[0, T]$.\\
1.  Let $s\in \xR$ and $q, q_1, q_2, r\in [1, \infty]$ with $\frac 1q=\frac{1}{q_1}+\frac {1}{q_2}$. Then for any $\sigma>0$ we have
\bq
\Vert T_au\Vert_{\wL^q(I; B^s_{\infty,r})}\le K\min \big(\Vert a\Vert_{\wL^{q_1}(I; L^\infty)}\Vert u\Vert_{\wL^{q_2}(I; B^{s}_{\infty,r})},\Vert a\Vert_{\wL^{q_1}(I; \Cs^{-\sigma})}\Vert u\Vert_{\wL^{q_2}(I; B^{s+\sigma}_{\infty,r})} \big) \label{pest:CL}.
\eq
2. Let $s>0$ and $q, q_1, q_2, r\in [1, \infty]$ with $\frac 1q=\frac{1}{q_1}+\frac {1}{q_2}$. Then for any $\sigma\in \xR$ we have\bq
\Vert R(a, u)\Vert_{\wL^q(I; B^s_{\infty,r})}\le K\Vert a\Vert_{\wL^{q_1}(I; \Cs^{-\sigma})}\Vert u\Vert_{\wL^{q_2}(I; B^{s+\sigma}_{\infty,r})}. \label{Bony:CL}
\eq
\end{lemm}
\subsection{Parabolic regularity}
Define the following interpolation spaces
\bq
\begin{aligned}
	X^\mu(J)&=C^0_z(I;H^\mu(\R^d))\cap L^2_z(J;H^{\mu+\mez}(\R^d)),\\
	Y^\mu(J)&=L^1_z(I;H^\mu(\R^d))+L^2_z(J;H^{\mu-\mez}(\R^d)).
\end{aligned}
\eq
\begin{theo}[see \protect{\cite[Proposition~2.18]{ABZ3}}]\label{parabolic:Sob}
Let $\rho\in (0, 1),~J=[z_0, z_1]\subset \xR,~p\in \Gamma^1_{\rho}(\xR^d\times J),~q\in \Gamma^0_0(\xR^d\times J)$ with the assumption that 
\[
\RE p(z; x, \xi)\ge c|\xi|,
\]
for some constant $c>0$. Assume that $w$ solves
\[
\partial_zw+T_pw=T_qw+f,\quad w\arrowvert_{z=z_0}=w_0.
\]
Then for any $r\in\R$, if $f\in Y^r(J)$ and $w_0\in H^r$, we have~$w\in X^r(J)$ and

\[
\lA w\rA_{X^r(J)}\le K\left\{ \lA w_0\rA_{ H^r}+\lA f\rA_{ Y^r(J)}\right\}.
\]
for some constant $K=K(\mathcal{M}^1_{\rho}(\rho), \mathcal{M}^0_0(q), c^{-1})$ nondecreasing in each argument.
\end{theo}
\begin{theo}[see \protect{\cite[Proposition~3.1]{WaZh}}]\label{parabolic:Hol}
Let $r\in \xR, \ell\in [1, \infty]$ and $1\le q\le p\le \infty$. Let $\rho\in (0, 1),~J=[z_0, z_1]\subset \xR,~a\in \Gamma^1_{\rho}(\xR^d\times J)$ with the assumption that 
\[
\RE a(z; x, \xi)\ge c|\xi|,
\]
for some constant $c>0$. Assume that $w$ solves
\[
\partial_zw+T_aw=F,\quad w\arrowvert_{z=z_0}=w_0.
\]
If 
\[
w_0\in B^r_{\infty, \ell},~ F\in \wL^q(J, B^{r-1+\frac{1}{q}}_{\infty,\ell})
\]
and there exists $\delta>0$ such that $w\in \wL^p(J, \Cs^{-\delta})$, then we have $w\in \wL^p(J, B^{r+\frac{1}{p}}_{\infty, \ell})$ and 
\begin{align*}
\lA w\rA_{\wL^p(J, B^{r+\frac{1}{p}}_{\infty, \ell})}\le K\Big\{ \lA w_0\rA_{B^r_{\infty, \ell}}+\lA F\rA_{\wL^q(J, B^{r-1+\frac{1}{q}}_{\infty,\ell})}+\lA w\rA_{\wL^p(J, \Cs^{-\delta})}\Big\}.
\end{align*}
for some constant $K=K(\mathcal{M}^1_{\rho}(a), c^{-1})$ nondecreasing in each argument. When $p=\infty$, the left-hand side can be replaced by $\lA w\rA_{C(J, B^{r}_{\infty, \ell})}$.
\end{theo}
Finally, we recall a classical interpolation lemma.
 \begin{lemm}[see \protect{\cite[th~3.1]{Lions}}]\label{lemm:inter}
Let $I = (-1,0)$ and $s \in \xR$. Let   $u \in L_z^2(I, H^{s+ \mez}(\xR^d))$ such that 
$\partial_z u \in L_z^2(I, H^{s-\mez}(\xR^d)).$ Then $u \in C^0([-1,0], H^{s}(\xR^d))$ 
and there exists an absolute constant $C>0$ such that
$$
\sup_{z \in [-1,0]} \| u(z, \cdot) \|_{ H^{s}({\mathbf{R}}^d))} 
\leq C \bigl(\|u\|_{L^2(I,H^{s+ \mez}(\xR^d))}+  \|\partial_z u\|_{L^2(I,H^{s- \mez}(\xR^d))} \bigr).
$$
\end{lemm}
\section*{Acknowledgment}
\hk Quang Huy Nguyen was partially supported by the labex LMH through the grant no ANR-11-LABX-0056-LMH in the ``Programme des Investissements d'Avenir'' and by Agence Nationale de la Recherche project  ANA\'E ANR-13-BS01-0010-03. We  would like to sincerely thank T.Alazard, N.Burq and C.Zuily for many fruitful discussions, suggestions when this work was preparing, as well as their helpful comments at the final stage of the work. We thank the referee for his constructive comments, which have motivated us to improve significantly an earlier version of Theorem \ref{intro:theo:blowup}.

\end{document}